\documentclass[12pt]{amsart}
\usepackage{amsfonts,graphics,amsmath,amsthm,amsfonts,amscd,amssymb,amsmath,latexsym,multicol,euscript}
\usepackage{epsfig,url}
\usepackage{flafter}
\usepackage{hyperref}
\usepackage[all,cmtip,line]{xy}
\usepackage{array}
\usepackage[english]{babel}
\usepackage{overpic}
\usepackage{graphicx}
\usepackage{psfrag}
\usepackage[figurename=Diagram, tablename=Diagram, figurewithin=none, tablewithin=none, labelsep=space]{caption}
\usepackage{subfig}
\usepackage{multirow}
\usepackage{wrapfig}

\newtheorem{theorem}{Theorem}[section]
\newtheorem{main theorem}{Main Theorem}
\newtheorem{proposition}[theorem]{Proposition}
\newtheorem{conditional theorem}[theorem]{Conditional Theorem}
\newtheorem{conjecture}[theorem]{Conjecture}

\newtheorem{corollary}[theorem]{Corollary}

\newtheorem{conditional corollary}[theorem]{Conditional Corollary}

\newtheorem{lemma}[theorem]{Lemma}

\theoremstyle{definition}
\newtheorem{example}[theorem]{Example}
\newtheorem{convention}[theorem]{Convention}

\newtheorem{definition-lemma}[theorem]{Definition-Lemma}

\newtheorem{definition-proposition}[theorem]{Definition-Proposition}

\theoremstyle{remark}

\newcommand\R{\mathbb{R}}
\newcommand\Q{\mathbb{Q}}
\newcommand\Z{\mathbb{Z}}

\newcommand\bs{\bigskip}
\newcommand\ms{\medskip}
\newcommand\ssk{\smallskip}
\newcommand\noi{\noindent}

\newcommand{\cn}{{}_{\mathrm{lcm}}}

\newcommand{\wlc}{\sim_{\mathrm{wlc}}}
\newcommand{\lcm}{\sim_{\mathrm{lcm}}}
\newcommand{\md}{\sim_{\mathrm{mod}}}
\newcommand{\mob}{\sim_{\mathrm{mob}}}
\newcommand{\fix}{\sim_{\mathrm{fix}}}
\newcommand{\me}{\sim_{\mathrm{m+e}}}

\newcommand{\Lg}{{}^{\mathrm{log}}}
\newcommand{\ild}{\underline{a}}

\newcommand{\linsys}[1]{\left|{#1}\right|} 

\newcommand{\sK}{\mathcal K}
\newcommand{\fB}{\mathfrak B}
\newcommand{\fN}{\mathfrak N}
\newcommand{\fD}{\mathfrak D}
\newcommand{\fP}{\mathfrak P}                 
\DeclareMathOperator{\rP}{P}                     
\newcommand{\ofP}{\overline{\mathfrak P}}
\newcommand{\fC}{\mathfrak C} 
\newcommand{\fF}{\mathfrak F} 
\newcommand{\fR}{\mathfrak R} 

\newcommand{\fS}{\mathfrak S} 
\newcommand{\ofC}{\overline{\mathfrak C}}
\newcommand{\ofF}{\overline{\mathfrak F}}
\newcommand{\ofR}{\overline{\mathfrak R}}

\DeclareMathOperator{\N^1}{N^1}
\DeclareMathOperator{\Nef}{Nef}
\DeclareMathOperator{\sAmp}{sAmp}
\DeclareMathOperator{\Mob}{Mob}
\DeclareMathOperator{\Eff}{Eff}
\DeclareMathOperator{\Amp}{Amp}
\DeclareMathOperator{\NE}{NE}
\DeclareMathOperator{\iso}{\thickapprox}

\DeclareMathOperator{\nef}{\mathfrak{Nef}}
\DeclareMathOperator{\samp}{s\mathfrak{A}}
\DeclareMathOperator{\mobile}{\mathfrak{M}}
\DeclareMathOperator{\eff}{\mathfrak{E}}
\DeclareMathOperator{\amp}{\mathfrak{A}}

\newcommand{\id}[1]{\iota({#1})} 

\DeclareMathOperator{\mult}{mult}
\DeclareMathOperator{\Int}{Int}
\DeclareMathOperator{\Char}{char}
\DeclareMathOperator{\Supp}{Supp}
\DeclareMathOperator{\kd}{\kappa}
\DeclareMathOperator{\nd}{\nu}
\DeclareMathOperator{\pr}{pr}
\DeclareMathOperator{\Center}{center}
\DeclareMathOperator{\Bl}{Bl}
\DeclareMathOperator{\pt}{pt}
\DeclareMathOperator{\Spec}{Spec}

\begin{document}

\title{Geography of log models:\\
theory and applications.}

\author{V.V.~Shokurov}
\address{Department of Mathematics\newline
\indent  Johns Hopkins University\newline
\indent Baltimore, MD 21218, USA}
\address{Steklov Mathematical Institute\newline
\indent Russian Academy of Sciences\newline
\indent Gubkina str. 8, 119991, Moscow, Russia}
\email{shokurov@math.jhu.edu}
\author{Sung Rak Choi}
\address{Department of Mathematics \newline
\indent University of California, Riverside\newline
\indent 900 University Ave\newline
\indent Riverside, CA 92521, USA}
\email{schoi@math.ucr.edu}

\thanks{Both authors were partially supported
by NSF grant DMS-0701465.}

\begin{abstract}
An introduction to geography of log models with applications
to positive cones of FT varieties and to geometry of minimal models and
Mori fibrations.
\end{abstract}

\maketitle

\tableofcontents

\section{Introduction}

What is a relation between the LMMP and some concrete
geometrical results such as polyhedral properties of
effective and mobile cones for FT varieties, finiteness
results for modifications of FT varieties and
minimal models, factorization of birational isomorphisms of
Mori fibrations into Sarkisov links and birational rigidity?
One of the possible explanations for those relations
can be found by exploring the geography of log models.
This paper presents this explanation that was
originally introduced in \cite{3fold-logmod} \cite{isksh} and partially developed in \cite{choi}.

We always assume the LMMP and the semiampleness conjecture.
One can find a detailed information about
those conjectures and which of their variations
are needed in Section~\ref{section7-LMMP-semi}.

\section{Geography of log models}\label{section2-geography}

Take a pair $(X/Z, S)$ where $X/Z$ is a proper morphism and $S =\sum S_i$
is a reduced b-divisor of $X/Z$ with the distinct prime components $S_i$.
In this section we introduce $5$ equivalence relations on
the \emph{unit cube}:
$$
\fB_S:= \oplus_{i=1}^m[0,1] S_i \cong [0,1]^m.
$$

An element $B\in\fB_S$ will be referred to as a
{\em boundary\/} divisor even when it is actually a b-boundary.
For the concepts of the LMMP used in these relations, see
Section~\ref{section7-LMMP-semi}.

\bs

\paragraph{Model equivalence.}

Boundaries $B,B'\in\fB_S$ are {\em model equivalent\/} and
we write $B\md B'$ if
the pairs $(X/Z,B),(X/Z,B')$ have the same
set of resulting models \cite[2.9]{isksh}.
\ms

\paragraph{Weakly log canonical equivalence.}

Boundaries $B,B'\in\fB_S$ are {\em wlc model equivalent\/}  and we write $B\wlc B'$
if both pairs
$(X/Z,B),(X/Z,B')$ have the same wlc models
and the models are numerically equivalent, that is:
for any model $Y/Z$ of $X/Z$,

$(Y/Z,B\Lg_Y)$ is a wlc model of $(X/Z,B)$
if and only if $(Y/Z, B'\Lg_Y)$ is a wlc model of $(X/Z,B')$; and

the divisors $K_Y+B\Lg_Y,K_Y+B'\Lg_Y$
have the same signatures with respect to the intersection for all curves on $Y/Z$.

\ms

\paragraph{Log canonical model equivalence.}
Boundaries $B,B'\in\fB_S$ are {\em lc model equivalent\/} and we write $B\lcm B'$
if the pairs $(X/Z,B),(X/Z,B')$ have the same (rational) Iitaka
fibrations: for a given (rational) contraction $\varphi\colon X\dashrightarrow Y/Z$,
$\varphi$ is an Iitaka fibration for $(X/Y,B)$ if and only if
$\varphi$ is an Iitaka fibration for $(X/Y,B')$, or equivalently,
$(X/Z,B)$ has a (rational) Iitaka fibration $X\dashrightarrow Y$ if and only if
$(X/Z,B')$ has a (rational) Iitaka fibration $X\dashrightarrow Y'$ and
the fibrations are naturally isomorphic, that is, there is a commutative diagram

$$
\xymatrix{
& X\ar@{-->}[ld]\ar@{-->}[rd]& \\
Y\ar[rr]^{\cong}& & Y'
}
$$
with the natural isomorphism $Y \cong Y'/Z$.

\ms

\paragraph{Mobile equivalence.}
Boundaries $B,B'\in\fB_S$ are {\em mobile equivalent\/} and we write
$B\mob B'$ if the positive parts $P(B),P(B')$ of both b-divisors $\sK+B\Lg,\sK+B'\Lg$ have
the same signatures with respect to the curves on rather high models $Y/Z$.
On a rather high model we can take any model $Y/Z$ over some wlc models of $(X/Z,B),(X/Z,B')$.

\ms

\paragraph{Fixed equivalence.}
Boundaries $B,B'\in\fB_S$ are {\em fix equivalent\/} and we write $B\fix B'$ if
the fixed parts $F(B),F(B')$ of both b-divisors $\sK+B\Lg,\sK+B'\Lg$ have
the same signatures of their multiplicities, that is:
for any prime b-divisor $D$,
$$
\mult_DF(B)>0 \Leftrightarrow \mult_DF(B')>0.
$$

\bs

Note that all the relations above are defined on the whole  $\fB_S$.

\bs

The resulting models obtained by the LMMP are at least projective, and
even slt by the slt LMMP.
An abstract definition of resulting models gives
a larger class of them.
However, according to the following it is enough to
consider only the good ones:
projective $\Q$-factorial wlc models, e.g.,
slt wlc models.

\begin{proposition}\label{prop-qfact-proj}
The equivalence $\wlc$ for projective $\Q$-factorial
wlc models is the same as that
for all wlc models.
\end{proposition}

\begin{proof}
Any non-$\Q$-factorial wlc model can be obtained by
a crepant contraction of some $\Q$-factorial wlc model
that gives a wlc model by the semiampleness.
Any $\Q$-factorial nonprojective wlc model can be
obtained from a projective and $\Q$-factorial
wlc model by a (generalized) log flop.
Both constructions use
the slt LMMP and \cite{3fold-logmod}.
Thus if the pairs have the same
projective and $\Q$-factorial, numerically equivalent models,
they have the same numerically equivalent wlc models.

The converse statement means that each
projective and $\Q$-factorial
wlc model of $(X/Z,B)$ will  be also
projective and $\Q$-factorial
wlc model of
$(X/Z,B')$ when $B\wlc B'$
(cf. the proof of Lemma \ref{lemma-convex relations}).
\end{proof}

\bs

\begin{convention}\label{convention}
The model equivalences $\md,\wlc$ are defined for
projective and $\Q$-factorial resulting models
of pairs.
In fact, we can even use slt wlc models since Proposition \ref{prop-qfact-proj}
holds for slt wlc models too.
In most of constructions below, we use slt wlc resulting models.
However, slt wlc models are not stable for limits of boundaries (see Lemma \ref{lemma-closed N}).
\end{convention}
\bs

The following relations hold for model
equivalences.

\begin{proposition}\label{prop-equivalence}
We have the following implications:
$$
\fix\Leftarrow\md\Leftarrow\wlc\Rightarrow\lcm
\Leftrightarrow\mob
\text{ and }$$
$$
\wlc\Leftrightarrow\fix\cap\lcm 
$$
\end{proposition}

\begin{proof}
Immediate by definition and the semiampleness \ref{conj-semiample}.
\end{proof}
\bs

In general,
$$
\fix\not\Rightarrow\md\not\Rightarrow\wlc\not\Leftarrow\lcm.
$$
For example, $\fix\not\Rightarrow\md$ follows from the fact that
the small modifications of flopping facets preserve the $\fix$ equivalence relation
but not $\md$ (see Theorem \ref{thrm-facet types}, \ref{thrm-ridge types}).

\bs

Define the subset of $\fB_S$:
$$
\fN_S:=\{ B\in\fB_S \mid (X/Z,B) \text{ has a wlc model}\}.
$$
Equivalently, we can also use the condition: $\nu(X/Z,B)\geq 0$ (see Numerical Kodaira dimension in
Section \ref{section7-LMMP-semi}). Moreover, by the semiampleness \ref{conj-semiample},
this is equivalent to the nonnegativity of (invariant) Kodaira dimension (see Kodaira dimension in Section \ref{section7-LMMP-semi}):
$\kappa(X/Z,B)\ge 0$.

Proposition \ref{prop-equivalence} implies that $\wlc$ is the finest
of the model equivalences on $\fN_S$.
Thus,  for the pairs with wlc models,
the finiteness and polyhedral properties
of all the equivalence relations
follow from those of $\wlc$.
All the above equivalence relations, except for $\md$, are not interesting outside $\fN_S$
because they are trivial on $\fB_S\setminus\fN_S$.
It is expected that
many similar properties hold for $\md$ on $\fB_S$
in the case of Mori log fibrations \cite{isksh} and
that plays an important role for
the genuine Sarkisov program.

\bs

Recall that $\fD_S$ is the $\R$-vector space of b-divisors spanned by $S$:
$$\fD_S\colon=\oplus_{i=1}^m \R S_i \cong \R^m.$$
 The space $\fD_S$ contains the cube $\fB_S$.

A {\em convex polyhedron}
$\fP$ in $\fD_S$ is a set defined by the intersection of
finitely many (open or closed) half spaces and hyperplanes in $\fD_S$.
A {\em polyhedron\/} is a finite union of convex ones.
A convex polyhedron $\fP$ is said to be
{\em open} in $\fD_S$ if the half spaces in the intersection are open.
The closure $\ofP$, the interior $\Int\fP$,
and the boundary $\partial\fP$
of a convex polyhedron $\fP$ will be taken
in its linear span, the minimal intersection of hyperplanes containing $\fP$.
A {\em face} $\fF$ of a convex polyhedron $\fP$ is the intersection
$\fF=\ofP\cap H$ (possibly $\emptyset$ or $\ofP$ itself)
for some hyperplane $H=\{f=0\}$ such that
$\ofP\subsetneq\{f\ge 0\}$.
A {\em facet\/} of $\fP$ is a maximal face, that is,
a face of dimension $\dim \fP-1$. A \emph{ridge} is a facet of a facet.

For two distinct vectors $A,B\in\fD_S$, $\overrightarrow{AB}$ denotes the open ray starting from $A$
through $B$ and $\overline{AB}$ denotes the line through $A$ and $B$.
$[A,B]$ ($[A,B), (A,B)$, etc) is defined as the closed (resp. half closed, open, etc)  interval between $A$ and $B$.

Let $\fP$ be a convex polyhedron in $\fB_S$.
A facet (face) $\fF$ of $\fP$ is {\em given by a facet $\fF_i$ of\/} $\fB_S$ if
$\fF=\ofP\cap\fF_i$.
The facet (resp. face) $\fF_i$ is given by equation $b_i=0$ or $1$.
Thus $\fF$ can be given in $\ofP$ by one of those equations.
Note the following property of the linear function $b_i$ on $\fD_S$:
if $b_i=1$ (respectively $0$) in $B\in \Int\fF$ and $<1$
 (respectively $>0$) in some point of $\fP$,
then $b_i=1$ (respectively $0$) on $\fF$ and $<1$
 (respectively $>0$) in $\fP\setminus\fF$.

\bs

By a {\em geography\/} on $\fN_S$, we mean a finite rational polyhedral decomposition of it, and the subsets
$\fP\subseteq \fN_S$ of such a decomposition will be referred to as
{\em classes\/} or just {\em polyhedrons\/}.
We call a class $\fC$ of maximal
dimension, $\dim_\R\fN_S$, a {\em country},
a class $\fF$ of codimension 1 a {\em facet}, and
a class $\fR$ of codimension 2 a {\em ridge}.

According to the following, each of the above equivalence relations defines a geography on $\fN_S$,
but the notations $\fN_S$
 will be reserved only for the {\em wlc} geography given by $\wlc$.
Thus we consider $\fN_S$  as a set with the structure.

Let $\fF$ be a minimal face of $\fB_S$ which
contains a polyhedron $\fP$.
We say that a polyhedron $\fP$ is open in $\fB_S$ if it is open in $\fF$, that is,
$\fP=\fP'\cap \fF$, where $\fP'$ is an open polyhedron in $\fD_S$.

\begin{theorem}\label{mainthrm-geography}
The set $\fN_S$ is closed convex rational polyhedral
and it is decomposed into a finite number of $\wlc$ classes $\fP$.
Each class $\fP$ of the wlc geography is
a convex rational polyhedron which is open in $\fB_S$.
For two classes $\fP,\fP'$, $\Int\fP'\cap\ofP\not=\emptyset$ holds
if and only if $\overline{\mathfrak P'}$ is a face of $\ofP$.
\end{theorem}

\begin{example}
Let $X$ be a projective surface
with a contraction $f:X\rightarrow Z$ of
a nonsingular rational curve $C=S,C^2=-n\le -3$.
Consider the pairs $(X/Z,B)$ where $B=aC\in\fB_S=[0,1]C$.
Then by the adjunction formula, we have the equality $K\equiv f^*K_Z+\frac{2-n}{n}C$.
If $0\le a \le \frac{n-2}{n}$,  the pair $(X/Z,aC)$ is a wlc model of itself.
If $\frac{n-2}{n} \le a\le 1$,  $(Z/Z,0)$ is a wlc model of $(X/Z,aC)$.
In total, we have the following geography:

\begin{figure}[h]
\begin{overpic}[scale=0.4]{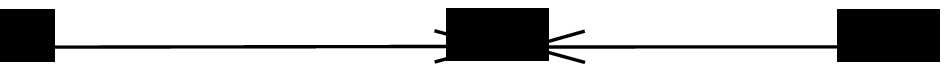}
\put(-1,-10){\small $0$}
\put(42,-10){\small $\frac{n-2}{n}=B_1$}
\put(96,-10){\small $1=B_2$}
\end{overpic}
\setlength{\abovecaptionskip}{20pt}
\caption{}\label{dia ex1}
\end{figure}
\bs
\bs

In this example, we have the following:
\begin{enumerate}
\item $\fB_S=\fN_S$;
\item there are three $\wlc$ classes in $\fN_S$: $\fP_1=[0,\frac{n-2}{n})C,$ $\fP_2=\frac{n-2}{n}C,$ $\fP_3=(\frac{n-2}{n},1]C$;
\item there are two $\lcm$ classes: $\fP'_1=[0,\frac{n-2}{n})C, \fP'_2=[\frac{n-2}{n},1]C$;
\item  $B_1\lcm B_2$, but $\not\wlc$: $\lcm\not\Rightarrow\wlc$.
\end{enumerate}
\end{example}
\ms

\begin{example}\label{example-cremona}
Suppose that $X_1=\mathbb{P}^2$ and $X_2=\mathbb{P}^2$ are
related by a standard quadratic transformation.
Choose a general nonsingular prime b-divisor $D_i\in\linsys{4H_i}$
where $H_i$ is a straight line on $X_i$ for $i=1,2$.
Consider the geography in the unit cube $\fB_S$ where $S=D_1+D_2$.

\begin{figure}[h]
\begin{overpic}[scale=0.4]{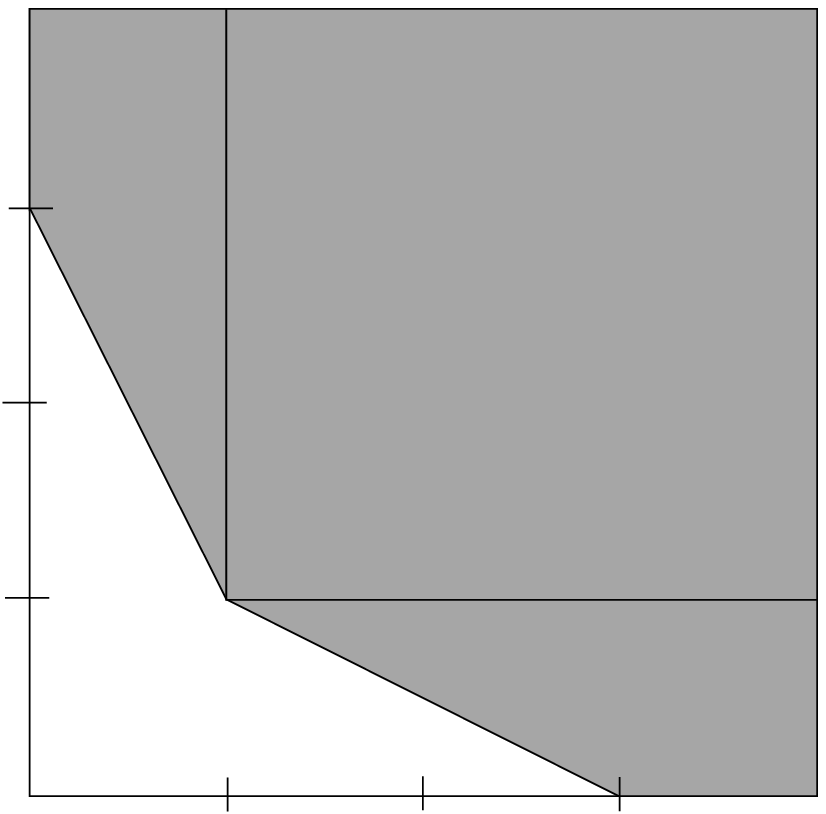}
\put(10,75){$\fC_1$}
\put(50,50){$\fC_2$}
\put(75,10){$\fC_3$}
\put(-23,71){\small $\frac{3}{4}D_2$}
\put(-12,95){\small $D_2$}
\put(65,-12){\small $\frac{3}{4}D_1$}
\put(95,-10){\small $D_1$}
\end{overpic}
\setlength{\abovecaptionskip}{20pt}
\caption{}
\end{figure}

\bs

In particular, $X_1=\mathbb P^2\iso X_2=\mathbb P^2$ are resulting models for $\fC_1$ and $\fC_3$.
Two $\mathbb{P}^2$'s
are isomorphic under an unnatural isomorphism (the quadratic transform):
$\fC_1$ and $\fC_3$
are different classes for $\md$.
However, $\frac{3}{4} D_1\lcm \frac{3}{4}D_2$ and
$\lcm$ is not convex.

\end{example}

\begin{lemma}\label{lemma-convex relations}
$\wlc$ is convex in $\fN_S$.
\end{lemma}

\begin{proof}
(Cf. the proof of \cite[lemma~2]{orderedterm}.)
If $B\wlc B'$ and $(Y/Z,B\Lg_Y)$ is a projective  $\Q$-factorial wlc model
of $(X/Z,B)$, then $(Y/Z,B'\Lg_Y)$ is
a projective $\Q$-factorial wlc model of $(X/Z,B')$.
Moreover, for any $D\in [B,B']$, $(Y/Z,D\Lg_Y)$ is a projective  $\Q$-factorial wlc model
of $(X/Z,D)$.
The models are numerically equivalent to the edge models
$(Y/Z,B\Lg_Y),$ $(Y/Z,B'\Lg_Y)$.
Thus all $(Y/Z,D\Lg_Y)$ have the same (generalized) log flops and wlc models.
Note that $P(D),F(D)$ are linear with respect to $D$ and by Proposition \ref{prop-equivalence},
$F(D)$ has the same
signatures for all multiplicities with respect to the prime b-divisors.
Thus the log flops of $(Y/Z,D\Lg_Y)$ corresponding to
wlc models are the same for all $D\in [B,B']$
(cf. Corollary \ref{cor-property-p-e} (1)).
\end{proof}

We say that a \emph{set} $\fP$ in a finite dimensional $\R$-vector space $V$ is {\em locally conical}
if for any point $p$ in $V$, there exists an open neighborhood $U$ containing $p$ such that
the intersection $U\cap \fP$ is conical. In other words, for any $p'\in U$, either $(p,p')\subseteq \fP$
or $(p,p')\cap \fP=\emptyset$.
A \emph{decomposition} of a set into disjoint subsets
$\fP$ is said to be \emph{locally conical} if each $\fP$ is locally
conical  with the same $U$. As we will see shortly, any wlc geography is a locally conical decomposition of $\fN_S$.

\begin{lemma}\label{lemma-locally conical}
Let $\fD$ be a closed bounded (convex) set in a finite dimensional $\R$-vector
space. Then $\fD$ is locally conical if and only if $\fD$ is polyhedral.
\end{lemma}
\begin{proof}
We use induction on $\dim_\R\fD$.
The case $\dim_\R\fD\le 2$ is trivial.
Let $\dim_\R\fD=n\ge 3$
and assume that the lemma holds in $\dim_\R\fD< n$.
The property of being locally conical is preserved under taking hyperplane
sections. Therefore by the inductive assumption, all the sections of $\fD$ are polyhedral.
This implies that the set $\fD$ is polyhedral by the following standard criterion \cite{klee}:

$\fD$ is polyhedral if and only if
all of its $j$-dimensional plane sections are polyhedral where
$2\leq j \leq n-1$.

The converse is clear.
\end{proof}

\begin{lemma}\label{lemma-closed N}
Let $B_i\in\fN_S$ be a sequence of boundaries
converging to a limit $B=\lim_{i\to+\infty}B_i$, and
$Y/Z$ be a model of $X/Z$ such that
$(Y/Z,B_i\Lg_Y)$ is a sequence of wlc models of $(X/Z,B_i)$.
Then $(Y/Z,B\Lg_Y)$ is a wlc model of $(X/Z,B)$.
In particular, $B\in\fN_S$.
\end{lemma}

However, $(Y/Z,B\Lg_Y)$ may not be a slt wlc model of $(X/Z,B)$
even if all models $(Y/Z,B_i\Lg_Y)$ are slt wlc models of $(X/Z,B_i)$.
We call a model $(Y/Z,B\Lg_Y)$ in the lemma {\em limiting\/} from a class $\fP$ if all $B_i\in\fP$.

\begin{proof}Immediate by definition.
\end{proof}

\begin{proof}[Proof of \ref{mainthrm-geography}]
First, we establish the finiteness,
convexity, and polyhedral property for
the closures $\ofP$ of the classes in the geography.
Note that the closures do not give a 
 decomposition but, by Lemma \ref{lemma-closed N}, $\fN_S$
is a finite union of these closures. In particular, $\fN_S$ is closed.

The convexity of wlc classes follows from Lemma~\ref{lemma-convex relations}.

The polyhedral property of $\ofP$ follows
from the locally conical property of $\ofP$ by Lemma \ref{lemma-locally conical}
and its compactness.
The conical property for $\ofP$ follows from
that of $\fP$: the closure of a locally conical set is locally conical.
But $\fP$ is locally conical by
the stability of wlc models.
This means the following.
Take $B\in\fN_S$, and let $(Y/Z,B\Lg_Y)$ be a
{\em stable\/} wlc model of $(X/Z,B)$.
The stability here means that,
for all $B'$ in a neighborhood $U$ of $B$ in $\fB_S$,
the pairs $(Y,B'\Lg_Y)$ are slt initial models of $(X/Z,B')$.
Such a model $(Y/Z,B\Lg_Y)$ can be constructed from
a slt wlc model of $(X/Z,B)$ by additional blowups of
lc centers which lie in divisors $(S_i)_Y$ having
$0$ multiplicities in $B$.
Then, for any boundary $D\not=B\in \fN_S\cap U$,
there exists a (generalized) log flop
$(Y/Z,B\Lg_{Y})\dashrightarrow (Y'/Z,B\Lg_{Y'})$
such that, for all $B'\in \overrightarrow{BD}\cap U$,
the models $(Y'/Z,B'\Lg_{Y'})$ are slt
wlc models of $(X/Z,B')$, and
the boundaries $B'$ are wlc equivalent, where $\overrightarrow{BD}$
denotes the open ray from $B$ through $D$.
The existence of $Y'/Z$ follows from \cite[Corolary 9 and Addendum 5]{orderedterm}.
The equivalence of
$B'$ follows from the construction of the flop and from \cite[Corollary 11 and Addendum 6]{orderedterm}
(cf. the proof of Lemma~\ref{lemma-convex relations} above).
Note that $(Y'/Z,B\Lg_{Y'})$ may not be slt, but
by the slt LMMP the dlt condition can be replaced by
the  lc condition everywhere in \cite[remarks after Proposition 1]{orderedterm} by \cite[Conjecture and Heuristic Arguments]{appendix-nikulin}.

The finiteness also follows by induction on dimension $\fN_S$
and the stability. See \cite[Proposition 3.2.5]{choi} for details.

These established facts imply all stated properties of $\fN_S$,
except for the convexity and rationality of $\fN_S$.
Note that convexity follows from the convexity of  positive parts: for all $0\le t,t'\in\R,t+t'=1$,
$$
P(tB+t'B')\ge tP(B)+t'P(B').
$$
The slt LMMP assumption is sufficient here.
The proof of the rationality (of $\fN_S, \fP$) and the open property (of $\fP$)
uses properties of functions $p(C,B),e(D,B)$ introduced below.
Actually, for each $\fP$ there exist finitely many curves $C_i/Z$ on
a rather high model and finitely many prime
b-divisors $D_j$ such that $\fP$
in $\ofP$ can be
given by the system of inequalities:
$$
\begin{cases}
p(C_i,B)>0;\\
e(D_j,B)>0.
\end{cases}
$$
In particular, those inequalities determine any country $\fC$ in
$\fN_S$ while $\fN_S$ itself is given in $\fB_S$
(more accurately, in the minimal face containing $\fN_S$)
by nonstrict inequalities
$l(C_k,B), l(D_l,B)\ge 0$ for
some curves $C_k/Z$ and some b-divisors $D_l$,
where the linear functions $l(C_k,B),l(D_l,B)$ are
extensions of linear functions given by $p(C_k,B),e(D_l,B)$,
respectively on the linear span of $\fN_S$ 
as it will be explained in the continuation
of the proof by the end of the section,
and equations (of the face) $b_i=0$ or $1$.
The closure $\ofP$ of any other
lower dimensional class $\fP$
lies in a face of dimension $\dim_\R\fP$ of some country $\fC$
of (cf. Corollary \ref{cor-facet=face} below)
 and thus $\fP$ can be given in $\fB_S$ by strict inequalities and equations:
$p(C_i,B),e(D_j,B)>0,l(C_k,B),l(D_l,B)=0,b_i=0$ or $b_i=1$.

(To be continued)
\end{proof}

\bs

On a rather high model $V/Z$ of $X/Z$ and
for a curve $C$ on $V/Z$,
the function $p(C,-)=p(C;X/Z,-)\colon \fN_S\to\R$ is defined as
$$
(C,B)\mapsto p(C;X/Z,B)=(P(B),C)=(P(B)_V,C)=(K_Y+B\Lg_Y,g_*C),
$$
where $(Y/Z,B\Lg_Y)$ is a wlc model of $(X/Z,B)$, and
$g\colon V\to Y/Z$ is a birational morphism.
A \emph{rather high} model 
 satisfies the following universal
property: $V$ dominates some wlc model
for any $B\in\fN_S$, and $P(B)$ is $\R$-Cartier over $V$
(cf. definition of $\mob$).
Such a universal model $V/Z$ exists by the finiteness
in the first half of the proof of the theorem.
By \cite[2.4.3]{3fold-logmod}, the function is independent of the choice of such a model $V/Z$.

\bs

For a prime b-divisor $D$ of $X/Z$,
the {\em initial\/} log discrepancy $\ild(D,X,-):\fB_S\rightarrow \R$
is defined as
$$
\begin{array}{rl}
\ild(D,X,B)&=1-mult_D B\Lg\\
 &=\left\{
\begin{array}{ll}
1-mult_D B &\text{if $D$ is in $S$};\\
1 &\text{if $D$ is nonexceptional on $X$, not in $S$};\\
0 &\text{if $D$ is exceptional on $X$, not in $S$.}
\end{array}
\right.
\end{array}
$$

For a prime b-divisor $D$ of $X/Z$, the {\em immobility\/} function
$e(D,-)=e(D;X/Z,-)\colon \fN_S\to\R$
is defined as
$$
(D,B)\mapsto e(D;X/Z,B)=\mult_DF(B)=a(D,Y,B\Lg_Y)-\ild(D,X,B),
$$
where $(Y/Z,B\Lg_Y)$ is a wlc model of $(X/Z,B)$.
It is called immobility because $e(D,B)$ is the multiplicity in  b-divisors $D$ of the fixed
part $F(B)$ of $\sK+B^{log}$.
By the invariant property of log
discrepancies of wlc models, the function $e(D,B)$ is independent of the choice of the wlc model
$(Y/Z,B\Lg_Y)$ \cite[$2.4.2'$]{3fold-logmod}.

\begin{proposition}\label{prop-p-e}
The functions $p(C,B),e(D,B)$ on $\fN_S$ are nonnegative
continuous rational piecewise linear, and
linear on each $\overline{\fP}$, where
$\fP$ is a class of $\fN_S$.
Moreover, $e(D,B)$ is convex from above, and $p(C,B)$
is convex from below for rather general (mobile) curves, that is,
for a curve $C/Z$ through a general point of a rather high model $V/Z$ of $X/Z$.
\end{proposition}

\begin{proof}
The function $p(C,B)$ is nonnegative
by the nef property of b-divisor $P(B)$.
The function $e(D,B)$ is nonnegative by definition or because the fixed part $F(B)$ is effective.
The rationality and the linearity on each $\ofP$ follow from the
corresponding properties of the intersection pairing and the discrepancy function
by Lemma \ref{lemma-closed N}.
Thus the finiteness established in the first half of
the proof of Theorem \ref{mainthrm-geography} implies the continuity and
the rational piecewise linearity of the functions.
This includes also the rational property of any maximal
linear subset (cf. facets in Corollary \ref{cor-eqn for facets} below).

The convex properties of $p(C,B)$ from above and of $e(D,B)$ from below follow from the inequalities:
$P(tB+t'B')\ge tP(B)+t'P(B')$  and  $F(tB+t'B')\le tF(B)+t'F(B')$,
where $0\le t,t'$ and $t+t'=1$.
\end{proof}

\begin{corollary}\label{cor-property-p-e}
For any curve $C$ on a rather high model $V/Z$ of $X/Z$ and
for any prime b-divisor $D$ of $X/Z$, the functions $p(C,B),e(D,B)$ satisfy the following.
\begin{enumerate}
\item Each function $p(C,B),e(D,B)$ has
the constant signature on each class $\fP$: either $>0$ or $=0$ on $\fP$.
The converse holds also, that is,
if the values $p(C,B),e(D,B)$ and $p(C,B'),e(D,B')$
have the same signatures, then $B\wlc B'$.

\item $p(C,B)=(P(B),C)=0$ on a class $\fP$ if and only if
$C$ is contracted by the lc contraction
$I_\fP\colon V\to X\cn/Z$
(cf. Log canonical model in Section \ref{section7-LMMP-semi}.)
of the pair $(X/Z,B)$ with $B\in\fP$.
In particular, the contraction depends only on $\fP$.
Equivalently, $p(C,B)>0$ on $\fP$ if and only if
$I_\fP(C)$ is a curve on $X\cn/Z$.

\item $e(D,B)=0$ on a class $\fP$
if and only if $D$ can be blown up on a wlc model of this class, i.e.,
$D$ is nonexceptional on some wlc model $Y/Z$.
Equivalently, $e(D,B)>0$ on a class $\fP$
if and only if  $D$ is exceptional on any wlc model of
$(X/Z,B)$ for $B\in\fP$.

\end{enumerate}

\end{corollary}
\begin{proof}

(1) By definition  $\mob$ is equivalent to the same
signatures for all $p(C,B)$, and so is $\fix$ for $e(D,B)$.
Thus by Proposition \ref{prop-equivalence} $\wlc$ is
equivalent to the same signatures
for all $p(C,B),e(D,B)$.

(2) Immediate by construction of $V\rightarrow X\cn$ as a composition of
$g\colon V\to Y$ followed by the
lc contraction  of  a wlc model $(Y/Z,B\Lg_Y)$,
where $(Y/Z,B\Lg_Y)$ is an appropriate wlc model of $(X/Z,B)$.
The composition contracts exactly
the curves $C$ on $V/Z$ with $(K_Y+B\Lg_Y,g_*C)=0$ (cf. \cite[1.17 (ii)]{isksh}).

(3) By definition $e(D,B)=0$ for any wlc model $(Y/Z,B\Lg_Y)$, on which $D$ is nonexceptional.

On the other hand, if $e(D,B)=0$ and
$D$ is exceptional
on a wlc model $(Y/Z,B\Lg_Y)$, then
$a(D,Y,B\Lg_Y)=\ild(D,X,B)\le 1$ and one can blow up $D$
to a crepant slt wlc model $(Y'/Z,B\Lg_{Y'})$ of $(X/Z,B)$ with
a (higher) model $Y'/Y$ of $Y/Y$
by the slt LMMP.
By (1) this depends only on $\fP$.
\end{proof}
\bs

\begin{corollary}\label{cor-facet=face}
Let $\fP$ be a class of $\fN_S$ and $\fF$ be its face.
Then $\Int\fF\subseteq\fP'$ for some class $\fP'$ of $\fN_S$.
In particular, $\Int\ofP\subseteq \fP'=\fP$.

Thus the geography $\fN_S$ is determined up to finitely many
possibilities by its countries $\fC$ and even their closures $\ofC$.
\end{corollary}

Note that $\ofP$ is a convex polyhedron by the first half of
the proof of Theorem \ref{mainthrm-geography}.

\begin{proof}
Immediate by Proposition \ref{prop-p-e} and Corollary \ref{cor-property-p-e} (1).
An elementary property of linear functions is useful here:
a nonnegative linear function on a convex polyhedron $\fF=\ofF$
has the same signature on $\Int\fF$.
\end{proof}

\begin{corollary}\label{cor-eqn for facets}
Let $\fP$ be a class of $\fN_S$,
$\fF$ be its facet, and
$V/Z$ be a rather high model of $X/Z$.
Then either

$\fF$ is given
by a facet of $\fB_S$; 

there exists a curve $C$ on a rather high model $V/Z$ of $X/Z$
such that $p(C,\fP)>0$ and $p(C,\fF)=0$; or

there exists a prime b-divisor $D$ of $X/Z$ such that
$e(D,\fP)>0$ and $e(D,\fF)=0$.

Moreover, if $\Int\fF\cap\fP=\emptyset$, then
$\fF\cap\fP=\emptyset$.
However, if $\Int\fF\cap\fP\not=\emptyset$, then
$\Int\fF\subseteq\fP$, and this is possible only for facets 
given by a facet of $\fB_S$.
The same holds for the faces of $\fP$.
\end{corollary}

\begin{proof}
Take $B\in \Int\fF$, and an outer direction $\overrightarrow{BC}$
from $\fP$, that is, $\overrightarrow{BC}\cap\fP=\emptyset$ and
$\overline{BC} \cap\fP\not=\emptyset$.
Let $(Y/Z,B\Lg_Y)$ be a wlc model of $(X/Z,B)$ which is
a limit of (slt) wlc model $(Y/Z,B'\Lg_Y)$ of $(X/Z,B')$
for $B'\in \fP$ (e.g., along $\overline{BC}$; see Lemma \ref{lemma-closed N}).

 By the polyhedral property of lc \cite[1.3.2]{3fold-logflips},
the following three cases are only possible:

(i) $B'$ is not a b-boundary
for any $B'\in\overrightarrow{BC}$;

(ii) $(Y,B'\Lg_Y)$ is not lc for any $B'\in\overrightarrow{BC}$;

(iii) for any $B'\in\overrightarrow{BC}$ in
a neighborhood of $B$, $B'$ is a b-boundary and $(Y,B'\Lg_Y)$ is lc.

In the case (i), there exists prime $S_i$,
such that $b'_i=\mult_{S_i} B'\Lg_Y<0$ or $>1$. 
The facet $\fF_i$ of $\fB_S$ with the equation $b_i=0$ or $b_i=1$
respectively gives $\fF$ (see the property of function $b_i$ on page 5 above).

In the case (ii), there exists a prime b-divisor $D$ of $X/Z$
such that $a(D,Y,B'\Lg_Y)<0$ for any $B'\in\overrightarrow{BC}$.
If we do not assume (i), $D$ is exceptional on $Y$
and, by the linear property of discrepancy,
$a(D,Y,B\Lg_Y)=0$ and $a(D,Y,B'\Lg_Y)>0$ for any
$B'\in\overline{BC}\cap\fP$. Thus $e(D,B)=0$ and,  if for every $B'\in \overline{BC}\cap\fP$,
$e(D,B')=0$, then for all $B'$ near $B$ on $\overline{BC}$,
$\ild(D,X,B')=a(D,Y,B'\Lg_Y)$ and $<0$ for $B'\in\overrightarrow{BC}$,
a contradiction again by the negation of (i).
Hence $e(D,B')>0$ for some $B'\in \overline{BC}\cap\fP$.
This gives required $D$ by Corollary \ref{cor-facet=face}.

The case (iii) itself has two subcases:
for all $B'\in\overrightarrow{BC}$ in a neighborhood of $B$,

(iii-1) $(Y/Z,B'\Lg_Y)$ is not a wlc model; or

(iii-2) $(Y/Z,B'\Lg_Y)$ is a wlc model but not of $(X/Z,B')$.

\noindent
Indeed, otherwise, for all those $B'$, the pairs $(Y/Z,B'\Lg_Y)$ are
wlc models of $(X/Z,B')$.
By definition all functions $p(C,B'),e(D,B')$
are linear on $\overline{BC}$ in the neighborhood of $B$.
Thus $p(C,B'),e(D,B')$ have the same signatures and all those $B'$ are $\wlc$,
a contradiction with our assumptions.

In the case (iii-1), by \cite[Corollary~9]{orderedterm}
there exists a curve $C'$ on $Y/Z$ such that
$(K_Y+B'\Lg_Y,C')<0$ for $B'\in\overrightarrow{BC}$ and
$(K_Y+B\Lg_Y,C')=0$.
Thus $(K_Y+B'\Lg_Y,C')>0$
for $B'\in\overline{BC}\cap\fP$,
and for a curve $C$ on a rather high model $V'/Z$ of $X/Z$
with $g'\colon Z'\to Y$
and $C'=g'_*C$, again by Corollary \ref{cor-facet=face},
$p(C,\fP)>0$ and $p(C,\fF)=0$.
The same holds for a birational transform of $C$ on $V/Z$ \cite[2.4.3]{3fold-logmod}.

In the case (iii-2), there exists a b-divisor $D$ with
$a(D,Y,B'\Lg_Y)<\ild(D,X,B')$.
By the linear property of discrepancy and of initial discrepancy,
$e(D,B')>0$
for $B'\in\overline{BC}\cap\fP$ and
$e(D,B)=0$.
Thus $D$ is a required divisor.

The results about $\fF\cap\fP$ follow from
Corollaries \ref{cor-property-p-e} (1) and \ref{cor-facet=face};
in the case of faces by induction.
\end{proof}

\begin{proof}[Continuation of the proof of Theorem \ref{mainthrm-geography}]
Let $\fF\subseteq\fB_S$ be the minimal face of $\fB_S$ containing $\fN_S$.
Then the linear span of $\fN_S$ is the linear span of $\fF$ by
the monotonicity $\nu(X/Z,B')\ge \nu(X/Z,B)$ for $B'\ge B$ in $\fB_S$
(cf.  Proposition \ref{prop-monotone N}).
Moreover, the linear span of $\fN_S$ can be given in $\fD_S$ by
equations $b_i=1$.
In particular, the span is rational.

The rationality of $\fN_S$, of the polyhedrons $\ofP$, and
the open property of the classes $\fP$ follow
from Corollaries \ref{cor-facet=face} and \ref{cor-eqn for facets}.
This gives also required equations and inequalities.
The linear function $l(C,B)$ on the linear span of $\fP$
is determined by the linear function $p(C,B)$ on $\fP$, in
particular, those functions with $p(C,\fF)=0$
for a facet $\fF$ of $\fP$ will suffice.
Similarly, one can introduce functions $l(D,B)$.
Finally, the interior $\Int\fF$ of a face of $\fP$ given by a facet of $\fB_S$,
belongs to $\fP$ or is disjoint from it.

If $\fP,\fP'$ are two classes and $\overline{\fP'}$
is a face of $\ofP$, then
$\Int\fP'\cap\ofP\not=\emptyset$.
Conversely, if $\Int\fP'\cap\ofP\not=\emptyset$ but
$\overline{\fP'}$ is not a face of $\ofP$,
then by Corollary~\ref{cor-facet=face},
$\emptyset\not=\Int\fP'\cap\ofP\not=\Int\fP'$.
Moreover, by the convexity of classes and of $\fN_S$, we can suppose
that $\overline{\mathfrak P'}$ contains a face $\fF'$ of $\ofP$
(maybe for different $\fP$) which is a polyhedron of the same dimension as $\fP'$ and
with $\fF'=\ofP\cap\overline{\mathfrak P'}\not=\overline{\fP'}$.
Hence there exists a facet $\fF\not\supset\fF'$ of $\fP$ which
contains an internal point of $\fP'$.
This gives a contradiction with Corollary~\ref{cor-eqn for facets}.
For example, if $\fF$ is given by a facet of $\fB_S$ with
an equation $b_i=0$ or $1$, then $b_i=0$ or $1$ respectively on $\overline{\mathfrak P'}$,
in particular, on $\fF'$, and thus
on $\ofP$, a contradiction.
Similarly, for other functions.
\end{proof}

\bs
\section{Positive cones}\label{section3-positive cone}

\ms

For first applications of geography,
we recall and define natural cones in $\fD_S$ and $\N^1(X/Z)$.
Let $W/Z$ be a model of $X/Z$, and $S=\sum S_i$ be a reduced b-divisor of $X/Z$.
Our {\em  numerical space\/} $\N^1(X/Z)$ denotes the space of $\R$-divisors,
not necessarily $\R$-Cartier, mudulo the numerical equivalence $\equiv$.
Nonetheless, the relative Weil-Picard (class) number
$\dim_\R \N^1(X/Z)<+\infty$ for proper $X/Z$.

\ms

\paragraph{Semiample cones.}
In the numerical space $\N^1(X/Z)$ and in the space of b-divisors $\fD_S$, respectively, the cones
$$
\begin{array}{rl}
\sAmp(X/Z)&:=\left\{[D] \in \N^1(X/Z)\left|\begin{array}{l}
D\equiv D' \text{ for some semiample}\\
\text{divisor $D'$ on }X/Z
\end{array}\right.\right\}_,\\
\samp_S(W/Z)&:=\{D\in\fD_S\mid D_W \text{ is semiample on } W/Z\}
\end{array}
$$
are defined.

\paragraph{Nef cones.}
In $\N^1(X/Z), \fD_S$, respectively, the cones
$$
\begin{array}{rl}
\Nef(X/Z)&:=\{[D] \in \N^1(X/Z)\mid D \text{ is nef on }X/Z\},\\
\nef_S(W/Z)&:=\{D\in\fD_S\mid D_W \text{ is nef on } W/Z\}
\end{array}
$$
are defined and closed.
In the projective case, they are
the closures of the ample (K\"ahler) cones.
The cone $\Nef(X/Z)$ is dual to the Kleiman-Mori cone,
the closure of the cone of curves.
Thus its polyhedral properties are
related to the cone of curves (the Mori theory).
However, we consider only cones of divisors (see Corollary \ref{cor-positive cones} below).

\paragraph{Mobile cones.}
In $\N^1(X/Z), \fD_S$, respectively, the cones
$$
\begin{array}{rl}
\Mob(X/Z)&:=\left\{[D]\in \N^1(X/Z)\left|\begin{array}{l}
 D\equiv D' \text{ for some $\R$-mobile }\\
\text{divisor $D'$ on }X/Z
\end{array}\right.\right\}_,\\
\mobile_S(W/Z)&:=\{D\in\fD_S\mid D_W \text{ is $\R$-mobile on } W/Z\}
\end{array}
$$
are defined.

\paragraph{Effective cones.}
In $\N^1(X/Z), \fD_S$, respectively, the cones
$$
\begin{array}{rl}
\Eff(X/Z)&:=\left\{D\in\N^1(X/Z)\left|\begin{array}{l}
 D\equiv D' \text{ for some effective divisor}\\
\text{$D'$ on }X/Z
\end{array}\right.\right\}_,\\
\eff_S(W/Z)&:=\left\{D\in\fD_S\left|\begin{array}{l}
 D_W\sim_\R D' \text{ for some effective divisor}\\
\text{$D'$ on }W/Z
\end{array}\right.\right\}
\end{array}
$$
are defined.

\bs

In general, the numerical equivalence $\equiv$ of
divisors does not preserve the semiampleness and mobile
properties of divisors
while clearly $\sim_{\R}$ does.
Thus introducing cones in the numerical space,
we had a choice to assume such a property on the whole class $[D]$
or on some representative $D'\in[D]$.
We chose the second one to have a larger cone.
According to the choice, the linear map $[\ ]$
agrees the $\sim_\R$ equivalence property with the numerical one
(cf. the proof of Corollary \ref{cor-0log num cones}).
It is well-known that the above cones, numerical and linear,
are convex and satisfy the inclusions:
$$
\begin{array}{rl}
\sAmp(X/Z)\subseteq\Nef(X/Z),&\sAmp(X/Z)\subseteq\Mob(X/Z)\subseteq \Eff(X/Z),\\
\samp_S(W/Z)\subseteq\nef_S(W/Z),&\samp_S(W/Z)\subseteq \mobile_S(W/Z)\subseteq \eff_S(W/Z).
\end{array}
$$

\bs

We recall also that:
\begin{itemize}
\item[]
{\em $0$-log pair} $(X/Z,B)$ is a pair with a proper morphism $X/Z$ and a boundary $B$ on $X$ such that $(X,B)$
is klt and $K+B\sim_{\R}0$,  
or equivalently $K+B\equiv0$   according to \cite{ambro}; and

\item[]
(relative) {\em FT (Fano type)\/} variety $X/Z$ is a variety such that there exists
a boundary $B$ on $X$ such that the pair $(X/Z,B)$
is a klt log Fano variety, in particular,
$X/Z$ is projective.

\end{itemize}
For other equivalent characterizations of relative FT  varieties, see \cite[Lemma-Definition 2.8]{prsh}.
The most useful one among them to us is the following.

\begin{lemma}\label{lemma-FT=0log}
 Let $X/Z$ be be an FT variety. Then for any fixed
reduced divisor $S$ of $X$,
there exists an $\R$-boundary $B$
on $X$ such that
$(X/Z,B)$ is a $0$-log pair and $S\subseteq \Supp B$.
Moreover, $B$ is big.

Conversely, if $(X/Z,B)$ is a $0$-log pair with projective $X/Z$ and
big $B$, e.g., prime components of $\Supp B$ generate $\N^1(X/Z)$,
then $X/Z$ is FT.
\end{lemma}

\begin{proof}
See \cite[Lemma-Definition 2.8]{prsh}.
The big property for $B$ follows from
the big property of $-K$ for FT varieties.
\end{proof}

\bs

To obtain polyhedral properties of cones in $\fD_S$
from geography of log models, we
translate polyhedrons in the space.
Let
$$
\begin{array}{rl}
-B\colon \fD_S &\to \fD_S\\
D &\mapsto D-B
\end{array}
$$
be the translation by $-B$.
It translates $B$ into the origin $0\in\fD_S$, the
vertex of cones.
Geography allows to construct only some
polyhedrons in $\fB_S\subset\fD_S$.
However each polyhedron near $B$ is conical.
Thus to compare the translation of a polyhedron
with a cone in a neighborhood $U$ of $0\in\fD_S$ would be sufficient.
In what follows, we say that two sets $\fS_1,\fS_2$ (cones or polyhedrons)
coincide in $U$ if
$$
\fS_1\cap U=\fS_2\cap U.
$$

\begin{proposition}\label{prop-0log cones coincide in U}
Let $(X/Z,B)$  be a $0$-log pair and put $S=\Supp B$.
Then there exists a neighborhood $U$ of $0$ in $\fD_S$ in
which:
$$
\begin{array}{rl}
\samp_S(X/Z) &=\nef_S(X/Z)=\ofP_X-B=\cup\fP_{0,X}-B,\\
\mobile_S(X/Z)&=\cup \ofP^1_X-B=\cup \fP^1_X-B, \\
\eff_S(X/Z)&=\fN_S-B,
\end{array}
$$
where $\ofP_X$ is the closure of a
maximal dimensional class $\fP_X$
of wlc models with $Y=X$,
$\fP_{0,X}$ are the classes of wlc models with $Y=X$, and
$\ofP^1_X$ are the closures of
classes $\fP^1_X$ of wlc models with $Y$ isomorphic to $X$
in codimension $1$.

Each divisor $D\in\eff_S(X/Z)$ has a unique mob+exc decomposition:
$D\sim_\R M+E$ and $\equiv M+E$, where $M$ is $\R$-mobile and
$E$ is effective very exceptional with respect to a $1$-contraction
given by $M$; the $1$-contraction, $E$ are unique and $M$ is unique up to $\sim_\R$ and $\equiv$, respectively.
In particular, in $\fD_S$ the $\R$-mobile and $\sim_\R$ effective properties of
divisors are equivalent to those  for $\equiv$.
Moreover,
for $D\in U\cap\eff_S(X/Z)$, $M\sim_\R P(B')_X,E=F(B')_X$ and
$M\equiv P(B')$, respectively, where $B'=B+D$.
\end{proposition}

\begin{proof}
First we clarify the statement.
Let $B'\in\fP\subseteq\fN_S$ be a boundary in a class $\fP$.
If the pair $(X/Z,B')$ has a wlc model $(Y/Z,B'_Y)$ for which
$X\dashrightarrow Y$ is a small birational isomorphism, that is,
an isomorphism in codimension $1$, we denote the corresponding
class by $\fP=\fP^1_X$.
In particular, if $Y=X$, we have  
 isomorphism in codimension $\le \dim_k X$, or
in dimension $\ge 0$,
and we denote it by $\fP_{0,X}$.
If $X$ is nonprojective or/and non $\Q$-factorial, we allow such
models.  
But if we have to follow Convention \ref{convention}, we need
to replace the models by their small projective $\Q$-factorializations \cite[Corollary 6.7]{isksh}
(cf. Proposition \ref{prop-qfact-proj}).
For $B'\in\fP_{0,X}$, in addition, we suppose that
$(X/Z,B')\dashrightarrow (Y/Z,B'_Y)$ is a small log flop.
Thus we can replace $X/Z$ by a small projective $\Q$-factorialization.

Now we choose a conical 
neighborhood $U$ of $0\in\fD_S$ such
that,

for each $D\in U$, the pair $(X/Z,B')$ with $B'=B+D$ is
a klt initial model of itself, including the boundary condition;

$(U+B)\cap\fP$
is conical for each class $\fP\subseteq\fN_S$
(this is sufficient for the classes $\fP$ with $B\in\ofP$).

\noindent
Such a neighborhood exists by our assumptions and
by Theorem \ref{mainthrm-geography}.

For any divisor $D\in U$,
$$
 D\sim_\R K+B+D=K+B'.
$$

Thus
$$
D \text{ is semiample }\Leftrightarrow K+B' \text{ is semiample};
$$

$$
D \text{ is nef }\Leftrightarrow K+B' \text{ is nef};
$$

$$
D \text{ is $\R$-mobile }\Leftrightarrow K+B' \text{ is $\R$-mobile};
$$

$$
\;D\sim_\R D'\ge 0\Leftrightarrow K+B'\sim_\R D'\ge 0.
$$

This gives the required equalities and mob+exc decompositions.
Indeed, since $(X,B')$ is lc,
the $\R$-mobile part $M$ of $D$ is $\sim_\R P(B')_X$ and
respectively, the very exceptional part $E$ is $F(B')_X$.
The $\R$-linear decomposition gives the numerical one because the former
is essentially numerical by the LMMP (see \cite[Definition 3.3 and Proposition 3.4]{plflip}).
The semiampleness gives the equality $\samp_S(X/Z)=\nef_S(X/Z)$.
The closed property and the equalities with
closures follow from Lemma \ref{lemma-closed N} or Corollary \ref{cor-property-p-e} (1).

Finally, by Corollary \ref{cor-property-p-e} (1), by
convexity of $\nef_S(X/Z)$ and by the open property of classes,
$\Int \nef_S(X/Z)=\fP_X$.
holds for one class $\fP_X$ (cf. \cite[Lemma 2]{orderedterm}).

To construct all the required models, the klt slt LMMP is sufficient, and
for contractions of those
models, the semiampleness for
klt $\Q$-boundaries is sufficient (see Corollary \ref{cor-Qsamp=Rsamp}).
\end{proof}

\begin{corollary}\label{cor-0log cones}
Let $(X/Z,B)$  be a $0$-log pair and $S\subseteq \Supp B$ be a reduced divisor on $X$.
Then the cones
$$
\samp_S(X/Z)=\nef_S(X/Z), \;\mobile_S(X/Z),\eff_S(X/Z)
$$
in $\fD_S$ are
 closed convex rational polyhedral.

Each divisor $D\in\eff_S(X/Z)$ has a unique mob+exc decomposition:
$D\sim_\R M+E$ and $\equiv M+E$, where $M$ is $\R$-mobile and
$E$ is effective very exceptional with respect to a $1$-contraction
given by $M$.
Moreover,
$\equiv$ instead of $\sim_\R$ gives the same cones in $\fD_S$.
\end{corollary}

\begin{proof}
Immediate by Proposition \ref{prop-0log cones coincide in U} and
Theorem \ref{mainthrm-geography} for $S=\Supp B$.
If $S\subseteq\Supp B$, we obtain the required polyhedral properties by
intersection with the subspace $\fD_S$.
\end{proof}

\begin{corollary}\label{cor-0log num cones}
Let $(X/Z,B)$  be a $0$-log pair such that
prime components of $S=\Supp B$ generate $\N^1(X/Z)$.
Then the cones
$$
\sAmp(X/Z)=\Nef(X/Z),\Mob(X/Z),\Eff(X/Z)
$$
in $\N^1(X/Z)$ are closed convex rational polyhedral.
Thus $\Eff(X/Z)$ is the pseudo-effective cone too.

The relations $\equiv$ and $\sim_\R$ coincide on $X/Z$.
Each divisor $D\in\Eff(X/Z)$ has a unique mob+exc decomposition:
$D\equiv M+E$, where $M$ is $\R$-mobile and
$E$ is effective very exceptional with respect to a $1$-contraction
given by $M$.
\end{corollary}

The corollary gives a new result even for $\Nef(X/Z)$ when $X/Z$ is nonprojective.

\begin{proof}
The required properties of the linear cones in $\fD_S$
imply the same properties of the numerical cones in $\N^1(X/Z)$.
Indeed, the natural linear map
$$
\begin{array}{rl}
[\ ]\colon \fD_S &\rightarrow  \N^1(X/Z)\\
            D  &\mapsto [D],\text{ the numerical class of } D,
\end{array}
$$
maps the cones
$$
\samp_S(X/Z),\nef_S(X/Z)
,\mobile_S(X/Z),\eff_S(X/Z)\rightarrow\qquad\qquad\qquad\qquad$$
$$
\qquad\qquad\sAmp(X/Z),\Nef(X/Z),\Mob(X/Z),\Eff(X/Z) \text{, respectively.}
$$
Since  prime components $S_i$ of $S$ generate $\N^1(X/Z)$,
the map $[\ ]$ and the induced maps of
cones are surjective, that gives the above implication
by Corollary \ref{cor-0log cones}.
The surjectivity for cones uses the following:
$\sim_\R$ can be replaced by $\equiv$ in the definition of cones in $\fD_S$
by the same corollary.
Moreover,
the coincidence of $\equiv$ and $\sim_{\R}$ holds on $X/Z$.
This is well-known for projective $X/Z$
\cite[Lemma 4.1.12]{choi} (the proof uses the assumption that $\Char k=0$;
possibly, Conjecture \ref{conj-semiample} in the positive characteristic allows us to omit the assumption).
The general case uses the rationality of singularities of $X$.

All required models in construction
 will be FT varieties by Lemma \ref{lemma-FT=0log}.
Semiampleness holds for
FT varieties in characteristic $0$:
for $\Q$-divisors, it holds by the base freeness of \cite{nonvanishing},
that is sufficient for $\R$-divisors by Corollary \ref{cor-Qsamp=Rsamp}
(cf. \cite[Corollary~10]{orderedterm}).

By Lemma \ref{lemma-FT=0log}, the big klt slt LMMP is sufficient
for the FT varieties.

\end{proof}

\begin{corollary}\label{cor-positive cones}
Let $X/Z$ be an FT variety, and $S$ be a
reduced divisor of $X$.
Then the cones
$$
\begin{array}{l}
\sAmp(X/Z)=\Nef(X/Z),\Mob(X/Z),\Eff(X/Z)\text { and }\\
\samp_S(X/Z)=\nef_S(X/Z),\mobile_S(X/Z),\eff_S(X/Z)
\end{array}
$$
in $\N^1(X/Z)$ and in $\fD_S$,
 respectively are
closed convex rational polyhedral.

The relations $\equiv$ and $\sim_\R$ coincide on $X/Z$.
Each class $[D]\in\Eff(X/Z)$ (respectively divisor $D\in\eff_S(X/Z)$)
has a unique decomposition
mob+exc: $[D]=[M]+[E]$ (respectively $D\sim_\R M+E$),
where $M$ is $\R$-mobile and
$E$ is effective very exceptional with respect to a $1$-contraction
given by $M$.
\end{corollary}

\begin{proof}
Immediate by  Lemma \ref{lemma-FT=0log} and Corollaries \ref{cor-0log cones}, \ref{cor-0log num cones}.
We use here the finiteness of the Weil-Picard number:
$\dim_\R \N^1(X/Z)<+\infty$.
\end{proof}

We will give a natural decomposition of the cones $\Eff(X/Z), \Mob(X/Z)$ in $\N^1(X/Z)$
and  $\eff_S(X/Z), \mobile_S(X/Z)$ in $\fD_S$
in the next section.

\bs

\section{Finiteness results for log models}\label{section4-finiteness}
\bs

The support of a b-divisor $D=\sum d_iD_i$ is defined as the reduced b-divisor $\Supp D:=\sum_{d_i\neq 0} D_i$.
So, a reduced b-divisor $D=\sum D_i$ is identified with its support: $D=\Supp D$.

Let $(X/Z,B)$ be a $0$-log pair and $S=\Supp B$.
Then the geography of $\fN_S$ in Proposition \ref{prop-0log cones coincide in U} {\em induces\/}
a finite decomposition of the cone $\eff_S(X/Z)$ into  open convex
rational polyhedral subcones $\fP$.
The open and  polyhedral properties follow
from that of classes in
$\fN_S$ by Theorem \ref{mainthrm-geography}.
Note that $B$ is an internal point of $\fB_S$.
This decomposition gives also a decomposition of $\eff_{S'}(X/Z)$
for any reduced $S'\le S$
by Corollary \ref{cor-0log cones};
its subcones
will be called also classes.
To give an internal interpretation of this
decomposition, we introduce the following relation.

For divisors $D,D'\in \eff_S(X/Z)$, $D\me D'$ if, for their mob+exc decompositions
$D\sim_\R M+E,D'\sim_\R M'+E'$,
the rational $1$-contractions $X\dashrightarrow Y_\fP=Y=Y'/Z$ given by $M,M'$
are the same, and $\Supp E=\Supp E'$.
The relation $\sim_\R$ can be replaced by $\equiv$ and
used for classes: $[D]\me[D']$.
The corresponding {\em induced\/} classes of $\Eff(X/Z)$
in $\N^1(X/Z)$ will be denoted by $\rP$
and their  contractions by $X\dashrightarrow Y_{\rP}/Z$.
The rational $1$-contraction depends only on the class $\fP$  and $\rP$,  respectively,
and will be referred to as the $D$-contraction.
If $X/Z$ is a birational contraction with respect to $D$, then $Y_P/Z$ is
small and well-known as a $D$-flip of $X/Z$ \cite[p. 2684]{3fold-logmod}.
By construction, it is also a log flop of $(X/Z,B)$ in the latter case.

(\emph{Warning}: $D$-contraction is not a contraction with respect to a divisor $D$ in the $D$-MMP.)

In particular, this gives polyhedral decompositions of cones
$\Eff(X/Z)$, $\Mob(X/Z)$ in $\N^1(X/Z)$
and  $\eff_S(X/Z),$ $ \mobile_S(X/Z)$ in $\fD_S$ for
$0$-log pairs $(X/Z,B)$, when prime components of $S=\Supp B$ generate $\N^1(X/Z)$,
and for FT varieties $X/Z$.
The decomposition of mobile cones $\Mob(X/Z),$ $\mobile_S(X/Z)$ is
interesting for applications (see Proposition \ref{prop-FT mobile cones}) rather than their general structure:
each class of those cones is a class in $\Eff(X/Z)$ or $\eff_S(X/Z)$
respectively and it is determined by its $1$-contraction ($E=0$).
So, we focus on classes of effective cones.

Let $g\colon X\dashrightarrow Y/Z$ be a $1$-contraction with projective $Y/Z$.
A polarization on $Y/Z$ is an ample $\R$-divisor $H$.
The polarization is called \emph{supported} in $S$ if
$g^*H+E\sim_\R D\in\fD_S$
for some effective divisor $E$ on $X$, very exceptional on $Y$, where
$g^*H$ is the pull back (transform) of $H$ \cite[p. 84]{plflip}; in general, $g^*H$ is an $\R$-Weil divisor.

\begin{corollary}\label{cor-0log eff decomposition}
Let $(X/Z,B)$  be a $0$-log pair and $S\subseteq \Supp B$ be a reduced divisor on $X$.
Then the decomposition of the cone $\eff_S(X/Z)$ induced by
$\fN_S$ can be given by the mob+exc relation $\me$.

There are finitely many rational $1$-contractions $X\dashrightarrow Y_\fP/Z$, and
$Y_\fP/Z$ is projective with a polarization supported in $S$.
Conversely, any such contraction corresponds to a
class $\fP$.
\end{corollary}

\begin{proof} Immediate by Proposition \ref{prop-0log cones coincide in U}, Corollary \ref{cor-0log cones}
and Corollary \ref{cor-property-p-e} (1).
Note that, for $D,D'\in U$ such that
$D+B,D'+B\in\fN_S$,
in the proposition,
the contraction condition of $\me$
means that $B+D\mob B+D'$, and
the support condition that $\Supp F(B+D)_X=\Supp F(B+D')_X$.
But for any prime divisor $E$ exceptional  on $X$,
$e(E,B+D),e(E,B+D')>0$ by the klt property in $U$.
Thus the support condition means that $B+D\fix B+D'$, or
equivalently,
$F(B+D),F(B+D')$ have the same signatures.

The finiteness of rational $1$-contractions $X\dashrightarrow Y_\fP/Z$
follows from the finiteness of geography by Theorem~\ref{mainthrm-geography}.
By construction, each $Y_\fP/Z$ is projective with
a polarization supported in $S$.
The converse is by definition.
\end{proof}

\begin{corollary}\label{cor-0log num eff decomposition}
Let $(X/Z,B)$  be a $0$-log pair such that components of $S=\Supp B$
generate $\N^1(X/Z)$.
Then the decomposition of the
cone $\Eff(X/Z)$ induced by
$\fN_S$ can be given by the mob+exc relation $\me$.

There are finitely many rational $1$-contractions $X\dashrightarrow Y_{\rP}/Z$, and
$Y_{\rP}/Z$ is projective.
Conversely, any such contraction corresponds to a class $\rP$.
\end{corollary}

\begin{proof}Immediate by
the above corollary and Corollary \ref{cor-0log num cones}.
\end{proof}

\begin{corollary}\label{cor-FT eff decomposition}
Let $X/Z$ be an FT variety and $S$ be a reduced divisor on $X$.
Then the decomposition of cones $\Eff(X/Z),\eff_S(X/Z)$ induced by
$\fN_S$ can be given by the mob+exc relation $\me$.

There are finitely many rational $1$-contractions $X\dashrightarrow Y_{\rP},Y_\fP/Z$, and
$Y_{\rP},Y_\fP/Z$ are projective with a polarization supported in $S$ for $Y_\fP$.
Conversely, any such contraction corresponds to a class $\rP,\fP$.
\end{corollary}

\begin{proof}
Immediate by Lemma \ref{lemma-FT=0log} and Corollary \ref{cor-0log num eff decomposition}.
\end{proof}

\begin{corollary}\label{cor-0log finite contractions}
Let $(X/Z,B)$  be a $0$-log pair such that
components of $S=\Supp B$
generate $\N^1(X/Z)$.
Then there are only finitely many rational $1$-contractions $X\dashrightarrow Y/Z$
with projective $Y/Z$.
\end{corollary}

\begin{proof}Immediate by Corollary \ref{cor-0log num eff decomposition}.
\end{proof}
\begin{corollary}\label{cor-FT finite contractions}
Let $X/Z$ be an FT variety.
Then there are only finitely many rational $1$-contractions $X\dashrightarrow Y/Z$
with projective $Y/Z$.
\end{corollary}

\begin{proof} Immediate by Lemma \ref{lemma-FT=0log} and Corollary \ref{cor-0log finite contractions}.
\end{proof}
\begin{theorem}\label{thrm-finite wlc big klt}
Let $(X/Z,B)$ be a klt log pair such that $K+B$ is big.
Then there are only finitely many projective
wlc models $(Y_i/Z,B\Lg_{Y_i})$ of $(X/Z,B)$.
\end{theorem}

\begin{proof}
Let $(Y/Z,B\Lg_Y)$ be a slt wlc model of $(X/Z,B)$.
Then by the monotonicity \cite[Lemma 2.4]{isksh},
$(Y/Z,B\Lg_Y)$ is also klt and thus $B\Lg_Y=B_Y$.
Let $(X\cn/Z,B\cn)$ be its lc model.
By construction,
$(X\cn/Z,B\cn)$  is also klt and projective.
It is known also that any wlc model $(Y/Z,B_Y)$ of $(X/Z,B)$ is
a generalized log flop of $(X\cn/Z,B\cn)$.
Generalized means that it can blow up divisors $D$
of $X\cn$, but only with log discrepancies $a(D,X\cn,B\cn)\le 1$.
Thus any wlc model $(Y/Z,B_Y)$ of $(X/Z,B)$ is
a $1$-contraction $Y'\dashrightarrow Y/X\cn$ of a projective
\emph{terminalization} $(Y'/Z,B')$
where $(Y'/X\cn,B')$ is a crepant blow up of all prime
$D$ with $a(D,X\cn,B\cn)\le 1$.
Since the set of those $D$ is finite \cite[Corollary 1.7]{3fold-logmod},
such a  blow up exists.
The terminalization $(Y'/X\cn,B')$
 is FT by Lemma \ref{lemma-FT=0log}
(for birational contraction $Y'/X\cn$ any divisor is big) and
the number of $1$-contractions into projective models
is finite by Corollary \ref{cor-FT finite contractions}.
\end{proof}

\begin{corollary}\label{cor-finite big minimal}
Let $X/Z$ be a relative variety of general type.
Then there are only finitely many
projective minimal models of $X/Z$.
\end{corollary}

\begin{proof}
 Immediate by Theorem \ref{thrm-finite wlc big klt}
applied to a resolution of singularities of $X/Z$.
\end{proof}

$D$-contractions can be constructed by the $D$-MMP with the initial model $(X,D)$ \cite[1.1]{isksh}.

\begin{corollary}\label{cor-DMMP on 0log}
Let $(X/Z,B)$  be a $0$-log pair with projective $X/Z$ and
$D$ be an $\R$-Cartier divisor on $X$ supported in $\Supp B$.
Then the $D$-MMP holds for $(X,D)$ and, if $D$ is pseudo-effective,
then the $D$-MMP followed by the $D$-contraction of a resulting model
gives a rational $1$-contraction $X\dashrightarrow Y_\fP/Z$
for the class $\fP$ of $D$.
\end{corollary}
\begin{proof}It is well-known that, in this situation, the $D$-MMP is
equal to the LMMP with an initial model $(X/Z,B+\varepsilon D)$
for $0<\varepsilon$ and
$\varepsilon D\in U$ of Proposition \ref{prop-0log cones coincide in U}.
Each birational transformation is a rational $1$-contraction and
they are not isomorphic as log pairs (monotonicity)
of log discrepancies).
Existence of transformations (divisorial contractions and flips) and
their termination hold by the klt slt LMMP.
In the pseudo-effective case, we get a wlc model
and the required rational $1$-contraction
by the semiampleness.
Otherwise, $D$ will be negative on some fibration and
thus not pseudo-effective.
\end{proof}

\begin{corollary}\label{cor-num DMMP 0log}
Let $(X/Z,B)$  be a $0$-log pair with projective $X/Z$ and
such that the components of $S=\Supp B$ generate
$\N^1(X/Z)$.
Then for any $\R$-Cartier divisor $D$ on $X$,
the $D$-MMP holds for $(X,D)$ and, if $D$ is pseudo-effective,
then the $D$-MMP followed by the $D$-contraction of a resulting model
gives a rational $1$-contraction $X\dashrightarrow Y_\fP/Z$
for the class $\fP$ given by $D$.
The class $\fP$ can be replaced by numerical $\rP$.

Termination in this case is universally bounded
(the bound depends only on the variety $X/Z$).
\end{corollary}

\begin{proof}
Immediate by Corollary \ref{cor-DMMP on 0log}.
The boundedness for termination follows from Corollary \ref{cor-FT finite contractions},
because $X/Z$ is FT (cf. the following lemma).
\end{proof}

\begin{corollary}\label{cor-DMMP on FT}
Let $X/Z$ be an FT variety. Then
the $D$-MMP holds for $(X,D)$
and, if $D$ is pseudo-effective,
then the $D$-MMP followed
by the $D$-contraction of
a resulting model gives a rational $1$-contraction $X\dashrightarrow Y_{\rP}/Z$
for the class $\rP$ of $\Eff(X/Z)$ containing $D$.

Termination in this case is universally bounded
(the bound depends only on variety $X/Z$).
\end{corollary}
\begin{proof}
Immediate by Lemma \ref{lemma-FT=0log} and Corollary \ref{cor-num DMMP 0log}.
\end{proof}

\bs

Theorem \ref{thrm-finite wlc big klt} gives
the finiteness of projective wlc models for
a klt pair $(X/Z,B)$ of general type.
It is well-known that all wlc models are
crepant birationally isomorphic.
Moreover, if $(Y/Z,B_Y)$ and $(Y'/Z,B_{Y'})$ are
two wlc models of $(X/Z,B)$, then their natural
(identical on rational functions)
birational isomorphism
$Y\dashrightarrow Y'$ can be factored into
projective($/Y,/Y'$) crepant
minimizations $(V,B_V)\to (Y,B_Y),$ $(V',B_{V'})\to (Y',B_{Y'})$ and a small log flop
$(V/Z,B_V)\dashrightarrow (V'/Z,B_{V'})$:
$$
\xymatrix{
 V \ar[d]\ar@{-->}[r] &V'\ar[d]\\
 Y\ar@{-->}[r] & Y'/Z.
}
$$
A minimal model $(V,B_V)$ is maximal for the contraction
order $Y\to Y'$ for  wlc models of $(X/Z,B)$:
$V$ is $\Q$-factorial, projective$/Y$ and blows up
all prime $D$ with $e(D,B)=0$ (cf.
the proof of Theorem \ref{thrm-finite wlc big klt}).
Moreover, $(V/Z,B_V)$ is a slt wlc model of $(X/Z,B)$
exactly when $V/Z$ is projective.
In its turn, each crepant birational
contraction can be factored into
a sequence of extremal divisorial and small contractions and their number
is bounded by the relative Picard number
$\rho(V\cn/X\cn)$ for a
projective$/Z$ minimization $(V\cn,B_{V\cn})$ of
the log canonical model.
Thus the remaining factorization concerns the log flop.
Note also that $V\cn/Z$ is projective but $V/Z$ may be not.
If $\dim_k X=2$, any klt model, in particular, above
minimization is projective$/Z$.
However, if $\dim_k X\ge 3$, $V/Z$ can be nonprojective and
factorization of birational isomorphism
$V\dashrightarrow V\cn$ or $V'$
into elementary (extremal) flops is
out of our grasp.
Thus we consider only projective wlc models and
focus on minimal (slt) models $(Y/Z,B_Y),(Y'/Z,B_{Y'})$.
Then their birational isomorphism $Y\dashrightarrow Y'$ gives a small log flop and
can be factored into elementary flops.
(Note that, for an elementary flop,
the flopping locus can be very nonelementary,
e.g., not an irreducible curve in
dimension $3$.)
The general case, not necessarily of  general type models,
will be considered in Corollary \ref{cor-small klt slt wlc models in elementary log flops}
below.
Here we have a little bit stronger result.

It is natural to consider the factorization problem for
flops of projective $0$-log pairs $(X/Z,B)$.
However to control polarization in $\Supp B$ we
need to consider big $B$ and thus FT varieties by Lemma \ref{lemma-FT=0log}.
Note that a $\Q$-factorialization of an FT variety is
again an FT variety, and, for any rational $1$-contraction
$X\dashrightarrow Y/Z$ with projective $Y/Z$, $Y/Z$ is FT \cite[Lemma 2.8]{prsh}.

\bs

\begin{corollary}\label{cor-Qfact FT flops}
Let $Y,Y'/Z$ be two $\Q$-factorial FT varieties.
Then any small birational isomorphism $Y\dashrightarrow Y'$
can be factored into elementary small birational
(flips, antiflips, or flops) transformations$/Z$.
The number of such transformations is bounded
(the bound depends only on $Y/Z$).
\end{corollary}
\begin{proof}
The factorization is a general fact.
For instance, we can use the $D$-MMP, where $D$ is the birational
transform of any polarization from $Y'$.
Each transformation will be an elementary $D$-flip.
The number of flips is bounded by the number of
rational $1$-contractions with projective $Y_{\rP}/Z$
according to Corollary \ref{cor-FT finite contractions}.
\end{proof}

\begin{corollary}\label{cor-slt flops}
Let $(Y/Z,B_Y),(Y'/Z,B_{Y'})$ be two
klt slt wlc models of
general type.
Then any (small) log flop $(Y/Z,B_Y)\dashrightarrow(Y'/Z,B_{Y'})$
can be factored into elementary log flops and
the number of such flops is bounded
(the bound depends only on $(Y/Z,B_Y)$).
\end{corollary}
\begin{proof}
Immediate by Corollary \ref{cor-Qfact FT flops} for
FT variety $Y/Y\cn$.
\end{proof}
\bs

We conclude the section with results that
shed a light on terminology:
{\em a mobile cone $\Mob(X/Z)$
as a polarizations cone.}

\begin{proposition}\label{prop-small iso mobile cones}
Let $X/Z$ be a variety with a reduced divisor $S$ and
$g\colon X\dashrightarrow Y/Z$ be a small isomorphism.
Then $g$ induces
an isomorphism $\fD_S(X/Z)=\fD_S(Y/Z)$
that preserves cones $\eff_S(X/Z)=\eff_S(Y/Z),\mobile_S(X/Z)=\mobile_S(Y/Z)$.
If the  isomorphism
is a rational $1$-contraction of a class $\fP$
in $\eff_S(X/Z)$,
then under the above identification:
$\fP=\amp_S(Y/Z),\ofP=\nef_S(Y/Z)$.
For FT $X/Y$, the last subcone $=\samp_S(Y/Z)$.
\end{proposition}
\begin{proof}
Uses the birational invariance of mob+exc decomposition.
Of course, all definitions work and the result holds for general $X/Z$ but
cones are not necessarily closed nor polyhedral.
\end{proof}

\begin{proposition}\label{prop-FT mobile cones}
Let $X/Z$ be an FT variety, and $\rP$ be
the classes in $\Mob(X/Z)$
with rational $1$-contractions
$X\dashrightarrow Y_{\rP}/Z$.
Then $\Mob(X/Z)=\coprod_{\rP}\Amp(Y_{\rP}/Z)$ in $\N^1(X/Z)$, where
$\Amp(Y_{\rP}/Z)$ is the ample cone.
Furthermore, $\rP=\Amp(Y_{\rP}/Z)$, $\overline{\rP}=\sAmp(Y_{\rP}/Z)=\Nef(Y_{\rP}/Z)$
in $\N^1(X/Z)$.
\end{proposition}

\begin{proof}
Transform the decomposition into
the classes $\rP$ of $\Mob(X/Y)$
into their polarization cones: $\rP=\Amp(Y_{\rP}/Z)$.
The equation is given by the pull back $g^*[H]$
of the numerical classes
of ample divisors $H$ on $Y_{\rP}/Z$ by
$ g\colon X\dashrightarrow Y_{\rP}$.
\end{proof}

\section{Geography generalities}\label{section5-general geo}

The points of polyhedron $\fN_S$ fills up the
the cube $\fB_S$.

\begin{proposition}\label{prop-monotone N}
$\fN_S$ satisfies the following:
\begin{enumerate}
\item (monotonicity) if $B\in \fN_S,B'\in\fB_S$ and $B'\ge B$,
then $B'\in\fN_S$;

\item (the face of geography)
the linear span of $\fN_S$ is the linear span of
a minimal face $\fF_S$ of $\fB_S$ containing $\fN_S$;
moreover, equations of the span have the form
$b_i=1$,
and are the linear equations for $B\in \fN_S$;

\item  $\dim_\R\fN_S=\dim_\R\fF_S$.

\end{enumerate}
\end{proposition}
\begin{proof}
The monotonicity is immediate by
the inequality for Kodaira dimensions:
$\nd(B')=\kd(B')\ge\kd(B)=\nd(B)$.
However,  the slt LMMP is sufficient.
For this, one can use the stability of wlc models
\cite{orderedterm}.
Indeed, by the compact and convex
property of $\fN_S$,
there exists a maximal $D\in\fN_S$ in the direction
$\overrightarrow{BB'}$, that is, $\fN_S\cap\overrightarrow{BB'}=
(B,D]$.
If $\fB_S\cap\overrightarrow{BB'}=(B,D]$,
the monotonicity
holds in the direction.
Otherwise, let $(Y/Z,D\Lg_Y)$ be a slt wlc model of $(X/Z,D)$.
Then by \cite[Corollary~9 and Addendum~5]{orderedterm}
after finitely many log flops of this model, it
will become a wlc model such that, for each
$$
D'\in\fB_S\cap\overrightarrow{BB'}\setminus(B,D]
$$
near $D$,
$(Y/Z,D'\Lg_Y)$ has a Mori log fibration
(cf. the proof of Theorem~\ref{mainthrm-geography}).
Then, since $D'>D$,
for a generic curve $C/Z$ of the fibration,
$$
0>(C,K_Y+D'\Lg_Y)\ge (C,K_Y+D\Lg_Y)\ge 0,
$$
a contradiction.

The other two properties follow from
the monotonicity and definitions.
\end{proof}

\bs

\paragraph{Separatrix}
So, to find the largest lower bound of $\fN_S$, or
the minimum divisors (by the closedness) of $\fN_S$,
is sufficient to determine $\fN_S$.
Such a bound $\fS_S$ will be called a {\em separatrix\/}.
It consists of boundaries $B\in\fN_S$
which {\em separate\/} $\fN_S$:
for any neighborhood $U$ of $B$ in the face $\fF_S$,
there exists a boundary $B'\in U\setminus\fN_S$,
that is, $B'\in U$ and $\nu(B')=-\infty$.
Equivalently,
$B\in\fS_S$ if and only if
$B\in\fN_S$ and
is a limit of boundaries
$B'\in\fF_S$
with $\nu(B')=-\infty$.

\begin{proposition}\label{prop-S=closed rational poly}
$\fS_S$ is a closed rational polyhedron
(usually nonconvex), and
$$
\fS_S=\overline{\partial\fN_S\cap\Int\fF_S}=
\fN_S\cap\overline{\fF_S\setminus\fN_S}.
$$
More precisely, a
bordering facet of $\fN_S$
lies either

in a facet of $\fB_S$ (or of $\fF_S$), or

in $\fS_S$, and $\fS_S$ is
the closure of union of those facets.
\end{proposition}

A facet in the proposition is a facet of geography $\fN_S$. (However, the statement holds also for facets of the polyhedron
$\fN_S$.) A facet $\fF$ of \emph{geography} $\fN_S$ is {\em bordering\/} if
$\fF\subset\partial\fN_S$.

\begin{proof}
Immediate by Theorem~\ref{mainthrm-geography} and
Proposition \ref{prop-monotone N}.
\end{proof}

The separatrix plays an important role
in birational geometry of uniruled  varieties,
because it divides Mori log fibrations
from wlc models (see Corollary~\ref{cor-geoface=cube for general}
below).
To factorize the
birational transformations of
those fibrations into more elementary
standard transformations (links), we will use
a path on $\fS_S$ connecting corresponding
points.
The best one would be a segment on $\fS_S$
connecting the points but
the separatrix is usually not convex.
Thus we can use {\em geodesics\/} on $\fS_S$
according to the Euclidean metric of the ambient
space $\fD_S$ or induced by a projection.
We prefer a natural projection from
the {\em origin\/} $0_S$ of the
face of geography $\fF_S$:
$0_S=\sum b_i S_i$ where $b_i=1$ for
the equations of $\fN_S$
and $=0$ otherwise.

\begin{lemma}\label{lemma-nonempty separatrix}
$\fS_S\not=\emptyset$ if and only if
$0_S\not\in\fN_S$ and $\fN_S\not=\emptyset$, or equivalently,
$\nu(S)\ge 0$ and $\nu(0_S)=-\infty$.
\end{lemma}
\begin{proof}
Immediate by Proposition~\ref{prop-S=closed rational poly} and definition.
\end{proof}
\bs

Let
$$
\pr \colon \fF_S\setminus\{0_S\} \to H
$$
be a projection from $0_S$ onto a hyperplane
$H$ in $\fF_S$ given by the
equation $\sum b_i=1$
(here we discard all $b_i$ corresponding
to the equations of $\fN_S$).
(See Diagram \ref{dia projection} below.)

\begin{figure}[htb]
\begin{overpic}[scale=0.7]{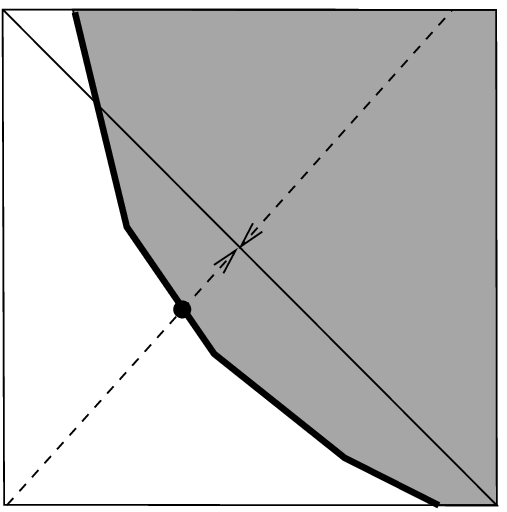}
\put(-7,-7){\small $0_S$}
\put(25,36){\tiny $B$}
\put(68,70){\tiny $D$}
\put(52,52){\tiny $\pr B=\pr D$}
\put(35,85){$\fN_S$}
\put(75,24){\small $H$}
\put(45,10){\small $\fS_S$}
\put(45,50){\tiny $\bullet$}
\put(63,70){\tiny $\bullet$}
\end{overpic}
\setlength{\abovecaptionskip}{20pt}
\caption{}\label{dia projection}
\end{figure}

\bs

By Lemma~\ref{lemma-nonempty separatrix},
$\pr$ is well-defined on $\fS_S$.

\begin{proposition}\label{prop-image of separatrix}
If $\fS_S\not=\emptyset$, the image
$$
\pr \fN_S=\pr \fS_S
$$
is a closed convex rational polyhedron.
Moreover,
$\pr$ is $1$-to-$1$ on $\fS_S$.
\end{proposition}

\begin{proof}
Immediate by Theorem \ref{mainthrm-geography} and Proposition~\ref{prop-monotone N}.
The $1$-to-$1$ property follows from definition.
\end{proof}

\paragraph{Geography of general type}
It is a geography
$\fN_S$ containing
a boundary $B$ such that
the pair $(X/Z,B)$ is of general type.
This assumes that the geography $\fN_S$
is equipped with certain functions, e.g.,
$p(C,B),e(D,B)$ of Section \ref{section2-geography},
$\nu(B),\kappa(B),$ etc.
One of those functions is the numerical  Kodaira dimension $\nu(B)$.
(For other functions, see Section \ref{section2-geography}.)
It behaves almost constant.

\begin{proposition}\label{prop-kd nd are constants}
For each $B\in \fN_S\setminus\pr^{-1}\partial\pr(\fN_S)$ (in particular, $B\in\Int\fN_S$),
$$
\nd(B)=\nd=\max\{\nd(B)|B\in\fB_S\}=\nd(S).
$$
The same holds for $\kd$.
\end{proposition}
\begin{proof}
Immediate by Theorem \ref{mainthrm-geography}
and the following fact on a rather high projective model $V/Z$.
If $D'\geq D$ are nef divisors on $V/Z$, then for any subvariety $W\subseteq V$
and any real numbers $a\ge 0,a'> 0$ such that $a+a'=1$,
$D'$ is big on $W$ if and only if $aD+a'D'$ is big on $W$,
and if $D$ is big on $W$, then $D'$ is big on $W$.
Thus $\nu(aD+a'D')=\nu(D')\geq \nu(D)$.
For $\kd$, one can use semiampleness.
\end{proof}

\begin{corollary}\label{cor-geoface=cube for general}
 Let $\fN_S$ be a geography of general type.
Then $\fF_S=\fB_S,\dim_\R\fN_S=\dim_\R\fB_S$
(the number of prime components of $S$),
in particular, $\fN_S\not=\emptyset$
 and,
for any $B\in\fS_S$,
$0\leq \nu(B)<\nu=\dim_k X/Z$.
\end{corollary}

\begin{proof}
Immediate by Proposition~\ref{prop-monotone N}
and the open property of general type:
if $(X/Z,B), B\in\fB_S$ is of general type,
then $(X/Z,B')$ is also of general type
for any $B'\in \fB_S$ near $B$.
For the Kodaira dimension, this
follows from definition, for numerical one from the slt LMMP.
\end{proof}

If $S'\ge S$ are two reduced b-divisors on $X/Z$,
then there exists a natural inclusion of
geographies:
$$
\fN_S\to\fN_{S'}, B\mapsto B'=B+S'',
$$
and
$\fN_S+S''=\fN_{S'}\cap(\fB_S+S''), \fS_S+S''=\fS_{S'}\cap(\fB_S+S'')$
where $S''=S'-S-(S'-S)_X$, the exceptional
part of $S'-S$.
Indeed, by definition $(Y/Z,B\Lg_Y)$ is a wlc model of $(X/Z,B)$
if and only if this holds for $(X/Z,B')$.
This correspondence preserves $\wlc$.
The transition from $S$ to $S'$ will be called
an \emph{extension}.
It allows us to perturb any wlc model to a much
better model as follows.
For instance,
if $(Y/Z,B\Lg_Y)$ is a wlc model
of $(X/Z,B)$ not slt or not of general type,
e.g., for $B\in \fS_S$, we
can convert it to a polarized wlc model.
Indeed, if $S'$ is sufficiently larger than $S$,
the geography $\fN_{S'}$ is of general type, and
 $B'$ belongs to a face of a country $\fC\subseteq\fN_{S'}$.
Thus $(X/Z,B'')$ where $B''\in\fC$, has a lc model
$(X\cn/Z,B''\cn)$, and
by Lemma \ref{lemma-closed N} a limit of such models gives
a wlc model $(Y/Z,B\Lg_Y)$ of $(X/Z,B)$
(might be different from $Y$  above or/and
not slt; cf. with Convention~\ref{convention}) with a polarization
of $Y/Z=X\cn/Z$.
Moreover, as we will see below, for a good choice
of $S'$,  $Y/Z$ will be $\Q$-factorial.
Of course, the variety $X\cn/Z$ obtained by a perturbation is not unique.
Nonetheless it is useful
(see Corollary \ref{cor-small klt slt wlc models in elementary log flops} and Theorem~\ref{thrm-mori morphism is cte}
below).
\bs

\paragraph{Generality conditions}
In what follows we assume that $S$ is
a reduced divisor, $(X/Z,S)$ is a slt initial pair
of general type, and
components of $S$ generate $\N^1(X/Z)$.
It is easy to satisfy these conditions on a log resolution of $(X/Z,S)$ after an extension of $S$.

Each country in such a geography has
a unique model.

\begin{theorem}\label{thrm-unique wlc for general geo}
Under the generality conditions on $(X/Z,S)$,
let $B\in\Int\fC$ be an internal point
of a country $\fC$
in the geography  $\fN_S$.
Then  $(X/Z,B)$ has  a unique wlc model $(Y/Z,B\Lg_Y)$.
Moreover, the model is
a klt slt wlc model of $(X/Z,B)$, in particular,
$B\Lg_Y=B_Y$,
$Y/Z$ is projective $\Q$-factorial, depending only on $\fC$,
and the model is lc;
components of $S_Y$ generate $\N^1(Y/Z)$.
\end{theorem}

The key instrument to investigate generic geographies
is the following.

\begin{lemma}\label{lemma-rho of wlc/lc}
Under the generality conditions on $(X/Z,B)$,
let $(Y/Z,B_Y)$ be a klt slt wlc model of $(X/Z,B)$
for a boundary $B$ in a subclass $\Int\fB_S\cap\fP$
and $Y\to T=X\cn$ be its lc contraction.
Then $Y/T$ is FT, components of $S_Y$ generate $\N^1(Y/Z)$, and
$$
\rho(Y/T)\le \dim_\R \fB_S-\dim_\R\fP.
$$
\end{lemma}

\begin{proof}
Since the initial model $(X/Z,B)$ is klt slt
and $\Supp B=S$,
by the klt slt LMMP there exists a
klt slt wlc model $(Y/Z,B_Y)$ of $(X/Z,B)$ and
components of $\Supp B_Y=S_Y$ generate $\N^1(Y/Z)$.
Otherwise, we get a Mori log fibration and nothing to prove.
By construction of the lc model,
$(Y/T,B_Y)$ is a $0$-log pair and
by Lemma~\ref{lemma-FT=0log},  $Y/T$ is FT.
The contraction exists by the base freeness \cite{nonvanishing},
because
$B_Y$ is big.
The klt slt properties imply also that the image $\fP_Y$ of $\fP$
 under
the natural linear map
$$
\fD_S\to\fD_{S_Y}, D\mapsto D_Y,
$$
has the same codimension: $\dim_\R\fB_S-\dim_\R\fP=\dim_\R\fB_{S_Y}-\dim_\R\fP_Y$
\cite[Lemma 3.3.2]{choi}.
By the open property in Theorem~\ref{mainthrm-geography},
$\Int\fP=\Int\fB_S\cap\fP$ when $\Int\fB_S\cap\fP\neq\emptyset$.
On the other hand, the generation property of $S_Y$ gives
the surjectivity of the natural linear map
$$
\fD_{S_Y}
\twoheadrightarrow
\N^1(Y/T), D\mapsto [K_Y+D].
$$
By construction and the definition of $\wlc$,
$\fP_Y$ lies in the kernel of the last map.
Thus

$$
\rho(Y/T)=\dim_\R\N^1(Y/T)\le \dim_\R\fD_{S_Y}-\dim_\R \fP_Y=
\dim_\R\fB_S-\dim_\R\fP.
$$
Finally,
note that the constructed model $Y/T$ can be different from
the one in
the statement of lemma.
However, the difference is in a log flop.
By Corollary~\ref{cor-Qfact FT flops},
it is a composition
of elementary small birational transformations which preserve the required properties.
\end{proof}

\begin{proof}[Proof of Theorem \ref{thrm-unique wlc for general geo}]
Let $(Y/Z,B\Lg_Y)$ be a slt wlc model of $(X/Z,B)$.
Since the latter pair is klt initial,
the former pair is klt and $B\Lg_Y=B_Y$ \cite[Lemma 2.4]{isksh}.
By our assumptions and Proposition~\ref{prop-kd nd are constants}, the pair $(Y/Z,B_Y)$
has general type and
a lc model $(T=X\cn,B\cn)$.
Thus by Lemma \ref{lemma-rho of wlc/lc} and Corollary \ref{cor-geoface=cube for general},
 $\rho(Y/T)\le 0$ and $Y=T$, the uniqueness.

The last uniqueness was established only for slt models.
Any other wlc model is a rational $1$-contraction of $Y/T$ and
thus coincide with $Y$ (see the projective case in Corollary~\ref{cor-FT eff decomposition};
in general, cf. the proof of Theorem~\ref{thrm-finite wlc big klt}).
\end{proof}

\bs

Models of facets and ridges for a generic  geography
also behave quite controllable.
In dimension $d=\dim_k X=1$, all facets are bordering and
ridges only cube bordering.
In dimension $d=2$, flopping facets are impossible and
the birational modifications of ridges are only divisorial
(see the classification of ridges below).

\bs

We divide facets into two types according to
their location in $\fN_S$: {\em bordering\/} and {\em internal\/}.
Furthermore, they are divided into two subtypes, respectively:
{\em cube bordering, (Mori) fibering\/} and {\em flopping, divisorial\/}.
Let $\fF$ be a facet of $\fN_S$ under the generality conditions,
$B\in\Int\fF$, and $(Y/Z,B\Lg_Y)$
be a projective wlc model of $(X/Z,B)$.
First, we describe
neighboring classes and properties of
their models (only projective, possibly not $\Q$-factorial)
and then prove them (see the proof of Theorem~\ref{thrm-facet types} below).
\bs

\paragraph{Cube bordering}
The facet $\fF$ is bordering for $\fN_S$ and
its closure $\ofF$ lies in a facet of $\fB_S$, or
satisfies an equation $b_i=0$ or $1$ in $\fN_S$.
In this case,
$(Y/Z,B\Lg_Y)$ can be non klt.
If it is klt, the wlc models can be classified as types below.
We are not  interested
in this because after a perturbation
we have internal boundaries (see the proof of
Corollary~\ref{cor-small klt slt wlc models in elementary log flops} below).

\paragraph{(Mori) Fibering} The facet $\fF\subseteq\fS_S$ is also
bordering for $\fN_S$ but $B\in\Int\fB_S$.
The pair $(Y/Z,B\Lg_Y),$ $B\Lg_Y=B_Y$, is
a klt slt wlc model of $(X/Z,B)$
and $Y$ has a Mori log fibration $Y\rightarrow T= X\cn/Z$
given by the lc contraction.
The base $T$ is $\Q$-factorial.
Such a Mori fibration will be called {\em polarized\/} by
the boundary $B_Y$.
The model $(Y/Z,B_Y)$ and the fibration are unique for $B$ and
both morphisms $Y/Z,Y/T$ depend
only on the facet $\fF$.
Moreover, there exists a unique
(neighboring) country $\fC$ of $\fN_S$
such that $\ofF$ is its facet.
For any $B'\in\fC$, $(Y/Z,B'_Y)$ is a
lc model of $(X/Z,B')$, and $(Y/Z,B_Y)$ is limiting from $\fC$.
On the other side,
if $B'\in \fB_S\setminus\fN_S$ and is
sufficiently close to $B$, the Mori log
fibration $Y\to T/Z$ is
a slt Mori log fibration of
$(X/Z,B')$.
The facet is given by the equation $p(C,B)=0$ in
$\overline{\fC}$
where $C$ is a curve on a rather high model $V/Z$ of $X/Z$,
contracted on $T$ but not on $Y$.
The function is $>0$ on $\fC$ and $<0$ in the points of $\fB_S\setminus\fN_S$ near the points of $\fF$
for the linear extension of $p(C,B')$ from $\fC$.
The functions $e(D,B')$ have
the same signatures on $\fC,\fF$.

\bs
The next two types are possible respectively only in dimension $d\ge 3$ and $d\ge 2$.

\paragraph{Flopping}
The facet is internal for $\fN_S$:
$\Int \fF\subset\Int\fN_S$, and
$\ofF$ is a facet of two (neighboring) countries $\fC_1,\fC_2$ in $\fN_S$.
The pair $(X/Z,B)$ has exactly three wlc models which
are klt of general type and
related by an elementary small flop, extraction or contraction.
Countries $\fC_1,\fC_2$ give exactly
two slt wlc models $(Y_1/Z,B_1),$ $B_1=B_{Y_1}=B\Lg_{Y_1},$
$(Y_2/Z,B_2),$ $B_2=B_{Y_2}=B\Lg_{Y_2}$ of $(X/Z,B)$, where
$(Y_1/Z,B'_{Y_1}),$ $(Y_2/Z,B'_{Y_2})$ are
lc models of $(X/Z,B')$ for $B'\in\fC_1,\fC_2$, respectively.
Those models of $(X/Z,B)$ are limiting
from $\fC_1,\fC_2$, respectively.
The third model $(T/Z,B_T), B_T=B\Lg_T$, of $(X/Z,B)$ is a lc one. 
Both contractions $Y_1,Y_2\to T$ are elementary small.
In particular, $T$ is not $\Q$-factorial and the third model is not slt.
The facet is given by
equations $p(C_1,B)=p(C_2,B)=0$
in $\overline{\fC_1}\cup\overline{\fC_2}$
where $C_1,C_2$ are curves on a rather high model $V/Z$ of $X/Z$
contracted on $T$ but not on $Y_1,Y_2$, respectively;
$p(C_1,B')>0$ on $\fC_1$ and $\ge 0$ on $\fC_2$;
$p(C_2,B')\ge0$ on $\fC_1$ and $> 0$ on $\fC_2$.
The functions $e(D,B')$ have the same signatures on $\fC_1,\fC_2,\fF$.

\paragraph{Divisorial}
The facet is also internal for $\fN_S$: $\Int \fF\subset\Int\fN_S$, and
$\ofF$ is a facet of two (neighboring) countries $\fC_1,\fC_2$ in $\fN_S$.
The pair $(X/Z,B)$ has exactly two wlc models which
are klt of general type and
related by an elementary divisorial extraction or contraction
of a prime divisor $D$.
The country $\fC_1$
gives exactly one slt wlc model $(Y_1/Z,B_1), B_1=B_{Y_1}=B\Lg_{Y_1}$, of $(X/Z,B)$, where
$(Y_1/Z,B'_{Y_1})$ is a lc model of $(X/Z,B')$ for $B'\in\fC_1$.
The model $(Y_1/Z,B_1)$ is limiting from $\fC_1$.
The second model $(T/Z,B_T), B_T=B\Lg_T$, of $(X/Z,B)$ is the lc one,
the contraction $Y_1\to T$ of $D$ is elementary divisorial.
In particular, $T$ is $\Q$-factorial, $(T/Z,B_T)$ is a slt wlc model,
but not a slt wlc model of $(X/Z,B)$.
The model $(T/Z,B_T)$ is limiting from $\fC_2$.
The facet $\fF$ is given by the equation $p(C,B)=0$ in
$\overline{\fC_1}$ where
$C$ is a curve on a rather high model $V/Z$ of $X/Z$
contracted on $T$ but not on $Y_1$;
$p(C,B')=0$ on $\ofC_2$ and $>0$ on $\fC_1$.
Similarly, $\fF$
can also be given by the equation
$e(D,B)=0$ in $\overline{\fC_2}$:
$e(D,B')=0$ on $\ofC_1$ and $>0$ on $\fC_2$.

\begin{theorem}\label{thrm-facet types}
Under the generality conditions on $(X/Z,S)$,
let $\fF$ be a facet of the geography $\fN_S$.
Then it has one of  four above types: cube bordering,
fibering, flopping or divisorial, and behaves
accordingly.
(See Diagram \ref{dia4 facets} below.)
\begin{figure}[htb]
\begin{overpic}[scale=0.65]{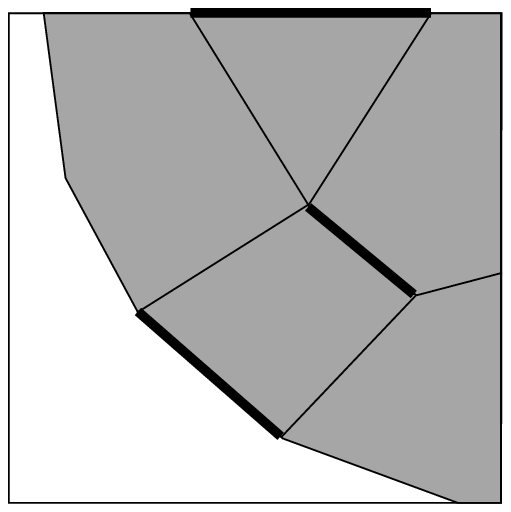}
\put(49,103){\small (1)}
\put(72,53){\small (3,4)}
\put(24,17){\small (2)}
\end{overpic}
\setlength{\abovecaptionskip}{10pt}
\caption{}\label{dia4 facets}
\end{figure}
$$
\begin{array}{lll}
\text{(1) Cube bordering type}& :\fF\subset \partial\fB_S;\\
\text{(2) (Mori) Fibering type}&: \fF\subseteq \fS_S;\\
\text{(3,4) Flopping or divisorial type}& : \Int\fF\subseteq \Int\fN_S.
\end{array}
$$
\end{theorem}

\begin{proof}
A bordering facet $\fF$ of $\fN_S$ has the cube bordering
or fibering type by
definition, Proposition~\ref{prop-S=closed rational poly} and Corollary \ref{cor-geoface=cube for general}.

Cube bordering type: $\fF$ satisfies
an equation $b_i=0$ or $1$
and lies in the corresponding
facet of $\fF_S=\fB_S$ by Corollary~\ref{cor-geoface=cube for general}.

For the boundaries $B\in\Int\fF$ of the other type,
$B$ is \emph{internal}: $B\in\Int\fB_S$.
By the generality assumption,
$(X/Z,B)$ is a  klt slt initial pair of itself and
components of $\Supp B=S$ generate $\N^1(X/Z)$.
In particular $B$ is big.
Thus any wlc model $(Y/Z,B\Lg_Y)$ of $(X/Z,B)$ is klt \cite[Lemma 2.4]{isksh}
with big $B\Lg_Y=B_Y$  and has
a lc contraction $Y\to T=Y\cn$ \cite{nonvanishing}.
If $Y/T$ is projective, it is FT by Lemma~\ref{lemma-FT=0log}.
Such a model $(Y/Z,B_Y)$ exists because $B\in\fN_S$ and
it satisfies Lemma~\ref{lemma-rho of wlc/lc} when it is a slt wlc model of $(X/Z,B)$.
The slt model can be constructed by the slt LMMP
from the above initial model.

Fibering type.
Suppose that $\fF$ is fibering and
$(Y/Z,B_Y)$ is a slt wlc model of $(X/Z,B)$.
By construction and Lemma~\ref{lemma-rho of wlc/lc},
$Y/T$ is a Mori log fibration.
Indeed, it is not isomorphism
and an extremal fibration
by Corollary~\ref{cor-geoface=cube for general}.
The model (in the projective class) and
contraction are unique by Corollary~\ref{cor-FT eff decomposition}.
Since $Y$ is $\Q$-factorial and $Y/T$ is FT,
$T$ is $\Q$-factorial too (cf. \cite[2.13.5]{nonvanishing}).
The boundary $B_Y$ is a polarization of $Y/T$.
Thus if we perturb $B$ to $B'$ in a positive direction,
$B'\in\Int\fN_S$, then
$B'\in\fC$ for a unique neighboring country
$\fC$ of $\fN_S$  and
$(Y/Z,B'_Y)$ is a wlc (actually lc) model of $(X/Z,B')$
with a polarization $K_Y+B'_Y$ over $T$.
Hence the functions $e(D, B')$ have
the same signatures on $\fC,\fF$.
By Lemma~\ref{lemma-closed N}, $(Y/Z,B_Y)$ is limiting from $\fC$, and
by Theorem \ref{mainthrm-geography}, $\ofF$ is a facet of $\fC$.
A perturbation $(Y/Z,B'_Y)$ in a negative direction,
$B'\in \fB_S\setminus\fN_S$, gives
a slt Mori log fibration $Y\to T/Z$ of $(X/Z,B')$.
By construction and Theorem~\ref{thrm-unique wlc for general geo},
a required equation $p(C,B)=0$ can be
given by a curve $C$ on $V/Z$ with
the curve image in a fiber of $Y/T$.

The next two types are
possible only in dimension $d\ge 2$
and internal:
$\Int\fF\subset\Int\fN_S$, by Theorem \ref{mainthrm-geography}
$\ofF$ is a facet of two
(neighboring) countries $\fC_1,\fC_2$
in $\fN_S$,
and $B\in\Int\fN_S$.
(According to the open property in Theorem~\ref{mainthrm-geography},
$\fC_1\cap\fF=\fC_1\cap\fF=\emptyset$.)
By our assumptions and Proposition~\ref{prop-kd nd are constants}, any wlc model $(Y/Z,B_Y)$ of $(X/Z,B)$
has general type
and a lc model $(T/Z,B_T)=(X\cn/Z,B\cn)$
with the lc contraction $Y\to T$.
Let $(Y/Z,B_Y)$ be a slt wlc model of $(X/Z,B)$.
Note that it is not lc: $Y\not= T$.
Otherwise, any small perturbation $B'$ of $B$ in $\fB_S$
is $\wlc B$: $\fF\subset\fC_1,\fC_2$, a contradiction.
Hence $\rho(Y/T)=1$, or equivalently, $Y/T$ is elementary
by Lemma~\ref{lemma-rho of wlc/lc}.

Flopping type.
Now we claim that $\fF$ is of flopping type if
$Y/Z$ is small.
It is possible only in dimension $d\geq 3$.
By Corollary~\ref{cor-FT eff decomposition} the pair $(X/Z,B)$ has exactly three
projective wlc models
$Y_1/Z,Y_2/Z,T/Z$
which are
related by an elementary small flop,
extraction or contraction:
$$
\xymatrix{
Y_1\ar@{-->}[rr]\ar[rd]& &Y_2\ar[ld]\\
  &T/Z.&
}
$$

\noindent (The cone $\Eff(Y/T)$ is a line $\R$ with three
subcones (classes):  $(-\infty,0),0,$
$(0,+\infty)$.)
The boundary $B_Y$ is a polarization of $Y/T$.
If we perturb $B$ to $B'$ in a positive direction,
$B'\in\Int\fN_S$, that is, $K_Y+B'_Y$ becomes
a polarization, then $Y/Z=Y_1/Z$ is limiting
from $\fC_1$ if $B'\in \fC_1$, otherwise
$Y/Z=Y_2/Z$ is  limiting  from $\fC_2$.
So, the boundary $B'\in\fC$ for
one of the two neighboring countries $\fC$
of $\fN_S$ and
$(Y/Z,B'_Y)$ is a wlc (actually lc) model of $(X/Z,B')$
with a polarization $K_Y+B'_Y$.
The same holds for another neighboring country after a log flop of $Y/T$.
By Theorem~\ref{mainthrm-geography},
$\ofF$ is a facet of $\fC_1$ and of $\fC_2$.
The third wlc model of $(X/Z,B)$ is $(T/Z,B_T)$.
In particular, $T$ is not $\Q$-factorial and the model is not slt.
By construction, the functions $e(D,B)$
have the same signatures on $\fC_1,\fC_2,\fF$.
By construction and Theorem~\ref{thrm-unique wlc for general geo},
required equations $p(C_1,B)=p(C_2,B)=0$ can be
given by curves $C_1,C_2$ on $V/Z$ with
the curve image in a fiber of $Y_1,Y_2/T$, respectively.

Divisorial type.
In the remaining
last type case, $Y/T$ is a divisorial
contraction of a unique prime divisor $D$
by the extremal and $\Q$-factorial properties.
Again by Corollary~\ref{cor-FT eff decomposition} the pair $(X/Z,B)$ has exactly two
projective wlc models which are
related by an elementary divisorial
extraction or contraction.
(The cone $\Eff(Y/T)$ is a line $\R$ with two
subcones (classes): $(-\infty,0],(0,+\infty)$.)
The boundary $B_Y$ is a polarization of $Y/T$.
Thus if we perturb $B$ to $B'$ in a positive direction,
$B'\in\Int\fN_S$, that is, $K_Y+B'_Y$
becomes a polarization, say for $B'\in \fC_1$,
then $Y/Z=Y_1/Z$ is limiting  from $\fC_1$.
The second wlc model of $(X/Z,B)$ is $(T/Z,B_T)$.
In particular, $T$ is $\Q$-factorial, $(T/Z,B_T)$ is a slt wlc model,
but not a slt wlc model of $(X/Z,B)$: $e(D,B)=0$.
The model $(T/Z,B_T)$ is limiting from $\fC_2$.
By construction and Theorem~\ref{thrm-unique wlc for general geo},
required equations $e(D,B)=0,p(C,B)=0$ can be
given respectively by the contracted divisor $D$
and by a curve $C$ on $V/Z$ with
the curve image in a fiber of $Y/T$.
\end{proof}

\begin{corollary}\label{cor-small klt slt wlc models in elementary log flops}
Let $(Y/Z,B_Y),(Y'/Z,B_{Y'})$ be two klt slt wlc models.
Then any small log flop $(Y/Z,B_Y)\dashrightarrow(Y'/Z,B_{Y'})$
can be factored into elementary log flops.

It is sufficient to assume that the
transformation is small and
$B'_Y$ is the birational image of $B_Y$.
\end{corollary}
\begin{proof}
We can prove the corollary using
elementary log flops of the $D$-MMP with
$D=H'_Y$ where $H'$ is an effective
polarization on $Y'$ \cite{isksh}.
According to the stability of \cite[Addendum 5]{orderedterm},
the $D$-MMP of $Y/Z$
is equivalent to the klt slt LMMP
of the initial model
$(Y/Z,B_Y+\varepsilon D)$
for some $0<\varepsilon\ll 1$.

We can use a geography instead of the $D$-LMMP
which is more symmetric.
Pick up two effective (generic) polarizations $H,H'$ on $Y,Y'$ respectively.
(In dimension $\ge 2$, one can suppose that $H,H'$
are prime.)
Take a log resolution $(X/Z,S)$
of $(Y/Z,B_Y+H+H'_Y)$
where
$S=\Supp(B_Y)+H+H'+\sum E_i$ and
$E_i$ are the exceptional divisors of $X\to Y$.
Then for an appropriate klt boundary
$B\in\fB_S$,
both models $(Y/Z,B_Y)$ and $(Y'/Z,B_{Y'})$ are
klt slt wlc models of $(X/Z,B)$.
After a perturbation by polarizations $H,H'$,
they will become klt log canonical models
$(Y/Z,B_Y+\varepsilon H),(Y'/Z,B_{Y'}+\varepsilon' H')$
respectively of $(X/Z,D),(X/Z,D')$ where $D,D'$ are
perturbations of $B$ in $\fB_S$.
In other words, there are two classes $\fC,\fC'$ in
the geography $\fN_S$ such that $D\in \fC,D'\in\fC'$ and
$(Y/Z,B_Y),(Y'/Z,B_{Y'})$ are limiting from $\fC,\fC'$ respectively;
$B\in\ofC,\ofC'$.
As above by the stability \cite[Addendum 5]{orderedterm}, the geography is conical
near $B$.
Hence for a perturbation, $\fC,\fC'$ are actually countries and
the segment $[D,D']$ intersects only facets $\fF$,
which are also adjacent to $B$: $B\in\ofF$
is limiting for the facets.
By Theorem~\ref{thrm-facet types}, this gives
a required factorization.
But for this we need to expand $S$ to fulfil
the generality conditions.
We also need to verify that all intersections $[D,D']\cap\fF$
have flopping type.
In other words, if a prime divisor $F$ is not exceptional on $Y$,
then $F$ should not be exceptional on the 
wlc models of all $B'\in [D,D']$
and vice versa.
By construction, $H,H'$ are mobile and big on the
divisors of all those wlc models.
Thus if all divisors in the extension of $S$
are also mobile and big (e.g., ample on $X$;
cf. the perturbation in the proof of
Theorem~\ref{thrm-mori morphism is cte}) on those
divisors, we secure the required nonexceptionality:
each wlc model for $B'\in [D,D']$ corresponds
to an effective perturbation of $B_Y$.
On the other hand,
for $B$ with all $\mult_{E_i}B=1$ or close to $1$ and any exceptional
on $Y$ b-divisor $E$, by the klt assumption for $(Y,B_Y)$
the inequality $e(E,B')>0$ holds near $B$.
Hence $E$ will be exceptional on the wlc models of all $B'\in [D,D']$.

Finally, note that this proof allows us
to relax the klt slt LMMP to the klt slt big one.

A birational transformation $Y\dashrightarrow Y'/Z$ gives a (small) log flop
if and only if the transformation is small and
$B_{Y'}$ is the birational image of $B_Y$ \cite[Proposition 2.4]{3fold-logmod}.
\end{proof}

We divide ridges also into two types according to
their location in $\fN_S$: {\em bordering\/} and {\em internal\/}.
Furthermore, the bordering ones are divided into two
subtypes:
{\em cube bordering\/} and {\em (Mori) fibering\/}.
The internal ridges correspond to general type and
will be also referred to as {\em birational}.
In its turn, fibering and birational
types have three subtypes each.
Essentially they can be divided into types
as {\em only flopping\/}, {\em divisorial and flopping\/}, etc.
But we prefer just a
conventional ordering
2A, 2B, etc: e.g., 2B means fibering type B.

So,  let $\fR$ be a ridge of $\fN_S$ under the generality conditions,
$B\in\Int\fR$, and $(Y/Z,B\Lg_Y)$
be a projective wlc model of $(X/Z,B)$.
First, we describe the neighboring classes and properties of
their models (only projective, possibly not $\Q$-factorial)
and then prove the statements (see the proof of Theorem~\ref{thrm-ridge types} below).

\smallskip

\paragraph{Cube bordering}
The ridge $\fR$ is {\em bordering\/},
that is, $\fR\subset\partial \fN_S$, and
the closure $\ofR$ either lies in a ridge of $\fB_S$, or
in a facet $\fF$ of $\fB_S$.
In the former case, $\fR$ satisfies
two equations $b_i=0$ or $1$, $b_j=0$ or $1$ in $\fN_S$.
In the latter case, $\fF$ satisfies an equation $b_i=0$ or $1$ in
$\fB_S$ and $\fR$ is
a ridge of $\fN_S$ in $\fF$.
In both cases,
$(Y/Z,B\Lg_Y)$ can be non klt.
If it is klt, the wlc models
can be classified as types below.
\smallskip

For the internal types,
the wlc models $(Y/Z,B\Lg_Y)$ of $(X/Z,B), B\in \fR$,
are klt and $B\Lg_Y=B_Y$.
Each of them has the lc contraction
$Y\rightarrow T=X\cn/Z$ and $Y/T$ is FT.
The base $T$ depends only on
the ridge $\fR$.
The pair $(Y/Z,B_Y)$ with projective $Y/Z$ is a slt wlc model of $(X/Z,B)$
if and only if $\rho(Y/T)=2$ and there exist exactly two
extremal contractions of $Y$ over $T$.

\smallskip

\paragraph{Fibering or 2}
The ridge $\fR\subset\fS_S$ is also
bordering for $\fN_S$ but $B\in\Int\fB_S$.
Equivalently, $\Int\fR\subset\fS_S\setminus\partial\fB_S$.
The closure $\ofR$ is a facet of $m+1\ge 2$ (neighboring) facets
$\fF_1,\dots,\fF_{m+1}$ in $\fN_S$, and
a ridge of $m$ countries $\fC_1,\dots,\fC_m$;
$\fC_i$ has (closed) facets $\ofF_i,\ofF_{i+1}$
(see Diagram \ref{dia5 facets2} below).

\begin{figure}[h]
\begin{overpic}[scale=0.65]{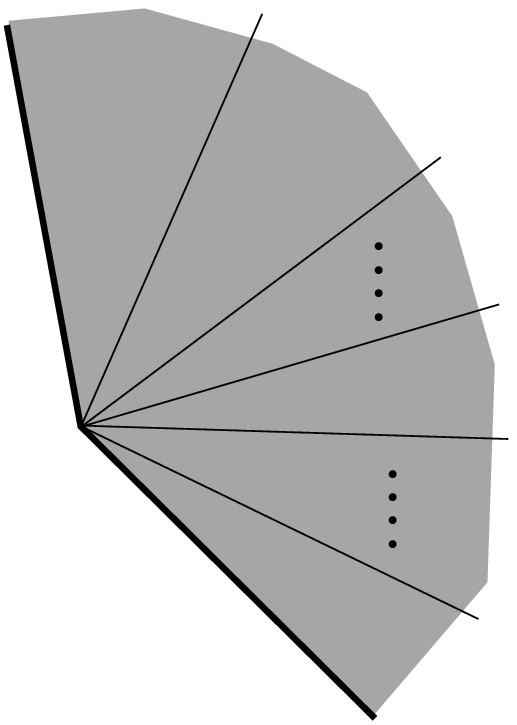}
\put(10,90){\small $\fC_1$}
\put(37,80){\small $\fC_2$}
\put(59,45){\small $\fC_i$}
\put(47,10){\small $\fC_m$}
\put(-5,103){\small $\fF_1$}
\put(32,103){\small $\fF_2$}
\put(63,83){\small $\fF_3$}
\put(72,60){\small $\fF_i$}
\put(74,40){\small $\fF_{i+1}$}
\put(50,-8){\small $\fF_{m+1}$}
\put(68,8){\small $\fF_m$}
\put(2,35){$\fR$}
\put(9.5,39.5){\small $\bullet$}
\end{overpic}
\setlength{\abovecaptionskip}{20pt}
\caption{}\label{dia5 facets2}
\end{figure}
\bs

\noindent Facets $\fF_1,\fF_{m+1}$ are fibering, and
$\fF_2,\dots,\fF_m,m\ge 2,$ are \emph{birational}:
flopping or divisorial.
Let $Y_1/T_1,Y_m/T_{m+1}$ be Mori log fibrations of facets $\fF_1,\fF_{m+1}$,
$T_i, 2 \le i\le m$, be lc models of
the birational facets, and
$Y_i,1\le i\le m$, be lc models of countries $\fC_i$.
If $\fF_m$ is divisorial,
we suppose that $\fF_2$ is divisorial
(after interchanging $\fC_1$ with $\fC_m$).
Each $\fF_i, 3\le i\le m-1$, is flopping, and
corresponding $T_i$ are not $\Q$-factorial.
The varieties $T_1,T_{m+1}$ are
$\Q$-factorial.
The varieties $T_2,T_m$ are $\Q$-factorial if and only if
$\fF_2,\fF_m$ respectively are divisorial.
More precisely, if $\fF_2$
is divisorial, then $Y_2\to Y_1=T_2$
is an elementary divisorial contraction of
a prime divisor $D_2$.
The similar holds for divisorial $\fF_m$
 with contraction $Y_{m-1}\to Y_m=T_m$ of a prime divisor $D_m$
(possibly the same as b-divisor $D_2$).
For $2\le i\le m-1$,
the variety $Y_i/T$ has (only) two elementary contractions $Y_i/T_i,T_{i+1}$.
The same holds for $Y_1,Y_m$
 if $\fF_2,\fF_m$ respectively are flopping.
All contractions $Y_i/T_i,T_{i+1}, 3\le i\le m-2, m\geq 5 $
 are small.
The same holds for $Y_2,Y_{m-1}$
if $\fF_2,\fF_m$ respectively are flopping.

Each $T_i,1 \le i\le m+1$, has a natural contraction to $T$,
$T_i/T$ is extremal, FT and fibered for $2\le i\le m$.
The base $T$ is  not $\Q$-factorial exactly when
$T_1/T$ is small.
The base $T$ is $\Q$-factorial for $\dim T\le 2$,
in particular, in dimension $d\le 3$, and
for types 2A with $m=1$, 2B-C.
In all other cases, $T$ can be non
$\Q$-factorial.

The pair $(X/Z,B)$ has exactly two wlc models
$(Y_1/Z,B_1),$ $B_1=B_{Y_1}$, $(Y_m/Z,B_m)$, $B_m=B_{Y_m}$, $ m\ge 2$,
equipped with a Mori log fibration: $Y_1/T_1,$ $Y_m/T_{m+1}$ over $T$, respectively.
The fibrations $Y_1/T_1,Y_m/T_{m+1}$ are
{\em square} birational,
that is, both contractions $T_1,T_{m+1}/T$ are birational
if and only if
the facets $\fF_1,\fF_{m+1}$ are coplanar.
For $m=1$, 2A is only possible,
the pair $(X/Z,B)$ has exactly one wlc model
$(Y_1/Z,B_1)$, the latter is slt and has two
Mori log fibrations: $Y_1/T_1,T_2$ over $T$,
the base $T$ is $\Q$-factorial and
$T_1,T_2/T$ are Mori log fibrations too,
$\fF_1,\fF_2$ are never coplanar.

The flopping (small) contractions
exist only for $d\ge 3$.
Thus, $m\le 3$ in dimension $d=2$
(there are no fibering ridges in dimension $d=1$).
In dimension $d\ge 3$, the only nonflopping cases are:
2A with $m=1$, 2B with $m=2$ and 2C with $m=3$.

Each pair $(Y_i/Z,B_i)$, $B_i=B_{Y_i}$,
$2\le i\le m-1$, is
a klt slt wlc model of $(X/Z,B)$, limiting from $\fC_i$.
The pairs $(Y_1/Z,B_1),(Y_m/Z,B_m)$ satisfy
exactly the same when respectively $\fF_2,\fF_m$ are flopping, and
they are limiting from $\fC_1,\fC_m$, respectively.
Each pair $(T_i/Z,B_{T_i}), 3\le i\le m-1$,
is a non $\Q$-factorial nonslt
projective wlc model of $(X/Z,B)$, limiting from $\fF_i$.
The pairs $(T_2/Z,B_{T_2}),(T_m/Z,B_{T_m})$ satisfy
exactly the same when respectively $\fF_2,\fF_m$ are flopping.
If $\fF_2$ is divisorial, $(Y_1/Z,B_1)=(T_2/Z,B_{T_2})$ is
an $\Q$-factorial projective wlc model of $(X/Z,B)$
but not a slt wlc model of $(X/Z,B)$.
The similar holds for $\fF_m$.
There are no other projective wlc models of $(X/Z,B)$ than above.

The closures $\ofF_1,\ofF_{m+1}$ are given by
equations $p(C_1,B')=0,p(C_{m+1},B')$
$=0$ in $\ofC_1,\ofC_m$, respectively
where $C_1,C_{m+1}$ are curves on a rather high model $V/Z$ of $X/Z$,
contracted on $T_1,T_{m+1}$ but not on $Y_1,Y_m$, respectively.
Moreover, if $C_1$ is rather general,
then $p(C_1,B')>0$ on
$\fC_i, 1\leq i \leq m,\fF_j, 2\leq j\leq m$, and $\ge 0$
on $\ofF_{m+1}$ whereas $=0$ if and only if the fibration
$Y_1/T_1$ factors through $Y_m/T_{m+1}$: the diagram
$$
\xymatrix{
Y_1\ar[d]\ar@{-->}[r]& Y_m\ar[d]\\
T_1\ar@{-->}[r] &T_{m+1},
}
$$
with the natural birational isomorphism
$Y_1\dashrightarrow Y_m$ and
a (natural) rational contraction $T_1\dashrightarrow T_{m+1}$, is commutative.

The similar holds for general $C_{m+1}$.
The models $Y_1/T_1,Y_m/T_{m+1}$ are square birational if and only if
$p(C_1,\fF_{m+1})=p(C_{m+1},\fF_1)=0$. Moreover, $Y_1/T_1,Y_m/T_{m+1}$ are
{\em  generically isomorphic\/}, that is, they are natural isomorphic over the general points of
$T_1, T_{m+1}$, if and only if the facets $\fF_1,\fF_{m+1}$ are coplanar.

If $\fF_2$ is divisorial, its closure $\ofF_2$ is
given by the equation $e(D_2,B')=0$ in $\ofC_1$:
$e(D_2,B')>0$ on $\fF_1,\fC_1$ and $=0$ on
$\fC_i, 2\leq i \leq m-1,\fF_j, 2\leq j\leq m,\fR$;
$e(D_2,B')=0$ on $\fC_m,\fF_{m+1}$ if 
 $D_2\neq D_m$ as b-divisor or $\fF_m$ is flopping.
If $D_2=D_m$, then $e(D_2,B')>0$ on $\fC_m,\fF_{m+1}$.
The similar holds for $D_m$ if $\fF_m$ is divisorial.

Other equations for the divisorial and
flopping facets were given in Theorem~\ref{thrm-facet types} above.

\medskip

\paragraph{Fibering A or 2A}
(The Sarkisov link of type IV.)
Both $T_1,T_{m+1}\not\cong T, m\ge 1$.
The pair  $(Y_1/Z,B_1)$ is a slt wlc model of $(X/Z,B)$.
The base $T$ is possibly not $\Q$-factorial for $m\ge 2$.
All intermediate facets $\fF_i, 2\le i\le m$, are
flopping.
The only possible other slt wlc models $(Y_i/Z,B_i)$ of $(X/Z,B)$
correspond to the countries $\fC_i, 2\le i\le m$, and
they are log flops of $(Y_1/Z,B_1)$, directed by
a polarization on $Y_i/T$.
More precisely, those flops are composition of elementary ones:
$$
Y_1\dashrightarrow Y_2\dashrightarrow\dots\dashrightarrow
Y_i/T.
$$
The total birational transformation of fibrations:
$$
{\small \xymatrix{
Y_1\ar@{-->}[r]\ar[d]&Y_2\ar@{-->}[r] &\cdots \ar@{-->}[r]&Y_i\ar@{-->}[r]&\cdots\ar@{-->}[r]& Y_{m-1}\ar@{-->}[r] & Y_m\ar[d]\\
T_1\ar[rrrd]^{\not\cong} & & & && &T_{m+1}\ar[llld]_{\not\cong}\\
          &  &&T/Z.& &&
}}
$$

\paragraph{Fibering B or 2B}
(The Sarkisov link of type I and its inverse type III.)
$T_1\cong T$ but $T_{m+1}\not\cong T, m\ge 2$.
The base $T$ is $\Q$-factorial.
The facet $\fF_2$ is divisorial with the elementary divisorial
contraction $Y_2\to T_2\cong Y_1$ and the only possible extremal contraction
of $Y_1/T$ is its Mori fibration.
All other intermediate facets $\fF_i, 3\le i\le m$, are
flopping.
The only slt wlc models of $(X/Z,B)$ are
$(Y_i/Z,B_i), 2\le i\le m$,
limiting from the countries $\fC_i$,
 and they are log flops of $(Y_2/Z,B_2)$,
directed by
a polarization on $Y_i/T$.
More precisely, those flops are composition of elementary ones:
$$
Y_2\dashrightarrow\dots\dashrightarrow
Y_i/T.
$$
The total birational transformation of fibrations:

$$
{\small \xymatrix{
Y_1\ar[dd]&Y_2\ar[l]\ar@{-->}[r]&\cdots \ar@{-->}[r]&Y_i\ar@{-->}[r]&\cdots\ar@{-->}[r] & Y_{m-1}\ar@{-->}[r] &Y_m\ar[d]\\
 & & &&& &T_{m+1}\ar[llld]_{\not\cong}\\
T_1\ar[rrr]^{\cong}   &       &  &T/Z.& &&
}}
$$

\paragraph{Fibering C or 2C}
(The Sarkisov link of type II.)
$T_1\cong T_{m+1}\cong T, m\ge 3$.
The base $T$ is $\Q$-factorial.
Facets $\fF_2,\fF_m$ are divisorial with elementary divisorial
contractions $Y_2\to T_2\cong Y_1,Y_{m-1}\to T_m\cong Y_m$, respectively.
The only possible extremal contractions
of $Y_1,Y_m/T$ are their Mori fibrations.
All other intermediate facets $\fF_i, 3\le i\le m-1$, are
flopping.
The only slt wlc models of $(X/Z,B)$ are
$(Y_i/Z,B_i), 2\le i\le m-1$,
which are limiting from the countries $\fC_i$,
 and they are log flops of $(Y_2/Z,B_2)$, directed by
a polarization on
$Y_i/T$.
More precisely, those flops are composition of elementary ones:
$$
Y_2\dashrightarrow\dots\dashrightarrow
Y_i/T.
$$

The total birational transformation of fibrations:
$$
\xymatrix{
Y_1\ar[d] &Y_2\ar[l]\ar@{-->}[r] &\cdots\ar@{-->}[r]&Y_i\ar@{-->}[r]&\cdots\ar@{-->}[r] &Y_{m-1}\ar[r]& Y_m\ar[d] \\
T_1\ar[rrrd]^{\cong} &  && && & T_{m+1}\ar[llld]_{\cong}\\
 & &&T/Z. && &
}
$$

\medskip

\paragraph{Internal, or birational, or 3}
The ridge $\fR$ is {\em internal\/} if $\Int\fR\subset\Int\fN_S$.
In this case, as for internal facets, all wlc models $(Y/Z,B\Lg_Y)$
are klt with $B\Lg_Y=B_Y$ and of general type with the lc model
$(T/Z,B_T)=(Y\cn/Z,B\cn)$.
The model $T/Z$ of $X/Z$ depends only on $\fR$.
The closure $\ofR$ is a facet of $m\ge 3$ (neighboring) facets
$\fF_1,\dots,\fF_m$ in $\fN_S$, and
a ridge of $m$ countries $\fC_1,\dots,\fC_m$;
$\fC_i$ has (closed) facets $\ofF_i,\ofF_{i+1},m+1=1$
(see Diagrams 6-8 below).
Any two subsequent facets $\fF_i,\fF_{i+1}$ are
strictly in a half-plane.
All facets are birational: flopping or divisorial.
Let
$T_i,Y_i, 1 \le i\le m$, be lc models of
the birational facets $\fF_i$ and of countries $\fC_i$, respectively.
If there is a divisorial facet, we suppose that $\fF_2$ is divisorial
with contraction $Y_2\to T_2\cong Y_1$
after shifting indices $i$.

Any variety
 $Y_i$ is $\Q$-factorial, and
$T_i$ is $\Q$-factorial if and only if
$\fF_i$ is divisorial.
The base $T$ is $\Q$-factorial if and only if
there are two subsequent divisorial facets $\fF_1,\fF_2$
with contractions $Y_2\to T_2\cong Y_1\to T\cong T_1\cong Y_m$.

All varieties $Y_i/T_i,T_{i+1}/T$ are FT extremal or identical,
and $Y_i/T$ are FT with $\rho(Y_i/T)\le 2$.
The pairs $(T_i/Z,B_{T_i}),(T/Z,B_T)$ are the nonslt
wlc models of $(X/Z,B)$.

\bs

\paragraph{Birational A or 3A}
The Weil-Picard number $\dim_\R\N^1(T/T)=2$ all facets $\fF_i$ are flopping, and all
contractions $Y_i/T_i,T_{i+1}, T_i/T$ respectively elementary and extremal small.
The pairs $(Y_i/Z, B_i),$ $B_i=B_{Y_i}$,  are
the only slt wlc models of $(X/Z,B)$.
The models $T_i$ and base $T$ are not $\Q$-factorial.
There exists a sequence of elementary log flops in the facets $\fF_i$:
$$
Y_1\dashrightarrow Y_2\dashrightarrow\dots
\dashrightarrow Y_m\dashrightarrow Y_1/T.
$$
A subsequence $Y_i\dashrightarrow\dots\dashrightarrow Y_j$ of flops
is directed with respect to a polarization on $Y_j/T$ if
the flopping facets $\fF_l$
are strictly in a half-plane.
In general, the configuration of facets and countries is
not symmetric, in particular, facets are not coplanar.
However two facets $\fF_i,\fF_j$ are
coplanar if and only if $T_i\dashrightarrow T_j$ is
a generalized extremal and directed log flop$/T$ with respect to
their polarizations.
Otherwise the span of $\fF_i$ intersects $\fC_j$ and $T_i\dashrightarrow Y_j$ is
a generalized non-extremal and directed log flop$/T$ with respect to
a polarization.
An equation $p(C_i,B)=0$ for the closure $\ofF_i$
in $\ofC_i,\ofC_{i-1}, 1-1=m$,
can be given by a curve  $C_i$
on a rather high model $V/Z$ of $X/Z$.
This allows us to find coplanar facets in terms
of linear extensions for functions $p(C_i,B')$ from
certain countries.
The commuting, centrally symmetric, case is possible
only for $m=4$ with coplanar pairs $\fF_1,\fF_3$ and $\fF_2,\fF_4$;
both flopping loci are disjoint.
See Diagram \ref{dia6} below:

$$
\xymatrix{
&\ar@{-->}[ld]&&\ar@{--}[ll]\ar@{--}[rd]&\\
Y_1\ar@{-->}[r]\ar[rrd]^{\not\cong}&Y_2\ar@{-->}[r] &\cdots\ar@{-->}[r] &Y_{m-1}\ar@{-->}[r]&Y_{m}\ar[lld]_{\not\cong}\\
& &T & &
}
$$
\bs
 \begin{table}[htp]
  \begin{center}
    \begin{tabular}{ l c r }
 \begin{overpic}[scale=0.8]{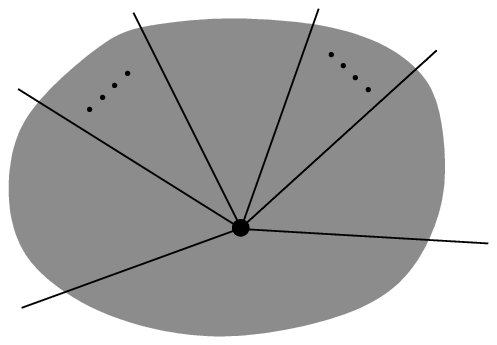}
\put(25,-20){($m\geq 4$)}
\put(-5,-1){\small $\fF_2$}
\put(23,70){\small $\fF_i$}
\put(43,50){\small $\fC_i$}
\put(-8,55){\small $\fF_3$}
\put(100,14){\small $\fF_1=\fF_{m+1}$}
\put(63,73){\small $\fF_{i+1}$}
\put(90,59){\small $\fF_m$}
\put(45,12){\small $\fR$}
\put(60,10){\small $\fC_1$}
\put(6,23){\small $\fC_2$}
\put(70,30){\small $\fC_m$}
\end{overpic} & \qquad\qquad\qquad\qquad
 & \begin{overpic}[scale=0.5]{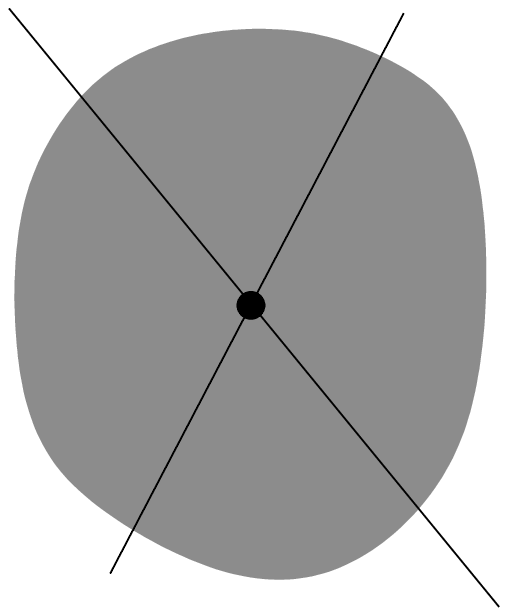}
\put(23,46){\small $\fR$}
\put(85,-8){\small $\fF_1$}
\put(8,-7){\small $\fF_2$}
\put(-8,105){\small $\fF_3$}
\put(67,103){\small $\fF_4$}
\put(40,15){\small $\fC_1$}
\put(2,50){\small $\fC_2$}
\put(30,80){\small $\fC_3$}
\put(64,50){\small $\fC_4$}
\put(-64,-28){($m=4$, flopping loci are disjoint)}
\end{overpic} \\
    \end{tabular}
  \end{center}
  \captionsetup{type=figure}
  \setlength{\abovecaptionskip}{30pt}
  \caption{}\label{dia6}
\end{table}

\bs

\paragraph{Birational B or 3B}
The Weil-Picard number $\dim_\R\N^1(T/T)=1$, and
the only divisorial facets are $\fF_2,\fF_m$
with elementary divisorial contractions
$Y_2\to T_2\cong Y_1,Y_{m-1}\to T_m\cong Y_m$ of
a unique prime b-divisor $D_2$.
All other facets $\fF_i$ are flopping with
elementary small  contractions $Y_{i-1},Y_i/T_i, i=1$ with $1-1=m$ or $3\leq i \leq m-1$,
and with extremal divisorial contractions $T_i/T$, $3\leq i\leq m-1$.
The pairs $(Y_i/Z,B_i),$ $B_i=B_{Y_i}$,  $2\le i\le m-1$, are
the only slt wlc models of $(X/Z,B)$.
The models $T_2,T_m$ are $\Q$-factorial.
All other models $T_i$ and the base $T$ are not $\Q$-factorial.
There exists a sequence of  an elementary divisorial extraction, elementary log flops
in the facets $\fF_1,\fF_i, 3\le i\le m-1$, and of an elementary divisorial contraction:
$$
Y_1\leftarrow Y_2\dashrightarrow\dots
\dashrightarrow Y_{m-1}\to Y_m\dashrightarrow Y_1/T.
$$

\begin{table}[bht]
\begin{center}
\begin{tabular}{ccc}
\multicolumn{3}{c}{$\xymatrix{
Y_2\ar[d]\ar@{-->}[r]&Y_3\ar@{-->}[r] &\cdots\ar@{-->}[r] &Y_{m-2}\ar@{-->}[r]&Y_{m-1}\ar[d]\\
Y_1\cong T_2 \ar[rrd]\ar@{<--}[rrrr] &  & & & Y_{m}\cong T_{m}\ar[lld]\\
& &T_1\cong T & &
}$} \\
&  &\\
& &\\
\begin{overpic}[scale=0.9]{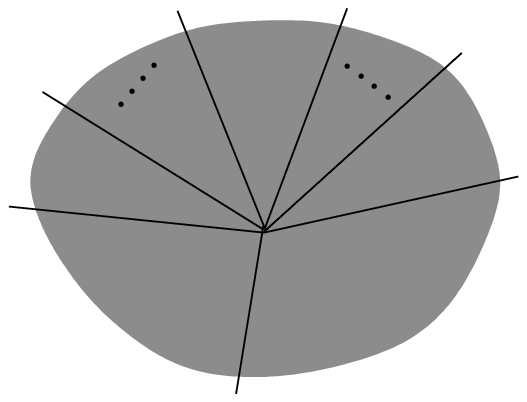}
\put(30,-20){($m\geq 4$)}
\put(-11,38){\small $\fF_2$}
\put(-6,59){\small $\fF_3$}
\put(26,79){\small $\fF_i$}
\put(44,62){\small $\fC_i$}
\put(102,40){\small $\fF_m$}
\put(65,80){\small $\fF_{i+1}$}
\put(90,67){\small $\fF_{m-1}$}
\put(35,-7){\small $\fF_1=\fF_{m+1}$}
\put(52,24){\small $\fR$}
\put(23,17){\small $\fC_1$}
\put(10,43){\small $\fC_2$}
\put(75,47){\small $\fC_{m-1}$}
\put(75,20){\small $\fC_m$}
\put(48,30){\small $\bullet$}
\end{overpic}& \qquad\qquad\qquad & \begin{overpic}[scale=0.55]{bir-ridge-AB4.eps}
\put(-55,-30){($m=4, \text{flopping locus} \cap D_2=\emptyset$)}
\put(23,46){\small $\fR$}
\put(-8,105){\small $\fF_3$}
\put(67,103){\small $\fF_4$}
\put(85,-8){\small $\fF_1$}
\put(8,-4){\small $\fF_2$}
\put(40,15){\small $\fC_1$}
\put(2,50){\small $\fC_2$}
\put(30,80){\small $\fC_3$}
\put(65,50){\small $\fC_4$}
\end{overpic} \\
\end{tabular}
\end{center}
  \captionsetup{type=figure}
  \setlength{\abovecaptionskip}{30pt}
\caption{}\label{dia7}
\end{table}

Any subsequence $Y_i\dashrightarrow\dots\dashrightarrow Y_j$
of flops
is directed with respect to a polarization on $Y_j/T$, and
the flopping facets $\fF_l$ are strictly in a half-plane.
Moreover, the facets $\fF_i,2\le i\le m$, are in a half-plane
and strictly exactly when $\fF_2,\fF_m$ are not coplanar.
In general, the configuration of facets and countries is
not symmetric, in particular, pairs of facets
$\fF_1,\fF_j$ and $\fF_2,\fF_m$ are not coplanar.
However, two facets $\fF_1,\fF_j$ are
coplanar if and only if 
$T_j\rightarrow T_1\cong T$ is a generalized divisorial contraction, that is, for $B'\in\fF_1$,
$K_{T_j}+B'_{T_j}$ is
negative on $T_j/T$, and $K_T+B'_T$ is $\R$-Cartier.
If $\fF_1, F_j, 3\le j\le m-1,$ are not coplanar and the span of $\fF_j$ intersects $\fC_1$ (respectively $\fC_m$),
then $T_j\dashrightarrow T_2$ (respectively $T_j\dashrightarrow T_m$) is a generalized extremal and directed
log flop$/T$ with respect to a polarization.
The facets $\fF_2,\fF_m$ are coplanar if and only if
the flopping locus of  $Y_m\dashrightarrow Y_1$
does not contain $\Center_{Y_1} D_2$.
This is possible only for $m\ge 4$.
Moreover, the commuting, centrally symmetric, case is possible
only for $m=4$ with coplanar pairs $\fF_1,\fF_3$ and $\fF_2,\fF_4$;
the flopping locus is disjoint from $D_2$.
An equation $p(C_i,B)=0$ for the closure
$\ofF_i$ in $\ofC_i,\ofC_{i-1}$,
can be given by a curve $C_i$ on a rather high model $V/Z$ of $X/Z$.
For closures  $\ofF_2,\ofF_m$ respectively in $\ofC_1,\ofC_m$,
the equations $e(D_2,B)=0$ are given by the contractible divisor $D_2$.
This allows us to find coplanar facets in terms
of linear extensions for functions $p(C_i,B'),e(D_2,B')$ from
certain countries.

\bs

\paragraph{Birational C or 3C}
The Weil-Picard number $\dim_\R\N^1(T/T)=0$, $m\ge 4$,
the only divisorial facets are $\fF_1,\fF_2,\fF_3,\fF_m$
with elementary divisorial contractions
$Y_3\to T_3\cong Y_2\to  T_2\cong Y_1\cong T_1\cong T,$
$Y_{m-1}\to T_m\cong Y_m\to  Y_1$ of
two prime divisors $D_3,D_2$, respectively.
The other contractions $T_i/T, 4 \leq i \leq m-1,$ are extremal divisorial with
both exceptional divisors $D_2,D_3$.
The contractions $Y_{i-1}, Y_i/T_i$
corresponding to $\fF_i, 4\leq i\leq m-1,$
are elementary small.
The pairs $(Y_i/Z,B_i),$ $B_i=B_{Y_i}$, $3\le i\le m-1,$ are
the only slt wlc models of $(X/Z,B)$.
The models $T_3,T_m,T\cong T_1\cong T_2$
are $\Q$-factorial.
All other models $T_i, 4\leq i\leq m-1$,
are not $\Q$-factorial.
There exists a sequence of elementary divisorial extractions, elementary log flops
in the facets $\fF_i, 4\le i\le m-1$, and of elementary divisorial contractions:
$$
Y_1\leftarrow Y_2\leftarrow Y_3\dashrightarrow\dots
\dashrightarrow Y_{m-1}\to Y_m\to Y_1/T.
$$
Any subsequence $Y_i\dashrightarrow\dots\dashrightarrow Y_j$
of flops
is directed with respect to a polarization on $Y_j/T$, and
the flopping facets $\fF_l$ are strictly in a half-plane.
Moreover, the facets $\fF_i, 3\leq i \leq  m+1=1$, and
$2\le i\le m,$ respectively, are in a half-plane
and strictly exactly when
$\fF_1,\fF_3$ and $\fF_2,\fF_m$,  respectively, are not coplanar.
In general, the configuration of facets and countries is
not symmetric, in particular, pairs of facets
$\fF_1,\fF_3$ and $\fF_2,\fF_m$ are not coplanar.
However, pairs $\fF_1,\fF_3$ and $\fF_2,\fF_m$ are
coplanar if and only if
respectively $\Center_T D_3\not\subset \Center_TD_2$ and
$\Center_T D_2\not\subset \Center_T D_3$.
Each case or both cases simultaneously can occur for
$m\ge 5$.
However, for $m=4$, both pairs are coplanar or
none of the pairs are coplanar.
The former  is possible only for the commuting centrally symmetric case
with disjoint contractible divisors $D_2,D_3$.
An equation $p(C_i,B)=0$ for the closure
$\ofF_i$ in $\ofC_i,\ofC_{i-1}$,
can be given by a curve
$C_i$ on a rather high model $V/Z$ of $X/Z$.
For closures $\ofF_2,\ofF_m$ respectively in $\ofC_1,\ofC_m$,
the equations $e(D_2,B)=0$ are
given by the contractible divisor $D_2$.
For closures $\ofF_3,\ofF_1$ respectively in $\ofC_1,\ofC_2$
the equations $e(D_3,B)=0$ are
given by the contractible divisor $D_3$.
This allows to find coplanar facets in terms
of linear extensions for functions $e(D_2,B'),e(D_3,B')$ from
certain countries.

$$
\xymatrix{
Y_3\ar@{-->}[r]\ar[d]&Y_4\ar@{-->}[r] &\cdots\ar@{-->}[r] &Y_{m-2}\ar@{-->}[r]&Y_{m-1}\ar[d]\\
Y_2\cong T_3 \ar[rrd] &  & & & Y_m\cong T_m\ar[lld]\\
& &Y_1\cong T_1\cong T_2\cong T & &
}
$$

\begin{table}[th]
\begin{center}
    \begin{tabular}{lll}
\begin{overpic}[scale=0.8]{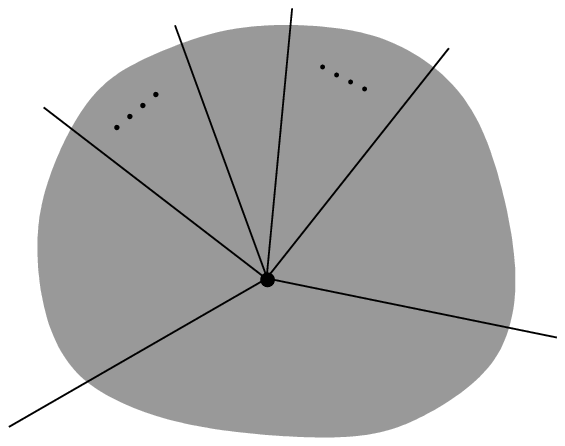}
\put(44,19){\small $\fR$}
\put(-8,-3){\small $\fF_2$}
\put(-2,60){\small $\fF_3$}
\put(24,78){\small $\fF_{i}$}
\put(39,65){\small $\fC_{i}$}
\put(50,80){\small $\fF_{i+1}$}
\put(81,69){\small $\fF_m$}
\put(100,13){\small $\fF_1=\fF_{m+1}$}
\put(55,5){\small $\fC_1$}
\put(10,25){\small $\fC_2$}
\put(77,35){\small $\fC_m$}
\end{overpic}&
\qquad\qquad\qquad\qquad&
\begin{overpic}[scale=0.55]{bir-ridge-AB4.eps}
\put(-45,-30){($m=4$, $D_2, D_3$\;\text{ are disjoint.})}
\put(25,46){\small $\fR$}
\put(-8,105){\small $\fF_3$}
\put(67,103){\small $\fF_4$}
\put(85,-8){\small $\fF_1$}
\put(8,-4){\small $\fF_2$}
\put(40,15){\small $\fC_1$}
\put(3,50){\small $\fC_2$}
\put(30,80){\small $\fC_3$}
\put(65,50){\small $\fC_4$}
\end{overpic}
\end{tabular}\end{center}
  \captionsetup{type=figure}
  \setlength{\abovecaptionskip}{30pt}
\caption{}
\end{table}

\bs
\bs
\bs
\bs
\bs
\bs
\bs
\bs
\bs
\bs
\bs

\begin{theorem}\label{thrm-ridge types}
Under the generality conditions on $(X/Z,S)$,
let $\fR$ be a ridge of the geography $\fN_S$.
Then it is one of the above types:
cube bordering, fibering, or birational, and behaves
accordingly.
(See Diagram \ref{dia9} below.)

\begin{figure}[htp]
\begin{center}
\begin{overpic}[scale=0.65]{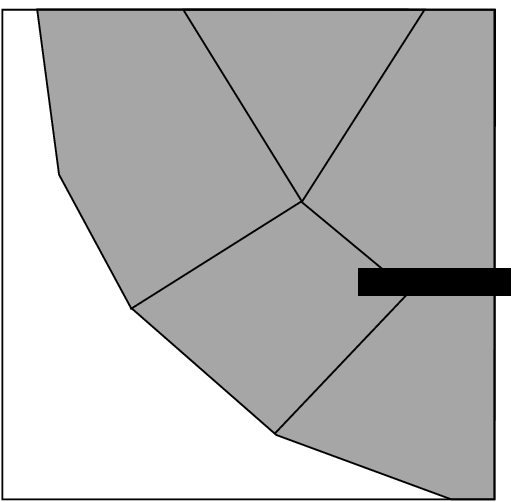}
\put(31,103){\small (1)}
\put(65,58){\small (3)}
\put(15,25){\small (2)}
\put(23,36){\small $\bullet$}
\put(35,96){\small $\bullet$}
\put(57,57){\small $\bullet$}
\end{overpic}
\end{center}
\caption{}\label{dia9}
\end{figure}
$$
\begin{array}{ll}
\text{(1) Cube bordering type}&: \fR\subset \partial\fB_S;\\
\text{(2) (Mori) Fibering type}&:\Int\fR\subset \fS_S\setminus\partial\fB_S;\\
\text{(3) Birational type}&:\Int\fR\subset \Int\fN_S.
\end{array}
$$

\end{theorem}

\begin{proof}
A bordering ridge $\fR$ of $\fN_S$
has the cube bordering
or fibering type by
definition,  Proposition~\ref{prop-S=closed rational poly} and
Corollary \ref{cor-geoface=cube for general}.
Indeed, if $\fR\subset\partial\fB_S$, it is
cube bordering.
Otherwise, $\Int\fR\subset\fS_S\setminus\partial\fB_S$ and
$\fR$ is fibering.

Cube bordering type: $\fR$ satisfies at least one equation
$b_i=0$ or $1$.
If $\fR$ satisfies another equation $b_j=0$ or $1$, it
lies in the ridge of $\fB_S$ given by these equations.
Otherwise, $\fR\subset\partial\fN_S\cap\fF$ for a facet $\fF$ of $\fB_S$
which satisfies the equation $b_i=0$ or $1$.

For the boundaries $B\in\Int\fR$ of the other type ridges,
$B$ is internal: $B\in\Int\fB_S$.
As in the proof of Theorem~\ref{thrm-facet types},
any wlc model $(Y/Z,B\Lg_Y)$ of $(X/Z,B)$ is klt
with big $B\Lg_Y=B_Y$  and has
a lc contraction $Y\to T=X\cn$.
If $Y/T$ is projective, it is FT by Lemma~\ref{lemma-FT=0log} with
$\rho(Y/T)\le 2$ and with at most two extremal contractions$/T$
by Lemma \ref{lemma-rho of wlc/lc}.
As we see from the following, $\rho(Y/T)=2$
 holds and there exist two
extremal contractions of $Y$ over $T$ exactly when
$(Y/Z,B_Y)$ is a slt wlc model of $(X/Z,B)$.
In particular, $d\ge 2$.
Such a model $(Y/Z,B_Y)$ exists because $B\in\fN_S$ and
it satisfies the assumptions of
Lemma~\ref{lemma-rho of wlc/lc}.
The slt model can be constructed by the slt LMMP
from the above initial model.

Fibering type. Suppose that $\fR$ is fibering.
Then by the open property of classes and the last
statement in Theorem~\ref{mainthrm-geography},
there exist required countries $\fC_i$ and facets $\fF_j$.
Most of stated properties follow from Theorem~\ref{thrm-facet types} and
standard facts.

Type 2A. The pair $(Y_1/Z,B_1)$ is a wlc model of $(X/Z,B)$ by Lemma~\ref{lemma-closed N},
limiting from $\fC_1$.
By definition it has a Mori log fibration $Y_1\to T_1/T$ and $T_1\not=T$.
Thus $\rho(Y/T)\ge 2$.
On the other hand, since there exists a slt wlc model of $(X/Z,B)$ over $Y_1$,
this model is $(Y_1/Z,B_1)$ and $\rho(Y_1/T)=2$ by Lemma~\ref{lemma-rho of wlc/lc}.
Similarly, $(Y_m/Z,B_m)$ is a slt wlc model of $(X/Z,B)$ too.
Moreover, all intermediate facets $\fF_j, 2\le j\le m$, are
flopping and the models $(Y_i/Z,B_i), 1\le i\le m,$ are
the only slt wlc models of $(X/Z,B)$.
If $m=1$, $Y_1/T$ has two Mori log fibrations
$Y_1\rightarrow T_1,T_2/T$,
which are the only
extremal contractions$/T$, and by Corollary~\ref{cor-small klt slt wlc models in elementary log flops},
$(Y_1/Z,B_1)$ is the only wlc model of $(X/Z,B)$.
In the last case, the facets $\fF_1,\fF_2$ are noncoplanar by Corollary~\ref{cor-facet=face}.
For $m\ge 2$, another extremal contraction $Y_1\to T_2/T$ is
birational and flopping by Corollary~\ref{cor-small klt slt wlc models in elementary log flops}.
Indeed, if it is divisorial, $Y_1\to T_2=Y_2$ and $Y_2/T$ has
a Mori log fibration $Y_2/T_3=T$, a contradiction.
Similarly, all other intermediate facets are flopping.
Now the description of all wlc models can be obtained from
Corollary~\ref{cor-FT eff decomposition} and Theorem~\ref{thrm-facet types}.

Type 2B. Again the pair $(Y_1/Z,B_1)$ is a wlc model of $(X/Z,B)$,
limiting from $\fC_1$.
By definition, it has a Mori log fibration $Y_1\to T_1 \cong T$.
Thus $\rho(Y_1/T)=1$ and
$Y_1\rightarrow T_1 \cong T$
is the only extremal contraction for $Y_1/T$.
By Theorem~\ref{thrm-facet types}, $\fF_2$ is divisorial.
Otherwise, $Y_1$ has another extremal contraction to $T_2/T$.
In particular, $m\ge 2$.
So, the pair $(Y_2/Z,B_2)$ is a wlc model of $(X/Z,B)$, limiting from $\fC_2$,
with an elementary divisorial contraction $Y_2\to T_2 \cong Y_1$.
Again $\rho(Y_2/T)=2$ and $(Y_2/Z,B_2)$ is a slt model of $(X/Z,B)$.
Moreover, all intermediate facets $\fF_j, 3\le j\le m$, are
flopping and the models $(Y_i/Z,B_i), 2\le i\le m$ are
the only slt wlc models of $(X/Z,B)$.
If $m=2$, $Y_2/T$ has exactly one more contraction$/T$,
the Mori log  fibration $Y_2/T_3\not\cong T$, and by Corollary~\ref{cor-small klt slt wlc models in elementary log flops},
$(Y_2/Z,B_2)$ is the only slt wlc model of $(X/Z,B)$.
For $m\ge 3$, another extremal contraction $Y_2\to T_3/T$ is
birational and flopping by Corollary~\ref{cor-small klt slt wlc models in elementary log flops}.
Indeed, if it is divisorial, $Y_2\to T_3 \cong Y_3$ and $Y_3/T$ has
a Mori log fibration $Y_3/T_4 \cong T$, a contradiction.
Similarly, all other intermediate facets are flopping.
As above the description of all wlc models can be obtained from
Corollary~\ref{cor-FT eff decomposition} and Theorem~\ref{thrm-facet types}.

Type 2C.
As in type 2B, $\fF_2$ is divisorial and
$(Y_2/Z,B_2)$ is a slt
model of $(X/Z,B)$ with
an elementary divisorial contraction $Y_2\to T_2 \cong Y_1$.
Similarly, $\fF_m$ is divisorial and
$(Y_{m-1}/Z,B_{m-1})$ is a slt
model of $(X/Z,B)$ with
an elementary divisorial contraction $Y_{m-1}\to T_m \cong Y_m$ and
the subsequent Mori log fibration $Y_m/T_{m+1}\cong T$.
In particular, $m\ge 3$.
As above all intermediate facets $\fF_j, 3\le j\le m-1$, are
flopping and the models
$(Y_i/Z,B_i),2\le i\le m-1$ are
the only slt wlc models of $(X/Z,B)$.
The description of all wlc models can be obtained from
Corollary~\ref{cor-FT eff decomposition} and Theorem~\ref{thrm-facet types}.

Note that flopping facets only possible for $d\ge 3$.

The properties of equations given by $p(C_1,B'),$ $p(C_{m+1},B')$
and $e(D_2,B'),$ $e(D_m,B')$ follow
from Theorem~\ref{thrm-facet types}.
Moreover, if $C_1$ is a sufficiently general
curve of a Mori log fibration $Y_1/T_1$ (from a rather high model $V/Z$ of $X/Z$), it is not contracted
on all models $Y_i,T_i,2\le i\le m$.
Hence $p(C_1,B')>0$ on
$\fC_i, 1\leq i \leq m,\fF_j, 2\leq j\leq m$, and $\ge 0$
on $\ofF_{m+1}$.
More precisely, $p(C_1,B')=0$ on $\ofF_{m+1}$ if and only if the birational image
of $C_1$ belongs to a fiber of $Y_m/T_{m+1}$, that is,
$Y_1/T_1$ factors through
the last fibration.
The similar holds for a general
curve $C_{m+1}$ on $Y_m/T_{m+1}$.
Thus the fibrations $Y_1/T_1,Y_m/T_{m+1}$ are
square birational to each other if and only if
$T_1,T_{m+1}/T$ are birational (the same general curves in fibers), or
$p(C_1,\fF_{m+1})=p(C_{m+1},\fF_1)=0$ for general $C_1,C_{m+1}$.
Moreover, $Y_1/T_1,Y_m/T_{m+1}$ are generically isomorphic  if and only if the facets
$\fF_1,\fF_{m+1}$ are coplanar. Indeed, if the isomorphism holds then $C_1=C_{m+1}$
and the natural modification $Y_1\dashrightarrow Y_m$ does not touch the general fiber.
Thus $p(C_1,B')$ is linear on $\fN_S$ near $\fR$ and $\fF_1,\fF_{m+1}$ are coplanar.
Conversely, if $\fF_1,\fF_{m+1}$ are coplanar, the convexity of Proposition \ref{prop-p-e}
implies that $p(C_1,\fF_{m+1})=p(C_{m+1}, \fF_1)=0$ and $p(C_1,B')$ is linear on
$\fN_S$ near $\fR$. Thus the modifications
$Y_1\dashrightarrow Y_2\dashrightarrow\cdots \dashrightarrow Y_m$ do not touch $C_1$
and $Y_1/T_1,Y_m/T_{m+1}$ are  generically isomorphic.

Birational type.
Suppose that $\fR$ is birational.
Then by the open property of classes and the last
statement in Theorem~\ref{mainthrm-geography},
there exist required countries $\fC_i$ and facets $\fF_i$.
Most of stated properties follow from Theorem~\ref{thrm-facet types} and
standard facts. We discuss the others below.

In this case as for internal facets, all wlc models $(Y/Z,B\Lg_Y)$
are klt with $B\Lg_Y=B_Y$,
but now of general type with the lc model
$(T/Z,B_T)=(X\cn/Z,B\cn)$ by our assumptions and Proposition~\ref{prop-kd nd are constants}.
The facets $\fF_i,\fF_{i+1}$ are in a half-plane
by convexity of $\fC_i$ and strictly by Corollary~\ref{cor-facet=face}.
In particular, $m\ge 3$.

Type 3A.
Since $\dim_\R\N^1(T/T)=2$ and by Lemma~\ref{lemma-rho of wlc/lc},
for a slt wlc model $(Y/Z,B_Y)$ of $(X/Z,B)$,
the contraction $Y/T$ is small with $\rho(Y/T)=2$ and 
two extremal contractions$/T$ are also elementary small.
As above, for some $i$,
 $Y=Y_i$, limiting from a country $\fC_i$
with two flopping facets $\fF_i,\fF_{i+1}$,
corresponding to two extremal contractions$/T$
by Theorem~\ref{thrm-facet types}.
The rest also follows from the classification of
facets in the theorem and Corollaries~\ref{cor-FT eff decomposition},
\ref{cor-small klt slt wlc models in elementary log flops}.
Note also, that if two facets $\fF_i,\fF_j$ are
coplanar, then $T_i\dashrightarrow T_j$ is a generalized extremal
and directed log flop$/T$ with respect to
a polarization.
Indeed, if the facets are coplanar, the one dimensional subspace of $\R$-Cartier
divisors in $\N^1(T_i/T)$ is generated by $B'\in\fF_i$ or
$B'\in\fF_j$, and respectively
$K_{T_i}+B'_{T_i}$ is ample or antiample$/T$.
The converse follows from definition and
the description of all projective rational $1$-contractions
of $Y/T$.
Similarly, one can treat the case with the span of $\fF_i$ intersecting $\fC_j$.
An equation $p(C_i,B)=0$ for a facet
$\fF_i$ in $\ofC_i,\ofC_{i-1}$
can be given by a curve $C_i$
on a rather high model $V/Z$ of $X/Z$
which is contracted on $T_i$ but not on $Y_i,Y_{i-1}$, respectively.
The commuting case can be treated with
the convexity of Proposition~\ref{prop-p-e}.

Type 3B.
Since $\dim_\R\N^1(T/T)=1$, in particular,
$T$ is not $\Q$-factorial, and by Lemma~\ref{lemma-rho of wlc/lc},
for a slt wlc model $(Y/Z,B_Y)$ of $(X/Z,B)$,
the contraction $Y/T$ is not small with $\rho(Y/T)=2$ but
two extremal contractions$/T$ can be small.
However, applying the $D_2$-LMMP for a prime divisor $D_2$ of $Y$
exceptional on $T$, we obtain a new slt wlc model of $(X/Y,Z)$,
say, $(Y_2/Z,B_2)$, limiting from $\fC_2$ and with
an elementary divisorial contraction $Y_2\to T_2 \cong Y_1$ of $D_2$.
Equivalently, $\fF_2$ is divisorial.
Thus $Y_1$ is $\Q$-factorial and $(Y_1/Z,B_1)$ is
a nonslt model of $(X/Z,B)$, limiting from $\fC_1$ .
The second facet $\fF_1$ of $\fC_1$ is flopping by Theorem~\ref{thrm-facet types}, because
it corresponds to a small extremal contraction $Y_1/T$.
In particular, $D_2$ is the only prime divisor of $Y$
exceptional on $T$.
The log flop in $\fF_1$ gives another $\Q$-factorial
nonslt model of $(X/Z,B)$: $(Y_m/Z,B_m)$, limiting from $\fC_m$.
Again by Theorem~\ref{thrm-facet types}, the facet $\fF_m$ is divisorial with
an elementary divisorial extraction $T_m \cong Y_m\leftarrow Y_{m-1}$ of
$D_2$.
Moreover, $\fF_2,\fF_m$ are
the only divisorial facets, e.g., by Corollary~\ref{cor-small klt slt wlc models in elementary log flops}.
The rest follows from the classification of
facets and Corollaries~\ref{cor-FT eff decomposition}, \ref{cor-small klt slt wlc models in elementary log flops}.
By construction and definition,
$e(D_2,B')>0$ on $\fC_1,\fF_1,\fC_m$ and
$=0$ on the other countries and facets.
Thus by the convexity of $e(D_2,B')$,
the facets $\fF_i,2\le i\le m$, are in a half-plane
and strictly exactly when $\fF_2,\fF_m$ are not coplanar.
The facets $\fF_2,\fF_m$ are coplanar if and only if
the function $e(D_2,B')$ is linear on the half-plane of $\fF_1$,
or equivalently, the flop in $\fF_1$ does not touch $\Center_{Y_1} D_2$
generically.
This is possible only for $m\ge 4$ by Corollary~\ref{cor-facet=face}.
By construction facets $\fF_1,\fF_j$ are
coplanar if and only if
$T_j\to T_1\cong T$ is a generalized divisorial contraction of $D_2$. 
Otherwise, $T_j\dashrightarrow T_2$ or $T_m$ is
a generalized extremal and directed log flop$/T$ with respect to
a polarization \cite[Proof-Explanation of Corollary 1.9]{plflip}.
Note that the contraction $T_j/T$ is not small, but
extremal, and $K_{T_j}+B'_{T_j}$ is negative on $T_j/T$
for $B'\in \fC_1\cup \fF_1\cup \fC_m$ when $K_{T_j}+B_{T_j}$ is
$\R$-Cartier, or equivalently $B'$ is coplanar with $\fF_j$.
In this situation, for $B'\in \fC_1$ (respectively $B'\in \fC_m$),
a generalized log flip with respect to $B'$ or a generalized directed log flop of
$(T_j/T,B)$ with the polarization $K_{T_2}+B'_{T_2}$ (respectively $K_{T_m}+B'_{T_m})$ is
$T_j\dashrightarrow T_2$ (respectively $T_j\dashrightarrow T_m$);
it is neither a divisorial contraction or
a small flip/flop. For $B'\in\fF_1$, a log flip or a log flop of $T_j/T$ is
a generalized divisorial contraction $T_j\to  T_1$.
The commuting case can be treated as above.

Type 3C.
Since $\dim_\R\N^1(T/T)=0$, $T$ is $\Q$-factorial.
As in type 2C, we can construct a slt wlc model of $(X/Z,B)$,
say, $(Y_3/Z,B_3)$ with $\rho(Y_3/T)=2$, limiting from $\fC_3$ and with
an elementary divisorial contraction $Y_3\to T_3=Y_2$ of
a prime divisor $D_3$.
Thus $\fF_3$ is divisorial,
$Y_2$ is $\Q$-factorial and $(Y_2/Z,B_2)$ is
a limiting from $\fC_2$ nonslt model of $(X/Z,B)$.
However, now by the $\Q$-factorial property of $T$,
the second facet $\fF_2$ of $\fC_2$ is also divisorial
with an elementary divisorial contraction $Y_2\to T=T_2=Y_1$
of a prime divisor $D_2$.
Again by Theorem~\ref{thrm-facet types} the facets $\fF_1,\fF_m$
are divisorial with elementary divisorial contractions $Y_{m-1}\to T_m=Y_m\to T=T_1=Y_1$
of divisors $D_2,D_3$, respectively.
In particular, $m\ge 4$.
All other facets are flopping.
The rest follows from the classification of
facets and Corollaries~\ref{cor-FT eff decomposition}, \ref{cor-small klt slt wlc models in elementary log flops}.
By construction and definition,
$e(D_2,B')>0$ on $\fC_1,\fF_1,\fC_m$ and
$0$ on the other countries and facets.
Thus by the convexity of $e(D_2,B')$,
the facets $\fF_i,2\le i\le m$, are in a half-plane
and strictly exactly when $\fF_2,\fF_m$ are not coplanar.
The facets $\fF_2,\fF_m$ are coplanar if and only if
the function $e(D_2,B')$ is linear on the half-plane of $\fF_1$,
or equivalently, the extraction $D_3$ in $\fF_1$ does not touch
$\Center_{Y_1} D_2$ generically,
that is,
$\Center_T D_2\not\subset \Center_T D_3$.
Similarly, the facets $\fF_i,3\leq i \leq m+1=1$, are in a half-plane
and strictly exactly when $\fF_1,\fF_3$ are not coplanar.
The facets $\fF_1,\fF_3$ are coplanar exactly when
$\Center_T D_3\not\subset \Center_T D_2$.
If $D_2,D_3$ are disjoint, then $m=4$ and
the commuting case holds with coplanar pairs
$\fF_1,\fF_3$ and $\fF_2,\fF_4$.
Conversely, if $m=4$, then $Y_3$ has two extremal
divisorial contractions$/T$ of $D_2,D_3$.
If $D_2\cap D_3=\emptyset$, the commuting case holds
with the coplanar pairs.
Otherwise, the extremal property implies that
$\Center_T D_2= \Center_T D_3$ and both pairs
$\fF_1,\fF_3$ and $\fF_2,\fF_4$ are not coplanar.
\end{proof}

\bs

\paragraph{Central model}
Let $\fR$ be a fibering or birational ridge of
a geography under the generality conditions,
$B\in\Int\fR$ be a boundary with a decomposition
$B=F+M$, where $F\ge 0$ and $M\ge 0$ is big and $\R$-mobile
over the lc model $T$ of $(X/Z,B)$.
Such a decomposition always exists and,
in applications, the decomposition mob+exc is typical.
A klt slt wlc model $(Y/Z,B_Y)$ of $(X/Z,B)$ is {\em central\/}
with respect to the decomposition $B=F+M$ if
$(Y/T,F_Y)$ is a weak log Fano \cite[Definition 2.5]{prsh}.
Since $M_Y\equiv -(K_Y+F_Y)/T$,
$Y$ depends only on $\fR$ when
$M_Y=\Mob(B_Y)=\Mob(-K_Y)$ (and $F_Y=0$) for the mob+exc
decomposition of $B_Y$.
Theorem \ref{thrm-ridge types} implies the following properties
of a central model $(Y/Z,B_Y)$:

\begin{enumerate}
\item[(i)] $Y/T$ is a $\Q$-factorial FT contraction with $\rho(Y/T)=2$.

\item[(ii)]
$(Y/Z,B_Y)=(Y_i/Z,B_i)$,
and all other slt wlc models $(Y_j/Z,B_j)$
can be reconstructed from $(Y/Z,B_Y)$ by
elementary log flops in both directions:
$$
Y_1,Y_2 \text{ or } Y_3\dashleftarrow\dots\dashleftarrow Y_i
\dashrightarrow\dots\dashrightarrow Y_{m-1}\text{ or }Y_m;
$$
$Y_i\dashrightarrow\cdots\dashrightarrow Y_3,Y_2,Y_1$ and $Y_i\dashrightarrow \cdots\dashrightarrow Y_{m-1},Y_m$
are sequences of  a log flop and log flips with respect to $F_{Y_i}$ (the modifications are optional),
with at most one log flop for both chains.
In particular, this restores
all other divisorial and fibering contractions over $T$:
$Y_3\to Y_2,Y_2\to Y_1,Y_{m-1}\to Y_m,Y_m\to T_{m+1},$
$Y_1\to T_1$
and all other wlc models of $(X/Z,B)$.
Thus a modification (link) of $Y_1$ (respectively $Y_2,Y_3$) into $Y=Y_{m-1}$ (respectively $Y_m$)
is factored into log antiflips, a log flop (optional) and log flips
with respect to $F_Y$.
The pair $(Y/T,F_Y)$ is a log Fano exactly when
there are no flops; flops are possible only in dimension $d\ge 3$.

\item[(iii)]
A prime divisor $D$ of $Y$ is contractible$/T$ by $-(K_Y+F_Y)$ if and only if
there exists a $\Q$-factorial wlc model $(Y'/Z,B_{Y'})$ of $(X/Z,B)$
where $D$ is exceptional on $Y'$ and with
$a(D,Y',F_{Y'})=1-\mult_DF\ge 0$. In particular, there are no
such divisors if, for any $\Q$-factorial wlc model $(Y'/Z,B_{Y'})$ of $(X/Z,B)$,
$(Y',F_{Y'})$ is terminal
(the last condition can be extended for non $\Q$-Cartier $K_Y'+F_{Y'}$,
assuming that $(Y',F_{Y'})$ is terminal when it is terminal for
a $\Q$-factorialization).

\item[(iv)]
If $(Y/Z,B_Y)$ is $\varepsilon$-lc (in codimension $\geq 2$), then $(Y/Z,F_Y)$ is also $\varepsilon$-lc
 (respectively, in codimension $\geq 2$). For example, this holds if $(X,B)$ is $\varepsilon$-lc
 (respectively, in codimension $\geq 2$, $\varepsilon \le 1$ and
the components of $B$ are nonexceptional on $Y$).
\end{enumerate}

\bs

\begin{corollary}\label{cor-central model exists}
Let $\fR$ be a fibering or birational ridge in $\fN_S$ under
the generality assumptions and $B\in\Int\fR$.
Then for any decomposition $B=F+M$
as above, $(X/Z,B)$ has a central model $(Y/Z,B_Y)$.
For such a decomposition,
the model $(Y/Z,B_Y)$ is unique up to a
log flop of $(Y/T,F_Y)$ over $T$.
Conversely, any klt weak log Fano $(Y/T,F_Y)$ with the property
(i) in Central model above corresponds to
such a ridge.
\end{corollary}
\begin{proof}
To construct a central model one can use the $M_Y$-MMP$/T$ with
any slt wlc model as an initial one.
Since $M_Y$ is $\R$-mobile, the modifications are small.
Note also that $M_Y\equiv -(K_Y+F_Y)$ by definition of $T$: $K_Y+B_Y\equiv 0/T$.
Thus a required model is defined up to a log flop of $(Y/T,F_Y)$.

Conversely, take $Z=T$, $B=F_Y+M$  where $M\ge 0$ and
$\sim_\R -(K_Y+F_Y)$
($\R$-complement).
The generality conditions can be satisfied near $B$ according to
the big and the semiample property of $M$.
Note that, for $S=\Supp B$, $(Y,S)$ is possibly non slt. Thus to apply Theorem \ref{thrm-ridge types},
we can consider its local version near $B$ or we can replace $(Y,S)$ by an appropriate log resolution
 $(X,S)$.
The latter uses the klt property of $(Y,B)$: for some $D\in \fB_S$, $\Supp D_Y=S_Y$ and $(Y/T,B=D_Y)$
is a slt wlc model of $(X/T,D)$ corresponding to a ridge.
\end{proof}

\bs

The following finiteness is
a key to birational factorizations.

\begin{corollary}\label{cor-central models bounded}
Let $\varepsilon$ be a positive real number.
Suppose that \cite[Conjecture 1.1]{prsh} holds
for $\varepsilon$-lc varieties of
dimension $d$.
Then the family of central models $Y/T$ and thus their  links in dimension $d$
for $\varepsilon$-lc slt models $(Y/Z,B_Y)$ are bounded over $k(T)$.
\end{corollary}
\begin{proof}
Immediate by Corollary~\ref{cor-central model exists} and
by Conjecture 1.1 in \cite{prsh}.
For links, see Section~\ref{section6-linkage} below.
\end{proof}

\begin{example}\label{example-d=2 links}
All possible central models with links in dimension $2$ with $Z= \pt,F=0, k=\overline{k}$,
e.g., $k=\mathbb C$, and
terminal singularities i.e., nonsingular, (see links in Mori category in Section~\ref{section6-linkage} below)
are as follows:

2A:
$\mathbb P^1\leftarrow\mathbb F_0=\mathbb P^1\times\mathbb P^1\to\mathbb P^1/T=\pt$,
a nonsingular quadric;

2B:
$\mathbb P^2\leftarrow\mathbb F_1\to \mathbb P^1/T=\pt$, a blowup of the plane $\mathbb P^2$;

2C:
$\mathbb F\leftarrow \Bl \mathbb F\to\mathbb F'/T=C$,
the standard fiber modification, where
$\mathbb F,\mathbb F'$ are $\mathbb P^1$-fibrations over
a nonsingular curve $C$ and $\Bl\mathbb F$ is a blowup of $\mathbb F$ with
blowdown of another curve to $\mathbb F'$.
Note that the last model is $\cong\mathbb P^1/k(C)$.
\end{example}

There are much more such links $(\geq70)$ and central models
 in dimension $3$ in the nonsingular case, and thousands are expected for
terminal singularities.

\section{Linkage}\label{section6-linkage}

\paragraph{Links}  Let $Y_1/T_1,Y_2/T_2$ be two extremal contractions (possibly
birational) of $\Q$-factorial projective$/T_1,T_2$
birationally equivalent varieties $Y_1,Y_2$.
Their {\em link\/} $Y_1\dashrightarrow Y_2$
is a composition (chain) of extremal divisorial
extractions, contractions and small modifications (e.g., flips, flops or
antiflips) over another variety $T$ with contractions $T_1,T_2\to T$:
$$
\xymatrix{
Y_1\ar[d]\ar@{-->}[r]&\cdots\ar@{-->}[r]& Y_2\ar[d]\\
T_1\ar[rd]& &T_2\ar[ld]\\
& T. &
}
$$
Such a link from $Y_1/T_1$ into $Y_2/T_2$ is {\em elementary\/}
if $Y_1,Y_2$ and all intermediate models
$Y/T$ in the link are $\Q$-factorial FT with $\rho(Y/T)\le 2$, and
modifications are elementary in the following order:
(at most 2) divisorial extractions, small modifications,
(at most 2) divisorial contractions;
all modifications are optional.
The sequences of modifications of fibering 2A-C and birational 3A-C ridges
are elementary.

If, in addition,
$Y_1/T_1,Y_2/T_2$ are Mori fibrations and all the
models in the link are terminal,
it is one of Sarkisov links of types I,II,III, or IV \cite{corti1}.
If $(Y_1/T_1,B_1),(Y_m/T_{m+1},B_m)$ are two wlc models
of $(X/Z,B)$ for $B$ in a fibering ridge $\fR$
connected by generalized log flops as in types 2A-C, then
these log flops give an elementary link from $Y_1/T_1$ into $Y_m/T_{m+1}$ by
Theorem~\ref{thrm-ridge types}.
We can take also their inverses but the only new will be for nonsymmetric type 2B.
They are Sarkisov links when all varieties $Y_i$ are terminal.
Those links will be referred to as
 {\em cte\/} ({\em centralized terminal elementary\/})
if  they have  a  terminal \emph{central} model $Y/T$ with $F_T=0$,
in particular,
$-K_Y\equiv B_Y/T$
 is big and $\R$-mobile itself$/T$
on any wlc model $(Y/Z,B_Y)$ of $(X/Z,B)$.
In addition, $Y$ should be terminal
for all $\Q$-factorial wlc models $(Y/Z,B_Y)$ of $(X/Z,B)$.
E.g., this is true if a slt model
$(Y_i/Z,B_i)$ in the link is terminal and
$-K_{Y_i}$ or
$B_i$ is big$/T$ on any divisor of $Y_i$ (cf. the mobility below).
Conversely, the Sarkisov links are cte.
The terminal and elementary properties  follow from definition.
The central property also follows from definition.
More precisely, if the first small modification is a flip or a flop,
$Y_1/T$ is a central model.
Otherwise, a central model $Y_i/T$ is a result of maximal chain
of antiflips; if the next modification is a flop, $Y_i/T$ is a weak Fano fibration,
otherwise, $Y_i/T$ is a Fano fibration.
In total, the link is factored into a divisorial extraction, antiflips,  a flop, flips and a divisorial contraction
(all factors are optional; cf. (ii) in Central model).
Flops are possible only in dimension $d\ge 3$.

The contractions $Y/T$ in a cte  or Sarkisov link satisfy the {\em mobility$/T$\/}:
$-K_Y$ is  $\R$-mobile and big$/T$ in codimension $1$. The latter means that $-K_Y$ is big$/T$
 and big$/T$ on any divisor of $Y$. Thus
$$
K_Y^2\ge 0/T \text{ as a cycle modulo}\equiv \text{ and even } >0/T \text{ for } \dim_k X/T\ge 2.
$$
For $\dim_k X/T=1$, $-K_Y$ is nonzero effective$/T$ modulo $\sim_{\R}$.

\bs
The boundedness of Corollary~\ref{cor-central models bounded}
with $\varepsilon=1$ holds also for the Sarkisov links (cf. \cite[Proposition 6]{bcz}).

\begin{corollary}\label{cor-links/pt bounded in 3}
(Global) Cte links with $T=\pt$ are bounded in dimension $d=3$.
\end{corollary}
\begin{proof}
Immediate by Corollary~\ref{cor-central models bounded}
with $\varepsilon=1$  and \cite{kmmt}.
\end{proof}

\bs

\paragraph{Mori category}
Let $X/Z$ be a variety.
Its Mori category is the category of minimal resulting models of $X/Z$
with birational transformations as morphisms.
So, an object is a projective terminal $\Q$-factorial model $Y/Z$ of $X/Z$ with
nef $K_Y$, a minimal model, or
a projective terminal $\Q$-factorial model $Y/Z$ of $X/Z$ equipped
with a Mori fibration $Y\to T/Z$.
A factorization of morphisms in the case of minimal models is
known (see Corollary~\ref{cor-small klt slt wlc models in elementary log flops} above).
A factorization between Mori fibrations was suggested by
Sarkisov.

\begin{theorem}\label{thrm-mori morphism is cte}
A morphism between Mori fibrations can be
factorized into cte links,
or equivalently,
into Sarkisov links.
\end{theorem}

However,
this is not a factorization of the genuine Sarkisov program:
the Sarkisov's invariants may not decrease.
So, this is not a solution to the Sarkisov program.

\bs

\begin{proof} Follows from a linkage \cite[Example 2.11]{isksh} after a perturbation.
For simplicity, we assume the absolute case: $Z=\pt=\Spec k$, and
omit the notation $/Z$ below.
We suppose also that the dimension $d=\dim Y\ge 2$.
So, let
$$
\xymatrix{
Y\ar[d]_g\ar@{-->}[r] & Y'\ar[d]^{g'} \\
T&T'
}
$$
be a birational transformation for Mori fibrations $Y/T,Y'/T'$.

We can convert each Mori fibration $Y/T$ into a polarized one $(Y/T,B_Y)$ with
an ample boundary $B_Y$. 
Take a sufficiently (very) ample divisor $H$ on $T$ and put
$B_Y=D/N$
where $D$ is a generic reduced irreducible divisor in
a linear system $\linsys{N(-K_Y+g^*H)},$ $N\gg 0$, and $g\colon Y\to T$.
By construction $(Y,B_Y)$ is a terminal (in codimension $\ge 2$)
slt model with the lc contraction
$g$, the Mori fibration.
Similarly, $Y'/T'$ has a polarization $B'_{Y'}$ with $B'_{Y'}=D'/N'$.
We can assume also that
$(Y,B_Y),(Y',B'_{Y'})$
have a common log resolution $X$ on which
(the birational transformations of) $D,D'$ are big,
mobile (even free) and disjoint.

According to \cite[Example 2.11]{isksh} and Corollary~\ref{cor-small klt slt wlc models in elementary log flops} (see also factorization for a birational transformation of wlc models in Section \ref{section4-finiteness}), there exists a factorization
of $Y/T\dashrightarrow Y'/T'$ into links.
The intermediate models $(Y_i,B_i)$
in the factorization are
slt wlc models of $(X,B_i)$ for $B_i\in \fP_i$, a class
(an open interval) in the separatrix $\fS_S,S=D+D'$.
(Since each $B_i$ is $\R$-mobile, we identify $B_i$ with its birational transform
$(B_i)_{Y_i}$
on $Y_i$ by writing the same $B_i$.)
The union $\cup\ofP_i$ gives a path from $B$ to $B'$
in $\fS_S$.
By construction each $Y_i$ has only terminal singularities and
has a Mori fibration $Y_i/T_i$ (possibly after a log flop of $(Y_i,B_i)$)
by Corollary~\ref{cor-geoface=cube for general}.
We can suppose that $Y_1/T_1=Y/T$
for the first interval $\fP_1$
and $Y'/T'=Y_l/T_l$ for the last interval $\fP_l$.
See Diagram \ref{dia10} below.
A link from $Y_i/T_i$ into $Y_{i+1}/T_{i+1}$ corresponds to the point $V_i$ of intersection
$\ofP_i\cap\ofP_{i+1}$.
Note that $(Y_i,V_i),(Y_{i+1},V_i)$ are wlc models of $(X,V_i)$ by Lemma~\ref{lemma-closed N}.
The link may not be elementary, in particular, $T_i$ is not necessarily
the lc model of $(Y_i,B_i)$, and
the Picard ranks for intermediate models $Y_i,Y_{i+1}$ over the lc model of $(X,V_i)$
can be rather high.

\bs
\begin{table}
\begin{center}
\begin{tabular}{ccc}
\begin{overpic}[scale=0.65]{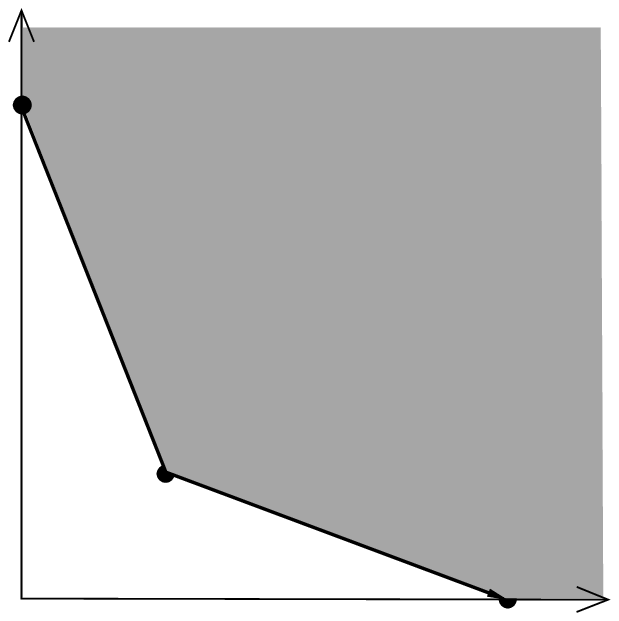}
\put(60,26){\tiny $B'$}
\put(60.5,30){\tiny $\bullet$}
\put(7,59){\small $\fP_1$}
\put(32,35){\small $\fP_2$}
\put(-2,87){\tiny $B$}
\put(22,48){\tiny $V_1$}
\put(21,44.8){\tiny $\bullet$}
\put(82,50){\LARGE $\longrightarrow$}
\put(75,56){\small perturbation}
\end{overpic}&
\qquad
\qquad
\;\;
&
\begin{overpic}[scale=0.65]{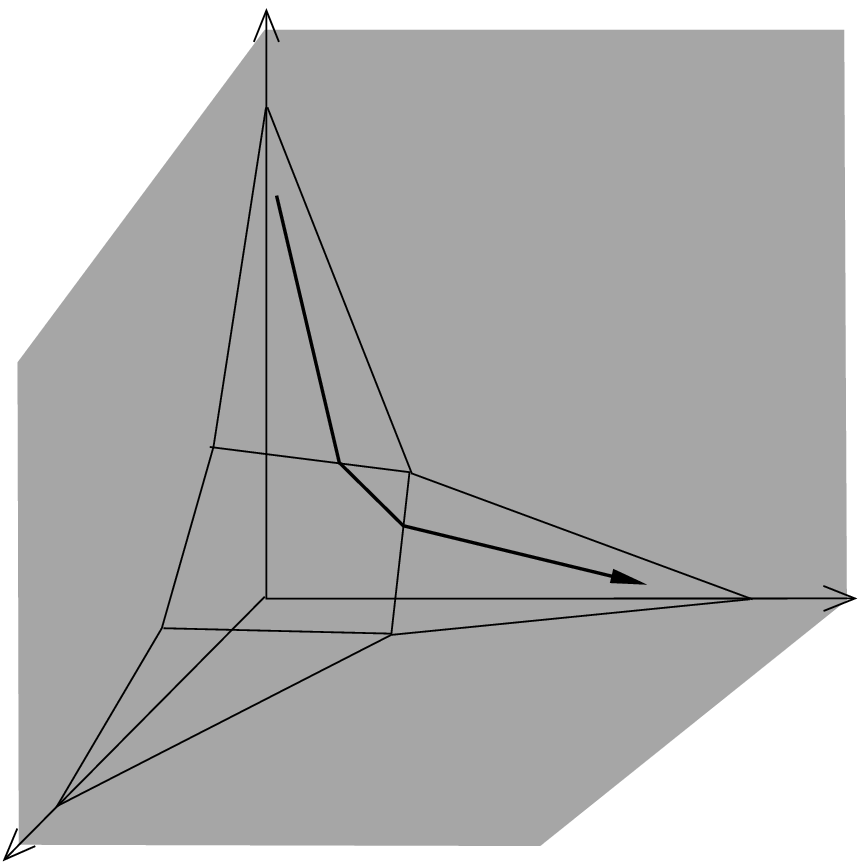}
\put(25,86){\tiny $B$}
\put(84,26){\tiny $B'$}
\put(5,-4){\tiny $\sum S_i$}
\put(35,78){\tiny $B_1$}
\put(40,49){\tiny $V_2$}
\put(48,34){\tiny $V_3$}
\put(73,37){\tiny $B_4=B_l$}
\put(35,39){\small $\fP_2$}
\put(29,55){\small $\fP_1$}
\put(56,31){\small $\fP_3$}
\put(31,76){\tiny $\bullet$}
\put(38.5,45){\tiny $\bullet$}
\put(45.5,38.5){\tiny $\bullet$}
\put(74,31){\tiny $\bullet$}
\end{overpic}
\end{tabular}
\end{center}
\captionsetup{type=figure}
\caption{}\label{dia10}
\end{table}

To attain cte links we use a perturbation and Theorem~\ref{thrm-ridge types}.
To fulfil the generality conditions for $(X,S)$, we add to
$D+D'$ generic (very ample) divisors $S_i$ so that
the conditions hold for  a suitable extension $S:=D+D'+\sum S_i$.
Note that after a perturbation our new subclasses $\fP_i$ and
$V_i$ are embedded into $\fS_S\subset\fN_S$.
(If $S\not=D+D'$, the subclasses are cube
bordering: $\fP_i\subset\partial \fB_S$.)
Now we perturb the whole path
$\cup \ofP_i$ into a path in $\fS_S$
away from $\partial\fB_S$.
We perturb the path in such a way that it goes
only through internal points of
fibering facets and ridges.
(This path can be straight under the projection of Proposition~\ref{prop-image of separatrix}.)
Thus now we suppose that new intervals $\fP_i$ are in fibering facets $\fF_i$ and new
vertexes $V_i$ in fibering ridges $\fR_i$ (we use the same notations
for the perturbed path even
though there is no $1$-to-$1$ correspondence
for intervals and vertexes under the perturbation).
By construction $Y=Y_1/T_1,Y'=Y_l/T_l$ are the Mori log  fibrations of
$\fF_1,\fF_l$, $Y_1,Y_l$ are terminal and $Y_i$
is terminal for
other Mori log fibrations of facets $\fF_i$.
Indeed, for the perturbation $B_1\in \fS_S\setminus\partial\fB_S$ of $B$,
the pair $(Y_1,B_1)$ has the same
contraction $Y/T$ because the contraction is extremal FT.
The same works for $Y'$: for the perturbation $B_l$ in the facet $\fF_l$
close to $B'$, a slt model
$(Y_l,B_l)$ is terminal, and so is $Y_l=Y'$.
The terminal property is also stable
under perturbation.
Each slt wlc model $(W,D''_W)$ of $(X,D'')$ for a (klt) boundary
$D''=bD+b'D', b,b'\in(0,1)$,
is terminal because $(X,D'')$ is terminal and $D''$ is non contractible.
(This model may have cn singularities on the lc model.)
Similarly,
the terminal property holds for intermediate $Y_i$.
Moreover, the same holds for any intermediate $\Q$-factorial model of
a link from $Y_i/T_i$ into $Y_{i+1}/T_{i+1}$ by Theorem~\ref{thrm-facet types}.
Indeed, if we have the lc model
$(W,D''_W)$ for $D''$ in a divisorial facet of $\fN_S$,
with a contracted divisor $E$ of the facet, then for pair $(W,D''_W)$,
the only cn singularity  is $\Center_WE$.
But by our assumptions $\Supp  D''_W$ contains the center, and
$W$ is terminal.
Since $V_i\in \Int\fR_i$ is big and
$\R$-mobile, the link is cte by Corollary~\ref{cor-central model exists}
 and a required factorization is $Y=Y_1\dashrightarrow\dots\dashrightarrow Y_l=Y'$.
\end{proof}

\paragraph{Polarized linkage}
So, we have constructed a chain of links with {\em polarizations} as follows.
There exist
a model $(X,S)$ of $Y$, where
$S=D+D'+\sum S_j$ is a reduced divisor with only
big and mobile prime components, with components of $S_{Y_i}$ generating
each $N^1(Y_i)$, and intervals $\fP_i=(V_i,V_{i+1})\subset\fS_S\setminus\partial\fB_S$
such that:

all boundaries in the interval $\fP_i$
are $\wlc$ equivalent with
terminal klt wlc models $(Y_i,B_i), B_i=B_{Y_i}$, of $(X,B_i)$;
actually $B_i$ on $Y_i$ is the birational transform of $B_i$ from $X$;

$Y_i\to T_i$ is the lc contraction of $(Y_i,B_i)$, that is,
$K_{Y_i}+B_i$ is a pull back of an ample divisor from $T_i$; and

each vertex $V_i$ gives a cte
or Sarkisov link between limiting models
$(Y_i,V_i),(Y_{i+1},V_i)$;
they can be cn and with lc models
$Z_i$ over which
the link goes: for $V_i$,

2A: $T_i\not\cong Z_i\not\cong T_{i+1}$ and each $Y_i,Y_{i+1}$ have 
at most one another elementary
small and numerically trivial contractions$/Z_i$ with respect to $K_{Y_i}+V_i, K_{Y_{i+1}}+V_{i}$;

2B: $T_i\cong Z_i\not\cong T_{i+1}$, an elementary divisorial extraction $Y_i'\to Y_i$ corresponds
to the unique \emph{maximal} center $D_i$ (exceptional divisor) with cn
singularity: $a(D_i,Y_i,V_i)=1$ with center $\Center_{Y_i}D_i$
(a small linear prolongation in direction
$\overrightarrow{V_{i-1}V_i}$
gives a unique noncn center), and
$Y_{i+1}$ has at most one another elementary small and numerically trivial 
 contraction$/Z_i$ with respect to $K_{Y_{i+1}}+V_i$;
the same can hold after interchanging $i$ and $i+1$;

2C: $T_i=Z_i=T_{i+1}$ with elementary divisorial extractions $Y_i'\to Y_i,Y_{i+1}'\to Y_{i+1}$
as in 2B;

each of those links can be centralized (see
Links above).

A limit to boundaries
in $\fS_{D+D'}$ supported in $D+D'$ (the original path) gives similar conditions
with cn singularities of pairs instead of terminal and
semiampleness instead of ampleness on $T_i$ and $Z_i$.

\bs

\paragraph{Rigidity} A Mori fibration
$Y/T$ is {\em rigid\/} if,
for any other Mori fibration $Y'/T'$,
any birational isomorphism  $Y\dashrightarrow Y'$
gives an isomorphism of the fibrations $Y/T,Y'/T'$.
The definition and the definitions below work over any $Z$.
However, if $Z=\pt$ the rigidity is
only possible when $\dim_k T\le 1$.
Thus, for $\dim_k T \ge 2$, it is
reasonable to use a weaker version: $Y/T$ is birationally rigid
(see below) and any Mori fibration $Y'/T$ satisfies the above
property.
There are other weaker versions of the rigidity.
If it holds over the general points of $T,T'$
we say that $Y/T$
is {\em generically\/} rigid.
For example, so do the conic bundles
satisfying Sarkisov's criterion below.
A Mori fibration
$Y/T$ (or over the general point of $T$) is {\em birationally\/} rigid if,
for any other Mori fibration $Y'/T'$,
any birational isomorphism  $Y\dashrightarrow Y'$
gives a birational isomorphism of the fibrations $Y/T$ and $Y'/T'$, that is,
they are square birational under the isomorphism.
The fibration is {\em weakly\/} rigid if, in addition,
any Mori fibration $Y'/T'$ square birational to $Y/T$
is isomorphic to $Y/T$ over the general point of $T$
(possibly under a different birational isomorphism $Y\dashrightarrow Y'$).
Such an isomorphism allows us to treat the birational isomorphism $Y\dashrightarrow Y'$
as a birational automorphism of $Y/T$ and
vice versa.
Del Pezzo fibrations of degree $2$ and $3$ give nontrivial
examples of the last rigidity
(see Corollary~\ref{cor-delpezzo123rigid} below). In general, it is
expected that a Mori fibration $Y/T$ is weakly rigid if its
general fiber is weakly rigid and $Y/T$ is sufficiently twisted,
e.g., the discriminant locus is large or the map of $T$ into
moduli of fibers has a large degree.
Moreover, in higher dimensions,
a rather general Mori fibration $Y/T$ is usually rigid$/T=Z$
if its general fiber is rigid.
We will treat this with more details and applications elsewhere.
Here we only illustrate this in some special cases
usually assuming $Z=\pt$.

\bs
\bs

\begin{corollary}\label{cor-conic rigid threshold}
Suppose \cite[Conjecture 1.1]{prsh} for terminal weak Fano varieties of dimension $d\ge 4$.
In dimension $d$, there exists a rational number $4>\mu_d \geq 1$
({\em the birational rigidity threshold for conic bundles of dimension $d$\/})
satisfying the following.
A Mori conic bundle $X/T$ with $\dim X=d$ is generically  rigid if,
for some real number
$\mu>\mu_d$,
$(T,C/\mu)$ is cn and $\mu K_T+C$ is pseudo-effective
where $C$ is the discriminant locus of the conic bundle $X/T$.

Moreover, for $\mu\geq 4$, we can omit \cite[Conjecture 1.1]{prsh}.
\end{corollary}

For $\mu=4$,
the corollary becomes the Sarkisov theorem \cite{sarkisov} \cite{sarkisov2}.
For any $\mu>0$ with
cn $(T,C/\mu)$, the condition
$\mu K_T+C$ is pseudo-effective is
equivalent to the condition $\mu K_T+C\in \Eff(T)$ by Conjecture \ref{conj-semiample}.
According to Zagorskii \cite{zag} in dimension $3$ and to Sarkisov \cite[Theorem 1.13]{sarkisov2}
in higher dimensions, any conic bundle $X/T$ has a Mori conic bundle model $Y/T$ (square birational to $X/T)$
with cn $(T,C/\mu)$ for any  $\mu \ge 2$.
Actually, in the corollary one can suppose that $X/T$ is a morphism which is
a Mori or standard conic bundle over codimension $1$.
Such a weak standard model with cn $(T,C/\mu)$ can be easily constructed
for any conic bundle and any $\mu\ge 2$.

\begin{proof}
Suppose that $X'/T'$ is another Mori fibration,
$X$ and $X'$ are birational, but not
 square birational.
Note that square birational links preserve
all conditions except for the cn property
  $(T,C/\mu)$.
However, if $K_T+C/\mu$ is pseudo-effective on such a cn model, it
is true on any model with the birational transform of $C$
plus any exceptional divisors (cf. \cite[Lemma 2.2]{sarkisov2}
\cite[Lemma 4]{isk3}).
Now consider a cte
link from $X/T$ into $X'/T'$ which is not square birational, that is,
over some $W$ with $\dim_k X/W\geq 2$,
there exists a fibering contraction $T\to W$.
According to the big and $\R$-mobile property of $-K/W$,
$K^2>0/W$.
Hence, for $\mu\ge 4$,  by the key formula $\frac{4}{\mu}(\mu K_T+C)\le 4K_T+C\equiv-g_*K^2<0/W $
\cite[Lemma 2]{isk3}, where $g\colon X\to T$,
a contradiction.

Now by Corollary~\ref{cor-central models bounded},
the FT varieties $X/W$ and thus $T/W$ are bounded
generically$/W$.
Note that $\rho(T/W)=1$ and they are Mori log
fibrations$/W$ with $B_T=0$.
This implies our improvement for
$$
\mu_d=\max\{\tau_{T/W}<4\}\text{ where } \tau_{T/W}=\max
\{r<4\mid\ rK_{T/W}+C\equiv 0/W\},
$$
for all Weil reduced divisors $C$ on $T$ and assuming $\dim_k T/W\le d-1$.
Actually, $\mu_d\ge \mu_2=7/2=\tau_{\mathbb P^1/\pt}$.
Indeed, the cn property of $(T,C/\mu)$ and
the pseudo-effective property of $\mu K_T+C$ for any $\mu>\mu_d$
implies that of $(T,C/\mu')$ and of $\mu'K_T+C$ for $\mu'=\max \{\mu,4\}\ge 4$.
\end{proof}

\begin{corollary}\label{cor-delpezzo123rigid}
 Let $X/C$ be a $3$-fold Mori fibration of degree $d=1,2,3$ over a curve $C$: $\dim X=3$,
$C$ is a nonsingular curve (not necessarily complete), $X/C$ is a Mori fibration whose
general fibers are Del Pezzo surfaces of degree $d$.
Then $X/C$ is weakly rigid if one of the following holds:
\begin{enumerate}
\item[(i)]
 $X/C$ has sufficiently many singular fibers over a nonsingular completion
 $\overline{C}$ of $C$, that is,
 there exists a natural number $n_d$ such that the number of
 nonsmoothable fibers on the completion of $C$ is $\ge n_d$;
\item[(ii)] the map
$C_{\overline{k}}\dashrightarrow \mathcal{M}_d$, $c\mapsto$ the class of the
smooth geometric fiber $X_c$ of $X_{\overline{k}}/C_{\overline{k}}$,
has a sufficiently large degree, where $\mathcal{M}_d$ is
a coarse moduli space of smooth Del Pezzo surfaces of degree $d$;

\item[(iii)] for each Mori model $Y/\overline{C}$
of $X/C$ over the completion $\overline{C}$,
$$
-K_Y\not\in\Int \Mob(Y)
$$
holds; or

\item[(iv)] for each Mori model $Y/\overline{C}$ of $X/C$,
$$
K_Y^2\not\in\Int \NE(Y)
$$
holds, where $\NE(Y)$ denotes the cone of curves on $Y$.
\end{enumerate}
Moreover, for degree $1$, $X/C$ is generically rigid and,
for degree $2,3$, any birational automorphism of $X/C$
is  a composition of
(Bertini and Geiser) involutions and of isomorphisms over the general point of $C$.
\end{corollary}

The number in (i) is the sum of degrees of points in $\overline{C}$,
the fibers over which are not smoothable.
The condition (ii) is geometrical and, according to the proof below,
(i) can be treated geometrically too.
The condition  in (iii) for $Y/\overline{C}$ was introduced by Grinenko
\cite{grin} and is known as the $K$-{\em condition\/} (cf. \cite[Conjecture 1.2 (ii)]{isk2}).
For a Mori fibration $Y/C$ over a nonsingular projective curve $C$,
it is equivalent to the inclusion:
$$
\Mob(Y)\subseteq \R_+[-K_Y]\oplus g^*\Eff(C)
$$
where $g\colon Y\to C$ and $\R_+=\{r\in\R\mid r\ge 0\}$.
The condition in (iv) for $Y/\overline{C}$ was introduced by Pukhlikov \cite{pukh} and is known as the $K^2$-{\em condition\/}.
The $K^2$-condition implies the $K$-condition.

\begin{proof}
After a completion we can suppose that $X$ and $C=\overline{C}$ are projective.
By the weak rigidity of minimal Del Pezzo surfaces of degree $d\le 3$ \cite[Proposition~3.3]{grin2},
each link over $T=C$ is square birational and preserves all conditions (i-iv).
For $d=1$, the links do not change $X/C$ generically over $C$.
For $d=2,3$, the only nontrivial fiberwise links are (Bertini and Geiser) involutions and
they give the same fibration generically.
But the (global) links$/T=\pt$ are
bounded by Corollary~\ref{cor-links/pt bounded in 3} that contradicts (i) and (ii) with
sufficiently large constants.
The property (ii) is geometrical because the degree is geometrical
and, actually, we need to consider moduli over an algebraically closed field.
The property (i) is geometrical too because the smoothness is geometrical due to Park: the smooth
compactification is unique for $d\le3$ (even $\le 4$) \cite{park}.

By the mobility of cte links,
there are no links  of $Y/C$ under the conditions (iii-iv).
More precisely, the $K^2$-condition implies the $K$-condition and
the latter contradicts the mobility.
Indeed, the link has type 2B or 2C.
Let $Y'$ be its central model and $Y''$ be an anticanonical model of $Y'$
with a small contraction $Y'\to Y''$.
Thus there exists a (natural) small modification $\varphi \colon Y\dashrightarrow Y''$
into a Fano $3$-fold $Y''$.
This gives a decomposition $-K_Y\sim_\Q M+ f F$, that
contradicts the $K$-condition, where $M$ is a (big) $\Q$-mobile divisor,
$F$ is a fiber of $Y/C$, and $f$ is a positive rational number.
Note that $-K_{Y''}=-\varphi(K_Y)$ is an ample divisor on $Y''$ and
$\varphi(F)$ is a Weil divisor on $Y''$.
\end{proof}

\begin{corollary}\label{cor-dp1nonsing rigid=k}
Let $X/C$ be a $3$-fold Mori fibration
of degree $1$ over a nonsingular projective curve $C$.
Then $X/C$ is generically rigid if and only if,
for each Mori model $Y/C$ of $X/C$,
the $K$-condition holds.
\end{corollary}

This is a weak version of the $K$-conjecture \cite[Conjecture~4.4]{grin2} for Del Pezzo fibrations of degree $1$.

\begin{proof}
Immediate by Corollary~\ref{cor-delpezzo123rigid} (iii) and [Grin2, Theorem~4.5].
\end{proof}

It is not difficult to construct many examples of weakly rigid
$3$-fold Mori fibrations satisfying Corollary~\ref{cor-delpezzo123rigid} (i-ii).

Those examples and Corollary~\ref{cor-dp1nonsing rigid=k} allow us to construct
many rigid Mori fibrations of degree $1$ satisfying Corollary~\ref{cor-delpezzo123rigid} (iii).
Examples of such fibrations satisfying Corollary~\ref{cor-delpezzo123rigid} (iv) for degree $1$ or
(iii-iv) for degrees $2,3$ are unknown at present,
but they are expected \cite[Problem~5.9 (3)]{corti2}.

\bs

\paragraph{Questions}
For what degrees $d$,
does the $K$-, or $K^2$-condition hold, or the mobility not hold
for almost all Del Pezzo fibrations with terminal singularities?
Equivalently, for what degrees $d$,
does the condition $-K\in \Int \Mob(X), K^2\in\Int\NE(X)$,
or the mobility is {\em bounded\/}, that is,
the Del Pezzo fibrations $X/C$ of degree $d$ with given property and
only with terminal singularities form
a bounded family?
The same questions are especially important for Mori fibrations.

\bs

For example, the condition $K^2\in \Int \NE(X)$ is unbounded for
cubic Mori fibrations \cite[Lemma~3.6]{bcz}.
This and Corollary~\ref{cor-delpezzo123rigid} give an unbounded family of
terminal Gorestein weakly rigid cubic Mori fibrations$/\mathbb{P}^1$
without the $K^2$-condition.
On the other hand, all nonsingular Mori fibrations with
$K$-condition of degrees $1,2$ and rather general of degree $3$ are
weakly rigid \cite[Theorems~5.6, 6.4 and 7.1]{grin2}.
Some of them does not satisfy the $K^2$-condition.
However, it is unknown whether the condition (iii) of Corollary~\ref{cor-delpezzo123rigid}
holds for some of them.
The proofs of these rigidity results use the maximal singularity method and
some elements of the Sarkisov program \cite{pukh} \cite{grin} \cite{corti2}.
Probably, the latter can be replaced by Polarized linkage.
Maximal singularities play a crucial role in birational geometry.
Anyway,
our approach improves the maximal singularity method.

\begin{theorem}\label{thrm-nonrigid Mori}
Let $X/T$ be a Mori fibration.
The following properties are equivalent:

\begin{enumerate}
\item[(i)] $X/T$ is nonrigid$/T=Z$;

\item[(ii)] there exists a mobile$/T$ linear system $L\subseteq|-dK|$ of degree
(possibly a rational number) $d>0$
which has a single exceptional maximal b-divisor;

\item[(iii)] there exists a mobile$/T$ linear system $L\subseteq|-dK|$ of degree $d>0$
which has an exceptional maximal b-divisor;

\item[(iv)] there exists a big mobile$/T$ linear system $L\subseteq|-dK|$ of degree $d>0$
which has a single exceptional cn b-divisor and is big$/T$ on all its cn b-divisors;

\item[(v)] there exists a big mobile$/T$ linear system $L\subseteq|-dK|$ of degree $d>0$
which has an exceptional cn b-divisor and is big$/T$ on all its cn b-divisors;

\item[(vi)] $X/T$ has a model $Y/T$ such that $Y/T$ is a weak projective terminal $\Q$-factorial Fano
variety with $\rho(Y/T)=2$ and with big$/T$ $-K_Y$ in codimension $1$;

\item[(vii)] $X/T$ has a model $Y/T$ such that $Y/T$ is a terminal Fano variety
with $\dim_\R \N^1(Y/T)=2$;

\item[(viii)] $X/T$ has a model $Y/T$ such that $Y/T$ is a weak projective terminal $\Q$-factorial
Fano variety with $\rho(Y/T)\ge 2$ and with big$/T$ $-K_Y$ in codimension $1$; and

\item[(ix)] $X/T$ has a model $Y/T$ such that $Y/T$ is a terminal Fano variety
with $\dim_\R \N^1(Y/T)\ge 2$.
\end{enumerate}
\end{theorem}

In the statement and below $\subseteq$ means that, for any Weil divisor $D \in L$,
$|D|=|-dK|$, that is, $D\sim -dK$.
In other words, $-dK$ denotes an integral Weil divisor $D$ such that
$D/d\sim_\Q -K$.
(Of course, for the integral numbers $d$,
$dK$ has a canonical choice up to $\sim$.)
Note that $D$ is not unique even if $d$ is integral.
The linear system $L$ {\em has a cn $b$-divisor\/} $E$ if, for general $D\in L$,
the cn threshold $t$ of $(X,D/d)$ is $\le 1$ and $a(E,X,tD/d)=1$.
The cn $b$-divisor is {\em maximal\/} if $t<1$.
Note that $t$ is independent of general $D$ and each
nonexceptional divisor is cn if $t\leq 1$.

\begin{proof}
Suppose that $X/T$ is nonrigid$/T=Z$ and $X/T\dashrightarrow X'/T'$ with $T'/T$
is the first cte link.
The link goes over $T$ and let $Y/T$ be its central model.
The model $Y/T$ satisfies (vi) and (viii), and (iv-v) with $L$,
the birational transformation of $|-NK_Y|$
for some $N\gg 0$.
Taking $L$ with the fixed exceptional
divisor for $Y\dasharrow X$, we get (iii).
The contraction of flopping curves of $Y/T$ in (vi) (the anticanonical model) gives
a model in (vii) and (ix).
Conversely, the $\Q$-factorialization of models in (vii) and (ix)
gives models in (vi) and (viii) respectively.
Thus (vi), (viii) and (vii), (ix) are equivalent respectively.
Note that, (i), (iv) and (vi) are equivalent
(see the proof of Corollary~\ref{cor-central model exists}), and follow
from (ii).
Similarly, (v) and (viii) are equivalent and follow from (iii).
A perturbation of the linear system $|L|$ in (iii) allows us to construct
a linear system in (ii).
Hence (ii) and (iii) are equivalent.
Similarly, (iv) and (v) are equivalent. Finally, (ii) implies (iv) by definition.
\end{proof}

\begin{corollary}\label{cor-rigid Mori}
A Mori fibration $X/T$ is rigid$/T$ if and only if
any mobile$/T$ linear system $L\subseteq |-dK|$ of rational degree $d>0$
has no exceptional maximal b-divisors.
\end{corollary}
\begin{proof}
Immediate by Theorem~\ref{thrm-nonrigid Mori}.
\end{proof}

This implies that rigidity and nonrigidity are constructive conditions
(cf. \cite{chegrin}).
Recall that a {\em Mori-Fano variety\/} is a Mori fibration over
a point $T=\pt$.

\begin{corollary}\label{cor-rigid Mori Fano}
Suppose that the base field k is algebraically closed and
assume \cite[Conjecture 1.1]{prsh} for projective weak Fano varieties of dimension $d\ge 4$.
Then the rigid (respectively nonrigid) Mori-Fano varieties of dimension $d$
form a constructive set in the coarse moduli of Mori-Fano varieties of dimension $d$.
\end{corollary}

\begin{proof}
It is enough to consider nonrigidity.
By Theorem~\ref{thrm-nonrigid Mori}, the nonrigid Mori-Fano varieties are the images
of the terminal Fano varieties $X$ with $\dim_\R \N^1(X)=2$.
More precisely, the latter should have a fixed $\Q$-factorialization and
a fixed extremal ray on it. Then the image is given by a modifications starting from the ray.
The image gives Mori-Fano varieties for a constructive subset and is constructive itself.
\end{proof}

\section{LMMP and semiampleness}\label{section7-LMMP-semi}

This section gives an overview of the LMMP and semiampleness and detailed assumptions
for the statements of the paper.
We use the basic notions
and notations from \cite{kmm} \cite{komo} \cite{isksh}.
Below we recall some of them.

\ms

\paragraph{Model}
Let $X/Z$ be a morphism of varieties.
A variety means a geometrically reduced and irreducible normal algebraic variety over some field $k$,
usually, of characteristic 0 (not necessarily algebraically closed).
The role of the last assumption will be explained below.
According to the following we can suppose that $X/Z$ is proper.
A model of $X/Z$ is another {\em proper\/} morphism $Y/Z$ of varieties such that
$X$ and $Y$ are birationally isomorphic$/Z$.
Moreover, the birational isomorphisms $\chi\colon X\dashrightarrow Y/Z$ is fixed and
called {\em natural\/} for the model $Y/Z$.
In other words, a {\em model\/} is such an isomorphism $X\dashrightarrow Y/Z$.
A {\em natural birational isomorphism\/} of models
$Y_1\dashrightarrow Y_2$ for $X/Z$ is a unique birational isomorphism in
the commutative diagram:
$$
\xymatrix{
& X\ar@{-->}[ld]\ar@{-->}[rd]& \\
Y_1\ar@{-->}[rr]& & Y_2
}
$$
with natural birational isomorphisms $X\dashrightarrow Y_1,Y_2$.
Similarly, if $Y_1\to T_1$, $Y_2 \to T_2$ are two proper morphisms of
models$/Z$, their {\em (rational) morphism\/} is a commutative square:
$$
\xymatrix{
Y_1\ar[d]\ar@{-->}[r]&Y_2 \ar[d]\\
T_1\ar@{-->}[r]&T_2
}
$$
with the natural birational isomorphism $Y_1\dashrightarrow Y_2$;
the {\em natural\/} rational morphism $T_1\dashrightarrow T_2$ is unique.
Two models $Y_1,Y_2/Z$ are considered as the {\em same (equal)\/} if
the natural birational isomorphism is biregular, that is,
they are naturally {\em isomorphic\/}.
Respectively, $Y_1/T_1,Y_2/T_2$ are {\em naturally isomorphic\/}
if $Y_1\dashrightarrow Y_2,T_1\dashrightarrow T_2$ are biregular isomorphisms.
Usually, everybody omit the adjective ``natural".
The symbol $\cong$ or $=$ denotes a natural isomorphism and
$\approx$ any other.
For example, a quadratic transformation $X=\mathbb{P}^2\dashrightarrow Y= \mathbb{P}^2$
gives a model $Y= \mathbb{P}^2$ of $X=\mathbb{P}^2$ which is unnaturally isomorphic to $X=\mathbb{P}^2$
(see Example \ref{example-cremona} above).

There are two subtleties in the model theory.
First, we can consider local and even formal theory, that is,
we can replace $Z$ by a local or formal germ of $Z$ at some
scheme point $z_0\in Z$.
This gives {\em local\/} and {\em formal\/} models over $Z$ whereas
models over a variety $Z$, especially for projective $Z$,
are known as {\em global\/}.
For simplicity, a reader can suppose that all models are global.
However, everything works ditto for local and formal models.
Second, that is more important for us, we consider models of pairs.
This leads to log models.
\ms
\paragraph{Pair}
A pair $(X/Z,D)$ consists of a proper morphism $X/Z$ and a b-divisor $D/Z$.
Usually, we suppose that $D=B$ is a b-boundary.
A b-boundary means a finite b-$\R$-divisor
$B=\sum b_i D_i$ where the sum runs over a finite set of
distinct prime b-divisors $D_i$ with multiplicities $b_i$
in the unit segment $[0,1]$ of real numbers.
The divisors $D_i$ are supposed to be fixed.
For example, all $D_i=S_i$ and
$$
B\in \fB_S=\oplus_{i=1}^m [0,1] S_i \cong [0,1]^m,
$$
where $S=\sum_{i=1}^m S_i$ is a fixed finite reduced b-divisor.
An element $B\in\fB_S$ will be referred to as a
{\em boundary\/} divisor even when it is actually a b-boundary.

The relations $\equiv,\sim_\R$ and the properties such as
ampleness, nefness, $\equiv0$, bigness for divisors on $X/Z$ will always be treated
relatively over $Z$.

\ms

\paragraph{Log model}
Let $(X/Z,B)$ be a pair with a boundary $B\in\fB_S$.
A log pair $(Y/Z,B\Lg_Y)$ is a {\em log model\/} of $(X/Z,B)$ if
$Y/Z$ is a model of $X/Z$, with a natural birational isomorphism
$\chi\colon X\dashrightarrow Y/Z$, and
$B\Lg_Y$ is the {\em log transformation\/} of $B$ on $Y$, that is,
$B\Lg_Y=B_Y+\sum D_i$ where
$B_Y$ is the trace of $B$ on $Y$ and
the summation runs over the distinct $\chi^{-1}$-exceptional prime
divisors $D_i$ on $Y$, not components of $S$.
Since $(Y/Z,B\Lg_Y)$ is a log pair,
for any canonical divisor $K_Y$ of $Y$,
the divisor $K_Y+B\Lg_Y$ is $\R$-Cartier.
Note that the boundary $B\Lg_Y$ is also the trace on $Y$ of
an {\em infinite\/} b-boundary $B\Lg:=B+\sum D_i$ where $D_i$
are the distinct prime divisors exceptional on $X$ and not
prime components of $S$.
Both boundaries $B\Lg$ and $B\Lg_Y$ depend on $S$.
However, if $S$ is a usual divisor (i.e., a b-divisor without exceptional components) on $X$,
the divisors $B\Lg$ and $B\Lg_Y$ are independent of $S$.

\ssk

\paragraph{Initial model}
It is a lc pair with a boundary.
Such a pair $(Y/Z,B\Lg_Y)$ is called an {\em initial model\/} of $(X/Z,B)$ if
$(Y/Z,B\Lg_Y)$ is a log model of $(X/Z,B)$ and
the following inequalities hold:
for each prime $\chi$-exceptional divisor $E$ of $X$,
$a(E,X,B)\le a(E,Y,B\Lg_Y)$ and,
for each prime component $S_i$ of $S$ exceptional on $Y$,
$1-b_i=1-\mult_{S_i}B\le a(E,Y,B\Lg_Y)$.
Using the initial log discrepancy
$\ild(D,X,B)$ defined in Section \ref{section2-geography},
we can combine the the two inequalities into one:
$\ild(E,X,B)\leq a(E,Y,B\Lg_Y)$.

An initial {\em strictly lt (slt)\/} model is a slt pair with a boundary.
Recall that the slt property means that
the variety is $\Q$-factorial and projective  over the base,
the pair is lt.
An initial model $(Y/Z,B\Lg_Y)$ of $(X/Z,B)$ is called
an initial {\em slt\/} model of $(X/Z,B)$ if
$(Y/Z,B\Lg_Y)$ is a slt pair and the above inequalities are strict
for all divisors $E$ of $X$ and all prime components of $S$
exceptional on $Y$.

Existence of an initial model for some pair $(X/Z,B)$ is not obvious and
related to log resolutions.
If a pair $(X,S)$ is log nonsingular and the divisor $S_X=S$ is nonexceptional,
then $(X/Z,B)$ is initial for any boundary $B\in \fB_S$.
For any pair $(X/Z,S)$ and for any of its projective$/Z$ log resolution $(Y,S\Lg_Y)$,
the pair $(Y/Z,B\Lg_Y)$ is a slt initial model of $(X/Z,B)$ for all $B\in\fB_S$.

\ssk

\paragraph{Weakly log canonical model}
It is an initial model $(X/Z,B)$ with a nef divisor $K+B$.
The pair is a {\em log canonical (lc) model\/} if $K+B$ is ample.
An initial model $(Y/Z,B\Lg_Y)$ of $(X/Z,B)$ is called
a {\em weakly log canonical (wlc) model\/} of $(X/Z,B)$ if
$(Y/Z,B\Lg_Y)$ is a wlc model.
It will be a {\em log canonical (lc) model\/} of $(X/Z,B)$ if
$(Y/Z,B\Lg_Y)$ is a lc model.
A lc model of $(X/Z,B)$ is unique (up to a natural isomorphism) if
it exists and will be denoted by $(X\cn/Z,B\cn)$ where
$B\cn:=B\Lg_{X\cn}$.

A wlc model is called a {\em strictly log terminal (slt)\/} wlc model
if it is also slt.
Respectively, a {\em strictly log terminal (slt)\/}
wlc model $(Y/Z,B\Lg_Y)$ of $(X/Z,B)$ is a slt wlc model which
is a slt initial model of $(X/Z,B)$.

\bs

\paragraph{Mori log fibration}
It is an elementary Fano log fibration \cite[Definition~1.6 (v)]{isksh}.
More precisely, it is an initial model $(X/Z,B)$ with
a contraction $X\rightarrow T/Z$ such that:
(a) $\dim_k T < \dim_k X$;
(b) $(X/Z,B)$ has only $\Q$-factorial and lt singularities;
(c) the relative Picard number $\rho(X/T):=\rho(X/Z)-\rho(T/Z) = 1$; and
(d) $-(K+B)$ is relatively ample$/T$.
Such a model without (b-c) (respectively only without (b)) is
known as an (elementary) Fano log fibration.
Note that the contraction $X/T$ is a fixed structure of
the Mori log fibration.
It is an {\em slt} Mori log fibration if, additionally,  $T/Z$ is projective.
A Mori log fibration is \emph{generic} if it is considered over the general point of $T$.

An initial model $(Y/Z,B\Lg_Y)$ of $(X/Z,B)$ with a contraction $Y\rightarrow T/Z$
is called a {\em Mori log fibration of\/} $(X/Z,B)$ if
$(Y/Z,B\Lg_Y)$ is a Mori log fibration with the contraction $Y/T$.
Similarly, one can define a {\em Fano log fibration of\/} $(X/Z,B)$.
The model $(Y/Z,B\Lg_Y)$ is an {\em slt} Mori log fibration of $(X/Z,B)$ if
$(Y/Z,B\Lg_Y)$ is a slt initial model of $(X/Z,B)$ and
$(Y/Z,B\Lg_Y)$ is also a slt Mori log fibration.
Thus, for each slt Mori log fibration of $(X/Z,B)$,
the pair $(Y/Z,B\Lg_Y)$ is a slt Mori log fibration, but not conversely in general
even when $(Y/Z,B\Lg_Y)$ is a Mori log fibration of $(X/Z,B)$.
Similarly, for a wlc model $(Y/Z,B\Lg_Y)$ of $(X/Z,B)$,
a slt wlc model $(Y/Z,B\Lg_Y)$ is not the same
as a slt model of $(X/Z,B)$, in general.

If $B=0$ and $X$ is trm then
a slt Mori log fibration $(X/Z,B)$ is simply
a {\em Mori fibration\/}.
Respectively, if $B=B\Lg_Y=0$ and $Y$ is trm then
a slt Mori log fibration $(Y/Z,0)$ is
a {\em Mori fibration\/} $Y/Z$ or a {\em Mori model\/} {\em of\/} $X/Z$.

\ssk

\paragraph{Resulting model}
It is a wlc model or a Mori log fibration
(or even a Fano log fibration).
A {\em resulting model\/} of a pair $(X/Z,B)$ is
either a wlc model or a Mori log fibration
(or even a Fano log fibration) of $(X/Z,B)$.
According to \cite[2.4.1]{3fold-logmod},
any pair $(X/Z,B)$ cannot have both resulting models simultaneously.
However, a pair can have many wlc models or Mori (Fano) log fibrations.
It is expected that every pair $(X/Z,B)$ with a boundary divisor $B$ has at least one of them, that is,
a resulting model.

If $S$ is a usual divisor, that is, nonexceptional on $X$
then $B$ is a boundary and the above concepts are
the same as in \cite{3fold-logmod} \cite{isksh}.

\paragraph{LMMP}
The Log Minimal Model Program (LMMP) allows us
to construct a resulting model.
The {\em LMMP\/} is a special case of
the $D$-Minimal Model Program ($D$-MMP) \cite[1.1]{isksh}.
An initial model for the LMMP is
an initial model $(X/Z,B)$ and $D= K+B$
(usually, it works with $\equiv$ instead of $=$).
We suppose also that $X/Z$ is projective for
reasons explained (see Characteristic $0$ below).
Then a $D$-minimal model is a projective wlc model and
a nonbirational $D$-contraction $X\to T$ is
a projective elementary Fano log fibration.
The program itself is a chain of birational transformations
of the initial model which are supposed to exist and terminate.
Each transformation is a (generalized) log flip
of a projective negative extremal contraction \cite{3fold-logmod},
but not necessarily elementary (see below).
Intermediate models obtained after each transformation are again
projective initial models.
The last model under the LMMP is either
a projective wlc model or
a projective elementary Fano log fibration of the initial model.
The last model under the LMMP starting from some initial
model of $(X/Z,B)$ is a resulting model of $(X/Z,B)$.

However, the slt LMMP is easier and more common.
An initial model of the {\em slt LMMP\/} is
a slt initial  model $(X/Z,B)$, in particular, $X/Z$ is projective.
Each transformation is either a log flip or a divisorial contraction
which are supposed to be elementary.
Intermediate models obtained after each transformations are again slt initial.
The last model under the LMMP is either
a slt wlc model or a slt Mori log fibration of the initial model.
The last model under the LMMP starting from some slt initial
model of $(X/Z,B)$ is a slt resulting model of $(X/Z,B)$:
either a slt wlc model or
a slt Mori log fibration of $(X/Z,B)$.

An {\em elementary\/} contraction (extraction) or
{\em elementary\/}
small transformation (modification) is respectively
a projective contraction (extraction) $Y\to Y'=T$ or
a small nonregular projective modification $Y\dashrightarrow Y'/T$ such that
$Y,Y'$ are $\Q$-factorial and $\rho(Y/T)=1,\rho(Y'/T)\le 1$.
Actually, for an elementary contraction (extraction),
$\rho(Y/T)=1,\rho(Y'/T)=0$ and,
for an elementary small transformation, $\rho(Y/T)=\rho(Y'/T)=1$.

An {\em extremal\/} contraction or
{\em extremal\/} transformation (modification) is respectively
a contraction (extraction) $Y\to Y'=T$ or a transformation  $Y\dashrightarrow Y'/T$ such that
$\rho(Y/T)=1,\rho(Y'/T)\le 1$.
Actually, for an extremal contraction (extraction),
$\rho(Y/T)=1,\rho(Y'/T)=0$ and,
for an extremal transformation, $\rho(Y/T)=\rho(Y'/T)=1$

Thus each pair $(X/Z,B)$ has a resulting model by the LMMP
and a slt resulting model of $(X/Z,B)$ by the slt LMMP,
of course, if such a program exists for $(X/Z,B)$,
including existence of an initial model
(see Characteristic $0$ below).
For such a construction, a {\em weak\/} termination
is enough, that is, some sequence of birational
transformations terminates, gives a resulting model.
For example, the (slt) LMMP is established over any perfect field
in dimension $2$: for any algebraic surface over such a field,
the LMMP exists.

\begin{conjecture}[semiampleness]\label{conj-semiample}\cite[2.6]{3fold-logmod} \cite[1.16]{isksh}
Let $(X/Z,B)$ be a wlc model. Then $K + B$ is semiample.
\end{conjecture}

One of the major applications of geography:

\begin{corollary}\label{cor-Qsamp=Rsamp}
 In any fixed dimension $d$, the slt LMMP  in dimension $d$ and
the semiampleness with $\Q$-boundaries imply
the semiampleness with $\R$-boundaries.
Moreover, for the klt semiampleness with $\R$-boundaries,
the klt slt LMMP and the klt semiampleness with $\Q$-boundaries
are sufficient.
\end{corollary}

\begin{proof}
Immediate by Theorem \ref{mainthrm-geography}.
By definition, semiampleness for
vertexes of geography implies that of for all boundaries.
\end{proof}

\bs

\paragraph{Log canonical model \cite[1.17 (ii)]{isksh}}
Let $(X/Z,B)$ be a wlc model.
Then by Conjecture \ref{conj-semiample}
the $\R$-divisor $K+B$ defines a
{\em lc (Iitaka) contraction} $I\colon X\rightarrow X\cn/Z$
onto a normal projective variety $X\cn/Z$, called a
{\em lc (Iitaka) model } of $(X/Z,B)$.
In particular, if $I$ is birational,
the pair $(X\cn/Z,B\cn)$ with $B\cn=B\Lg_{X\cn}$
is the lc model of $(X/Z,B)$ (see
Weakly log canonical model above).

If both the LMMP and Conjecture~\ref{conj-semiample} hold
for a pair $(X/Z,B)$ and the pair has a wlc model $(Y/Z,B\Lg_Y)$, then
one can define a rational Iitaka contraction $I:X\dashrightarrow X\cn/Z$ as
a composition of the birational transformation $X\dashrightarrow Y$ and
of the Iitaka contraction of $Y\to Y\cn=X\cn$.
If $(X,B)$ is lc, the rational Iitaka contraction is a $1$-contraction.

By \cite[2.4.3-4]{3fold-logmod}, the image $X\cn$ and the rational map $X\dashrightarrow X\cn$ depend
only on the pair $(X/Z,B)$. Thus we can write $X\cn$ instead of  $Y\cn$.
Furthermore, it is known that the rational Iitaka contraction $I$ and its image $X\cn$ depend only
on the $\lcm$-class by definition(see Corollary \ref{cor-property-p-e}).
In particular, the contraction $I$ is independent of $B$ modulo $\wlc$.

\ms

\paragraph{Mobile decomposition}
The \emph{positive part} of $\sK+B\Lg$ is the b-Cartier divisor $P(B)=\overline{K_Y+B\Lg_Y}$
where $(Y/Z,B\Lg_Y)$ is a wlc model of $(X/Z,B)$.
By definition $P(B)$ is nef and $\R$-mobile by
the semiampleness \ref{conj-semiample}.
If $(X/Z,B)$ does not have a wlc model, that is,
$\nu(B)=-\infty$, put $P(B)=-\infty$.

We can define the {\em mobile decomposition\/}: $\sK+B\Lg=P(B)+F(B)$.
Here, $F(B)$ is called the \emph{fixed part} of $\sK+B\Lg$.
If $\nu(B)=-\infty$, put $F(B)=+\infty$.

\paragraph{Numerical Kodaira dimension}\cite[2.4.4]{3fold-logmod}
Given a pair $(X/Z,B)$ with  $B\in\fB_S$, we define the {\em relative numerical Kodaira dimension}
$\nu(X/Z,B)(=\nu(B))$ as $-\infty$ if the resulting model of $(X/Z,B)$ is a Mori log fibration and
as the integer
$$
\nu(B)=\max \{ m\in\Z \mid (K_Y+B\Lg_Y)^m\not\equiv 0 \text{ generically over }Z\}
$$
if $(X/Z,B)$ has a wlc model $(Y/Z,B\Lg_Y)$.

By definition, the following are equivalent:
(i) $\nu(X/Z,B)\geq 0$,
(ii) there exists a wlc model of $(X/Z,B)$,
(iii) $B\in \fN_S$ (see Section \ref{section2-geography} above),
(iv) $\kappa(X/Z,B)\geq 0$ (if the semiampleness holds \ref{conj-semiample}).
If this is the case, the integer $\nu(B)$ is independent of the choice of the
wlc model $Y$ \cite[2.4.4]{3fold-logmod}.

\paragraph{Kodaira dimension}\cite[Definition 2.2.3]{choi}
Let $B$ be an $\R$-boundary on $X/Z$ and $Y/Z$ be a model of $X/Z$
such that the log pair $(Y/Z,B\Lg_Y)$ is lc. Then the {\em relative invariant log
Kodaira dimension} $\kappa(X/Z,B)$ of the pair $(X/Z,B)$ is defined as the integer:
$$\kappa(X/Z,B):= \id{Y/Z,K_Y + B\Lg_Y},$$
where $\id{X/Z,D}$ denotes the relative Iitaka dimension of $D$.
For a $\Q$-divisor $D$,
this is the same as the usual definition, but
for an $\R$-divisor $D$, see \cite[Definition 2.2.3]{choi}.

\paragraph{Characteristic $0$}
Most of the results in this paper, if not all,
can be established under the following assumptions:
existence of log resolutions, the LMMP and the semiampleness.
Of course, by the subsection LMMP above we can include log resolutions
into the LMMP.
Indeed, to construct an (slt) initial model of $(X/Z,B)$,
one can use a log resolution (see Initial model above).
However, usually it is not useful especially in characteristic $0$,
because in that case a log resolution always exists.
Moreover, in characteristic $0$ some parts of the LMMP are established or announced:
the cone theorem, the contraction of an extremal ray and
existence of some special flips (e.g., slt).
According to the proofs, these
results require projectivity of $Y/Z$.
The main remaining problem is termination.
A {\em weak\/} termination, that is, some sequence of flips terminates,
is sufficient for many results in birational geometry, in particular,
for this paper (cf. Assumptions for our results below).
Thus we suppose the weak termination in the LMMP.
Note also, that in characteristic $0$, the semiampleness holds
for any nef divisor of relative FT varieties.

The LMMP with usual termination is established only in dimension $\le 3$ \cite{3fold-logmod} and
the slt LMMP with weak termination in dimension $4$ \cite{orderedterm}.
For the klt big LMMP with weak termination in all dimensions see in \cite{bchm}.
The {\em ``big"\/} means that $B$ is big for an initial model,
or equivalently, the same holds for each initial model.

\bs

\paragraph{Assumptions for our results}
Since the LMMP is still not established even
in characteristic $0$ and some of our results can be
use in its proof, we prefer to put explicitly what
is needed in each statement even some of assumptions
is already proved.
In the low dimension applications, we do not do this.
So, in addition to the assumptions in Characteristic $0$,
we need the following.
Note that the big Conjecture \ref{conj-semiample} means
that of for big $K+B$.
Needed assumptions are mentioned after colons and each assumption is sufficient
in the dimension of varieties of the statement and the dimension $\ge 4$.

\bs

\noi Proposition \ref{prop-qfact-proj}:
the slt LMMP and big Conjecture \ref{conj-semiample}.

\noi Proposition~\ref{prop-equivalence}: Conjecture \ref{conj-semiample}.

\noi Theorem \ref{mainthrm-geography}: the slt LMMP;
the klt (big) LMMP for any convex closed polyhedral subset of
klt (respectively big) boundaries in $\fB_S$
on some slt initial model.

\noi Proposition \ref{prop-p-e}: the slt LMMP; etc as Theorem \ref{mainthrm-geography}.

\noi Corollary \ref{cor-property-p-e}: the slt LMMP;
etc as Theorem \ref{mainthrm-geography};
Conjecture \ref{conj-semiample} in (2)
(not needed for the klt big case).

\noi Corollary \ref{cor-facet=face}: the slt LMMP; etc as
Theorem \ref{mainthrm-geography}.

\noi Corollary \ref{cor-eqn for facets}: the slt LMMP; etc as
Theorem \ref{mainthrm-geography}.

\ms

\noi Proposition \ref{prop-0log cones coincide in U}: the klt slt LMMP and
Conjecture~\ref{conj-semiample} for klt pairs with $\Q$-boundaries;
the klt slt big LMMP for any convex closed polyhedral subset of
klt big boundaries.

\noi Proposition \ref{cor-0log cones}: the klt slt LMMP and
Conjecture~\ref{conj-semiample} for klt pairs with $\Q$-boundaries;
etc as Proposition \ref{prop-0log cones coincide in U}.

\noi Corollary \ref{cor-0log num cones}: the klt slt big LMMP.
Comment: It looks that $\Char k=0$ is only needed for the coincidence
$\equiv$ and $\sim_\R$.
For arbitrary characteristic,
one can use Conjecture \ref{conj-semiample}.

\noi Corollary \ref{cor-positive cones}: the klt slt big LMMP.
Comment as in Corollary \ref{cor-0log num cones}.

\ms

\noi Corollary \ref{cor-0log eff decomposition}: the klt slt LMMP and Conjecture~\ref{conj-semiample}
for klt pairs with $\Q$-boundaries; etc as Proposition \ref{prop-0log cones coincide in U}.

\noi Corollary \ref{cor-0log num eff decomposition}: the klt  slt big LMMP.

\noi Corollary \ref{cor-FT eff decomposition}:  the klt  slt big LMMP.

\noi Corollary \ref{cor-0log finite contractions}:  the klt  slt big LMMP.

\noi Corollary \ref{cor-FT finite contractions}:  the klt  slt big LMMP.

\noi Theorem \ref{thrm-finite wlc big klt}:  the klt  slt big LMMP.

\noi Corollary \ref{cor-finite big minimal}: the klt  slt big LMMP.

\noi Corollary \ref{cor-DMMP on 0log}: the klt slt LMMP and Conjecture~\ref{conj-semiample}
for klt pairs with $\Q$-boundaries; etc as Proposition \ref{prop-0log cones coincide in U}.

\noi Corollary \ref{cor-num DMMP 0log}: the klt  slt big LMMP.

\noi Corollary \ref{cor-DMMP on FT}: the klt  slt big LMMP.

\noi Corollary \ref{cor-Qfact FT flops}: the klt  slt big LMMP.

\noi Corollary \ref{cor-slt flops}: the klt  slt big LMMP.

\noi Proposition \ref{prop-small iso mobile cones}: for FT varieties, the klt slt big LMMP.

\noi Proposition \ref{prop-FT mobile cones}:  the klt  slt big LMMP.

\ms

\noi Proposition \ref{prop-monotone N}*: the slt LMMP.

\noi Proposition \ref{prop-S=closed rational poly}*: the slt LMMP.

\noi Lemma \ref{lemma-nonempty separatrix}*: the slt LMMP.

\noi Proposition \ref{prop-image of separatrix}*: the slt LMMP.

\noi Proposition \ref{prop-kd nd are constants}*: the slt LMMP.

\noi Corollary \ref{cor-geoface=cube for general}*: the slt LMMP.

(* can be treated as Theorem \ref{mainthrm-geography} for
any convex closed polyhedral subset in $\fB_S$, respectively
any closed polyhedral subset in $\fS_S$.)

\noi Theorem \ref{thrm-unique wlc for general geo}: the slt LMMP;
for any closed convex polyhedral subset in $\fB_S$ with klt boundaries,
the klt slt big LMMP.

\noi Lemma~\ref{lemma-rho of wlc/lc}: the klt slt big LMMP.

\noi Theorem~\ref{thrm-facet types}: the slt LMMP; etc as Theorem \ref{thrm-unique wlc for general geo}.

\noi Corollary~\ref{cor-small klt slt wlc models in elementary log flops}: the klt slt big LMMP.

\noi Theorem \ref{thrm-ridge types}: the slt LMMP;
etc as Theorem \ref{thrm-unique wlc for general geo}.

\noi Corollary~\ref{cor-central model exists}: the klt slt big LMMP.

\noi Corollary \ref{cor-central models bounded}: the klt slt big LMMP.

\ms

\noi Theorem~\ref{thrm-mori morphism is cte}: the trm slt big LMMP.
Comment: trm slt implies klt in dimensions $\ge 2$.
We prefer slt because it includes the $\Q$-factorial and
projective properties.

\noi Corollary~\ref{cor-conic rigid threshold}: the trm slt big LMMP.

\noi Theorem \ref{thrm-nonrigid Mori}: the trm slt big LMMP.

\noi Corollary \ref{cor-rigid Mori}: the trm slt big LMMP.

\noi Corollary \ref{cor-rigid Mori Fano}: the trm slt big LMMP.

\bs
\bs

\end{document}